\documentclass[11pt]{article}

\newcommand{\insieme}[1]{\left\{ #1 \right\}}
\usepackage{fullpage}
\usepackage{subcaption}
\usepackage{mathrsfs}  
\usepackage{bm}

\usepackage{tikz}
\usetikzlibrary{patterns}
\pgfdeclarepatternformonly{my crosshatch dots}{\pgfqpoint{-1pt}{-1pt}}{\pgfqpoint{6pt}{6pt}}{\pgfqpoint{8pt}{8pt}}%
{
    \pgfpathcircle{\pgfqpoint{0pt}{0pt}}{.5pt}
    \pgfpathcircle{\pgfqpoint{4pt}{4pt}}{.5pt}
    \pgfusepath{fill}
}

\pgfdeclarepatternformonly{nel}{\pgfqpoint{-1pt}{-1pt}}{\pgfqpoint{4pt}{4pt}}{\pgfqpoint{4pt}{4pt}}%
{
  \pgfsetlinewidth{0.3pt}
  \pgfpathmoveto{\pgfqpoint{0pt}{0pt}}
  \pgfpathlineto{\pgfqpoint{4pt}{4pt}}
  \pgfusepath{stroke}
}

\usepackage{enumerate}
\usepackage{amsmath,amsfonts,amssymb,amsthm, bbold}
\usepackage{svg}
\usepackage{graphics,esint}
\usepackage{epsfig}
\usepackage{xcolor}
\usepackage{xspace}
\definecolor{pingreen}{rgb}{0,39,14}

\usepackage{xfrac}
\usepackage{comment}
\usepackage{hyperref}
\usepackage{cleveref}

\crefname{section}{§}{§§}
\Crefname{section}{§}{§§}

\def\Mcal{\mathcal{M}}
\def\Wcal{\mathcal{W}}

\def\Ical{\mathcal{I}}

\def\rhs{r.h.s.\xspace}
\def\wrt{w.r.t. \xspace}

\DeclareMathOperator*{\argmin}{arg\,min}

\makeatletter
\newtheorem*{rep@theorem}{\rep@title}
\newcommand{\newreptheorem}[2]{%
\newenvironment{rep#1}[1]{%
 \def\rep@title{#2~\ref{##1}}%
 \begin{rep@theorem}}%
 {\end{rep@theorem}}}
\makeatother

\newtheorem{theorem}{Theorem}[section]

\newtheorem{lemma}[theorem]{Lemma}

\newtheorem{remark}[theorem]{Remark}

\newtheorem{proposition}[theorem]{Proposition}

\newtheorem{corollary}[theorem]{Corollary}

\newreptheorem{theorem}{Theorem}

\newcommand{\R}{\mathbb{R}}
\newcommand{\Z}{\mathbb{Z}}

\newcommand{\Fcal}{\mathcal{F}}
\newcommand{\Fcalbar}{\overline{\mathcal{F}}}
\newcommand{\FcalJ}{\tilde{\mathcal{F}}_{J,L,\eps}}
\newcommand{\Fcalt}{\mathcal{F}_{\tau,L,\eps}}

\newcommand{\Fcalo}{\mathcal{F}^{1d}_{\tau,L,\eps}}

\newcommand{\Khattau}{\widehat{K}_{\tau}}

\def\N{\mathbb N}
\def\eps{\varepsilon}

\def\per{\mathrm{Per}}

\def\eps{\varepsilon}

\def\d {\,\mathrm {d}}
\def\dx{\,\mathrm {d}x}
\def\dz{\,\mathrm {d}z}
\def\ds{\,\mathrm {d}s}

\def\dt{\,\mathrm {d}t}
\def\dy{\,\mathrm {d}y}

\def\loc{\mathrm{loc}}
\def\at{\alpha_{\eps,\tau}}
\def\Mi#1#2#3{\overline{\mathcal M}^{#3}_{\at}(u,x_{#3}^\perp,[#1,#2))}
\def\Mii#1#2{\overline{\mathcal M}^{#2}_{\at}(u,x_{#2}^\perp,#1)}

\def\Gcal#1{\overline{\mathcal{G}}^{#1}_{\at,\tau}(u,x_{#1}^\perp,[0,L))}

\newcommand{\ie}{i.e.,\xspace}

\numberwithin{equation}{section}

\usepackage{authblk}

\author[1]{Sara Daneri\thanks{sara.daneri@gssi.it}}
\author[2]{Alicja Kerschbaum\thanks{kerschbaum@math.fau.de}}
\author[3]{Eris Runa\thanks{eris.runa@gmail.com}}
\affil[1]{Gran Sasso Science Institute, L'Aquila, Italy}
\affil[2]{Friedrich-Alexander-Universit\"at Erlangen-N\"urnberg}
\affil[3]{Deutsche Bank, London, UK}

\title{One-dimensionality of the minimizers for a diffuse interface generalized antiferromagnetic model in general dimension}

\date{}

\begin{document}
\maketitle

\begin{abstract}
   In this paper we study a diffuse interface generalized antiferromagnetic model. 
   The functional describing the model contains a Modica-Mortola type local term and a nonlocal generalized antiferromagnetic term in competition. 
   The competition between the two terms results in a frustrated system which is believed to lead to the emergence of a wide variety of patterns. 
   The sharp interface limit of our model is considered  in~\cite{GR} and in~\cite{DR}. In the discrete setting it has been previously studied in~\cite{GLL, GLS, GS}.
   The model contains two parameters: $\tau$ and $\varepsilon$. The parameter $\tau$ represents the relative strength of the local term with respect to the nonlocal one, while the parameter $\varepsilon$ describes the transition scale in the Modica-Mortola type term. 
   If $\tau < 0$ one has that the only minimizers of the functional are constant functions with values in $\{0,1\}$. 
   In any dimension $d\geq1$ for small but positive $\tau$ and  $\varepsilon$, it is conjectured that the minimizers are non-constant one-dimensional periodic functions. 
   In this paper we are able to prove such a characterization of the minimizers, thus showing also the symmetry breaking in any dimension~$d >1$. 
\end{abstract}

\section{Introduction}
\label{sec:intro}

In this paper we consider the following mean field free energy functional.
For $L,J,\eps>0$, $d\geq1$, $p\geq{d+2}$, $u\in W^{1,2}_{\loc}(\R^d;[0,1])$ and $[0,L)^d$-periodic, define

\begin{equation}\label{E:F}
   \FcalJ(u):=\frac{J}{L^d}\Bigl[3\eps\int_{[0,L)^d}\|\nabla u(x)\|_1^2\dx+\frac{3}{\eps}\int_{[0,L)^d}W(u(x))\dx\Bigr]-\frac{1}{L^d}\int_{\R^d}\int_{[0,L)^d}|u(x+\zeta)-u(x)|^2K(\zeta)\dx\d\zeta,
\end{equation}
where, for $y=(y_1,\dots,y_d)\in\R^d$, $\|y\|_1=\sum_{i=1}^d|y_i|$, $W(t)=t^2(1-t)^2$ and $K(\zeta)=\frac{1}{(\|\zeta\|_1+1)^p}$. 

This type of local/nonlocal interaction functionals, with suitable choices of the kernel $K$, is used to model pattern formation in several contexts, among which thin-magnetic films~\cite{Seul476}, diblock copolymer melts \cite{OK} and colloidal systems \cite{BBCH,CCA,IR,GCLW,DR2,DR3}. 
Periodic patterns in the ground states are expected to emerge by the competition between the first term, short-range and attractive, and the second term, long-range and repulsive. 
Depending on the mutual strength between the two terms, modulated in this case by the constant $J$, different patterns are expected to occur. 
While pattern formation is observed in experiments and simulations~\cite{Seul476,MR2338353, BBCH,CCA,IR,GCLW}, a rigorous proof of the emergence of such phenomenon is still in many cases an open problem, due among others to the fact that minimizers display, in dimension $d\geq2$, less symmetries than the functional itself.
In the literature this phenomenon is called symmetry breaking.

Let 
\begin{equation}
   \label{eq:jc}
   J_c:=\int_{\R^d}|\zeta_1|K(\zeta)\d\zeta.
\end{equation}
One can show (see Lemma~\ref{lemma:jc}), that if $J\geq J_c$ then the minimizers of~\eqref{E:F} are the constant functions $u\equiv 0$ and $u\equiv 1$. 
We are interested in the structure of minimizers for $J\in[J_c-\tau,J_c)$ where $0<\tau\ll1$ and $0<\eps\ll1$. 
In analogy to what happens for the sharp interface limit of this problem (namely as $\eps\to0$), which was studied in~\cite{GR,DR} (and previously in the discrete in~\cite{GLL, GLS, GS}), it is conjectured that, for $\eps$ and $\tau$ sufficiently small, minimizers of~\eqref{E:F} are periodic one-dimensional functions, 
namely there exist  $g:\R\to \R$ and $h>0$ such that 
\begin{itemize}
   \item  the minimizers are functions of the form $u(x)=g(x_i)$ for some $i\in\{1,\dots d\}$ (one-dimensionality)
   \item for all $x_i\in\R$, $g(x_i+2h)=g(x_i)$ (periodicity)
   \item and there exists a translation parameter $\nu\in\R$ such that the following reflection property holds
\begin{equation}\label{eq:refl}
g(\nu+(2k+1)h+t)=1-g(\nu+(2k+1)h-t)\quad\text{ for all }k\in\N\cup\{0\},\,\, t\in[0,h].
\end{equation}

\end{itemize}

In this paper, we are able to prove the above conjecture on the one-dimensionality of minimizers for $\eps$ and $\tau$ small but positive, in general dimension.

In order to state our results properly, it is convenient to rescale the functional in order to have that the width of the admissible optimal periods for one-dimensional functions and their energy are of order $O(1)$.

For $\beta=p-d-1$, setting 
\begin{align*}
   &J=J_c-\tau=\int |\zeta_1|K(\zeta)\d\zeta-\tau,\quad x=\tau^{-1/\beta}\tilde x,\quad   \zeta=\tau^{-1/\beta}\tilde \zeta,\quad L=\tau^{-1/\beta}\tilde L,\\
   & \tilde u(\tilde x)=u(x),\quad\FcalJ( u)=\tau^{1+1/\beta}\mathcal{F}_{\tau,\tilde L,\eps}(\tilde u)
\end{align*}
and finally dropping the tildas, one has that the rescaled functional has the form
\begin{equation}
   \label{def:fcalt}
   \Fcalt(u)=\frac{1}{L^d}\Bigl[\mathcal M_{\alpha_{\eps,\tau}}(u,[0,L)^d)\Bigl(\int_{\R^d}K_\tau(\zeta)|\zeta_1|\d\zeta-1\Bigr)-\int_{\R^d}\int_{[0,L)^d}|u(x)-u(x+\zeta)|^2K_\tau(\zeta)\dx\d\zeta\Bigr],
\end{equation}
where for $\alpha>0$
\begin{equation}
   \label{eq:malpha}
   \mathcal M_{\alpha}(u,[0,L)^d)=3\alpha\int_{[0,L)^d}\|\nabla u(x)\|_1^2\dx+\frac{3}{\alpha}\int_{[0,L)^d}W(u(x))\dx,
\end{equation}
$\alpha_{\eps,\tau}=\eps\tau^{1/\beta}$ and 
\begin{equation}
   \label{def:Ktau}
   K_\tau(\zeta)=\frac{1}{(\|\zeta\|_1+\tau^{1/\beta})^p}.
\end{equation}

 For fixed $\tau > 0$ and $\eps>0$, consider first for all $L > 0$ the minimal value obtained by $\Fcal_{\tau,L,\eps}$ on $[0,L)^d$-periodic one-dimensional functions (denoted by $\mathcal U^{per}_L$)
and then the minimal among these values as $L$ varies in $(0,+\infty)$. We will denote this value by $C^*_{\tau,\eps}$, namely

\begin{equation}\label{eq:cstartau}
\begin{split}
C^*_{\tau,\eps} := \inf_{L>0} \  \inf_{u\in \mathcal U^{per}_L} \Fcal_{\tau,L,\eps}(u).  
\end{split}
\end{equation}

By the reflection positivity technique, in \cite{GLL1D} it is shown that such value is attained  by periodic one-dimensional functions with possibly infinite and not unique periods.

In Section~\ref{sec:1d} we prove that,  for $\tau$ and $\eps$ sufficiently small, there exist periodic functions of finite period $2h$ for which the property \eqref{eq:refl} holds and the energy value $C^*_{\tau,\eps}$ is attained.  We denote any of such finite optimal periods $2h$ (which may not be unique) as $2h^*_{\tau,\eps}$.

Our main result is the following

\begin{theorem}
\label{Thm:1}
   Let $L=2kh^*_{\tau,\eps}$, $k\in\N$. Then there exist ${\tau}_L>0$, $\eps_L>0$ such that, for any $0<\tau\leq{\tau}_L$ and $0<\eps\leq \eps_L$ the minimizers of~\eqref{def:fcalt} are one-dimensional periodic functions of period $2h^*_{\tau,\eps}$.
\end{theorem}

\begin{remark}
	
The fact that $L$ is a multiple of one of the optimal periods $2h^*_{\tau,\eps}$  is due to the fact that the periodicity of minimizers of $\Fcal_{\tau,L, \varepsilon}$ among one-dimensional functions as proved in \cite{GLL1D} is known a priory only when $L$ is a multiple of $2h^*_{\tau,\eps}$.
In particular, if the periodicity of one-dimensional minimizers would hold for arbitrary $L$ as in the sharp interface limit of \eqref{def:fcalt} as $\eps\to0$, then our result would give one-dimensionality and periodicity of $[0,L)^d$-periodic minimizers for the functional \eqref{def:fcalt} in general dimension.
\end{remark}

 Moreover, it is not difficult to see that the results contained in this paper can be used to prove analogous results for the diffuse interface version of the model for colloidal systems considered in~\cite{DR2}.

\subsection{Scientific context}

For the sharp interface limit of $\Fcalt$ as $\eps\to0$, namely the functional  

\begin{equation}\label{E:S}
   \mathcal{F}_{\tau,L}(E):=\frac{1}{L^d}\Bigl[\per_1(E;[0,L)^d)\Bigl(\int_{\R^d}K_\tau(\zeta)|\zeta_1|\d\zeta-1\Bigr)-\int_{\R^d}\int_{[0,L)^d}|\chi_E(x)-\chi_E(x+\zeta)|K_\tau(\zeta)\dx\d\zeta\Bigr],
\end{equation}
 and for $d\geq2$, the fact that for $\tau$ sufficiently small minimizers are periodic unions of stripes of width $h_{\tau,L}$ has been shown in the discrete setting in~\cite{GS} and for  the continuous setting in~\cite{DR}. {In \cite{Ker} the results in \cite{DR} have been recently extended to a small range of exponents below $p=d+2$.}
 
 In particular, one has that $|h_{\tau,L}-h^\ast_{\tau}|\leq C/L$ where $h^\ast_{\tau}$ is the unique admissible width of stripes $S$ attaining the value
 
 \begin{equation}\label{eq:cstartau1}
 \begin{split}
 C^*_{\tau} := \inf_{L>0} \  \inf_{S \text{ per. stripes}} \Fcal_{\tau,L}(S).  
 \end{split}
 \end{equation}

A periodic union of stripes of width $h$ is by definition a set which, up to Lebesgue null sets, is of the form $V_i^\perp+\widehat E e_i$ for some $i\in\{1,\dots,d\}$, where $V_i^\perp$ is the $(d-1)$-dimensional subspace orthogonal to $e_i$ and $\widehat E\subset\R$ with $\widehat E=\bigcup_{k=0}^N(2kh+\nu,(2k+1)h+\nu)$ for some $\nu\in\R$ and some $N\in\N$.

\vspace{1mm}
Some of the most physically relevant exponents $p$ in the literature are $p=d+1$ (thin magnetic films), $p=d-2$ (diblock copolymer) and $p=d$ (3D micromagnetics).
To our knowledge, there are no results where pattern formation for such models is shown if $d\geq2$ and the domain is symmetric under permutation of coordinates. 
This is the most challenging setting to consider due to the phenomenon of symmetry breaking.  
For $p=d-2$ in two-dimensional thin domains one-dimensionality of minimizers is shown in~\cite{MS}, while in~\cite{PV} the authors show one-dimensionality in a suitable asymptotic limit.   
Another very important family of kernels which is physically relevant and widely used in the literature is the Yukawa or screened Coulomb kernel (commonly used to model pattern formation in colloidal suspensions and protein solutions).
In a recent paper~\cite{DR2} the authors show that in a certain regime global minimizers of the corresponding functionals are periodic unions of stripes.

As for the structure of minimizers of diffuse interface functionals of the type~\eqref{E:F}, the best results which have been obtained in the literature so far are the following.
In a low density regime and for the Ohta-Kawasaki kernel, properties of the shape of droplets of minimizers for $\eps\ll1$ and $d=2$ were deduced from the analysis of the sharp interface limit in~\cite{GMS} and~\cite{GMS2}, while results on the minimizers of~\eqref{E:F} for $d=1$ and more general reflection positive kernels were proved in~\cite{GLL1D}.

Evolution problems of gradient flow type related to functionals with attractive-repulsive nonlocal terms in competition, both in presence and in absence of diffusion, are also well studied (see e.g. \cite{CCH,CCP, CDFS,DRR,craig, CT}). In particular, one would like to show convergence of the gradient flows or of their deterministic particle  approximations to configurations which are periodic or close to periodic states.

Another interesting direction would be to extend our rigidity results to non-flat surfaces without interpenetration of matter as investigated for rod and plate theories  in \cite{KS, LMP, OR}.

   In this paper we are able to show one-dimensionality and periodicity of minimizers of~\eqref{E:F} for $\eps$ and $\tau$ sufficiently small (see~Theorem~\ref{Thm:1}).

Most of the lower bounds and the estimates that we find for penalizing deviations from the set of one-dimensional functions are obtained directly for the diffuse-interface functional~\eqref{E:F}, independently on its limit behaviour as $\eps\to0$.

\subsection{Some ideas of the proof}
   Let us now describe the main ideas of the proof of Theorem~\ref{Thm:1}. 
   For simplicity, we will assume that $d=2$.
   Very roughly speaking, we will find a lower bound (which is easier to work with) such that on one-dimensional functions $u$ both the original functional and the lower bound coincide and such that the lower bound is minimized on non-constant one-dimensional functions. 

   Let us now be more precise.
   Given a one-dimensional $u(x_1,x_2) = u_0(x_1)$ (resp. $u(x_1,x_2) = u_0(x_2)$) for some $u_{0}:\R\to [0,1]$, let us define
\begin{equation*}
   \begin{split}
      \Fcalo(u_0) := \mathcal F_{\tau,L, \varepsilon}(u).
   \end{split}
\end{equation*}
Notice that similarly to~\cite{DR} the functional $\Fcalo$ attains a negative value on its minimizers and thus also  $\Fcal_{\tau,L,\varepsilon}$ attains a negative value on optimal one-dimensional functions $u$. 

\begin{enumerate}
   \item [Step 1.]
     We will bound the original functional from below as follows
      \begin{equation}
         \label{eq:decintro}
         \begin{split}
            \Fcal_{\tau,L, \varepsilon}(u) \geq \Fcalbar^1_{\tau,L, \varepsilon}(u)  +  \Fcalbar^2_{\tau,L, \varepsilon}(u) + \mathcal I_{\tau,L}(u) + \Wcal_{\tau,L,\eps}(u),
         \end{split}
      \end{equation}
      where 
      \begin{itemize}
         \item The functional $\Fcalbar^i_{\tau,L, \varepsilon}$ accounts for the energy contribution in direction $e_i$.  Moreover, suppose that 
            \begin{equation*}
               \begin{split}
                  u(x_1,x_2) =u_0(x_1) \qquad\text{(resp.  $u(x_1,x_2) =u_0(x_2)$)}.
               \end{split}
            \end{equation*}
            Then 
            \begin{equation*}
               \begin{split}
                  \Fcalbar^2_{\tau,L,\varepsilon}(u) = 0 \qquad \text{(resp.  $\Fcalbar^1_{\tau,L,\varepsilon}(u) = 0 )$.}
               \end{split}
            \end{equation*}
         \item The cross interaction term $\mathcal I_{\tau,L}$  penalizes functions $u$ which are not one-dimensional. 

         \item The term $\Wcal_{\tau,L,\varepsilon}(u)$ is a correction term in the sense that, if  $u(x_1,x_2) = u_0(x_1)$ (resp.  $u(x_1,x_2) = u_0(x_2)$), then
            \begin{equation}
               \label{eq:ideaDefFcal1D}
               \begin{split}
                  \Fcalbar^1_{\tau,L,\varepsilon}(u) + \Wcal_{\tau,L,\varepsilon}( u ) = \Fcalo(u)
                  \qquad(\text{resp. }\Fcalbar^2_{\tau,L,\varepsilon}(u) + \Wcal_{\tau,L,\varepsilon}( u ) = \Fcalo(u)).
               \end{split}
            \end{equation}

      \end{itemize}


  \item [Step 2.] Using a $\Gamma$-convergence argument,  we reduce ourselves (up to taking $\tau,\varepsilon$ sufficiently small) to the situation where the minimizers are  $L^1$-close to the minimizers of the limit functional~\eqref{E:S}, namely to periodic unions of stripes. 
    Thus without loss of generality, let us assume that $u$ is close to the optimal union of stripes whose boundary is orthogonal to $e_1$. 

  \item [Step 3.] 
     We will then show (see~Proposition~\ref{prop:stability}), that if $u$ is sufficiently close to optimal periodic union of stripes with boundaries orthogonal to  $e_1$, then 
     \begin{equation}
        \label{eq:idea1}
        \begin{split}
           \Fcalbar^2_{\tau,L,\varepsilon}(u) + \Ical_{\tau,L}(u) \geq 0,
        \end{split}
     \end{equation}
     where in the above equality is achieved if and only if there exists $u_0$ such that $u(x_1,x_2)=  u_0(x_1)$. 
     Thus, we have that
     \begin{equation}
        \label{eq:step3}
        \begin{split}
           \Fcal_{\tau,L,\varepsilon}(u) \geq \Fcalbar^{1}_{\tau,L,\varepsilon}(u) + \Wcal_{\tau,L,\varepsilon}(u).
        \end{split}
     \end{equation}

      Such inequality is obtained through slicing, one-dimensional estimates and blow-up of the cross interaction term for deviations from one-dimensional profiles. 

   \item[Step 4.] For any $x_2\in[0,L)$ we notice that the slice in direction $e_1$ of the functional $\Fcalbar^{1}_{\tau,L,\varepsilon} + \Wcal_{\tau,L,\varepsilon}$ passing through $x_2$ can be rewritten in the form $F(\gamma,u(\cdot,x_2))$, where $\gamma\geq1$ and $\gamma=1$ a.e. if and only if $u(x_1,x_2)=u_0(x_1)$. By reflection positivity as in \cite{GLL1D}, we observe that when $L=2kh^*_{\tau,\eps}$ the one-dimensional functional $(\gamma,g)\mapsto F(\gamma,g)$ is minimized by functions $g$ satisfying \eqref{eq:refl} with $h=h^*_{\tau,\eps}$ and $\gamma$ satisfying $\gamma(\nu+x+kh^*_{\tau,\eps})=\gamma(\nu+(k+1)h^*_{\tau,\eps}-x)$. Then, with a delicate analysis of the Euler-Lagrange equations associated to $F(\gamma,g)$ we prove that in the above class such a functional is minimized by $\gamma=1$ a.e.. Thus there are minimizers of  $\Fcalbar^{1}_{\tau,L,\varepsilon} + \Wcal_{\tau,L,\varepsilon}$ of the form $u(x_1,x_2)=u_0(x_1)$.  
\end{enumerate}

Let us now discuss some main differences compared to~\cite{DR}.

\begin{enumerate}[(i).]
   \item In Step 1 it is fundamental that if $u(x,y) = u_0(x)$, then $\Fcalbar^2_{\tau,L,\varepsilon}(u)=0$ (and analogously $\Fcalbar^1_{\tau,L, \varepsilon}(u)=0$ if $u(x,y)=u_0(y)$). 
      The construction in~\cite{DR} deeply relies on the fact that the analogues of the  functionals $\Fcalbar^i_{\tau,L,\varepsilon}$ depend only on slices in direction~$e_i$. 
      Such construction cannot be mimicked when the $1$-perimeter is replaced with the Modica-Mortola term. 
      Thus a new decomposition is needed.
   \item \label{introii} Another crucial part in~\cite{DR} is the one-dimensional optimization. 
      Namely, once shown that the sum of the second and the third term in the \rhs of~\eqref{eq:decintro} (due to~\eqref{eq:idea1}) is positive, the remaining terms are minimized on optimal periodic stripes.
         In order to do so the authors in \cite{DR} use that the remaining terms depend only on the slices in direction $e_i$. More precisely in~\cite{DR} the \rhs of~\eqref{eq:step3} can be written as
         \begin{equation*}
            \begin{split}
               \frac{1}{L^{d-1}}\int_{[0,L)^{d-1}}  \Fcal^{1d}_{\tau,L}(u_{x^\perp_i}) \dx^\perp_i, 
            \end{split}
         \end{equation*}
         thus in order to minimize the remaining terms one needs to minimize the one-dimensional  problem which is well studied. 
         This is not true anymore for our decomposition, namely the \rhs of~\eqref{eq:step3} cannot be written as above since it depends on $\nabla u$. Thus in principle a multidimensional optimization is needed.
         We show that even in this setting one-dimensional functions are optimal (see Section~\ref{subsec:micoef}). In doing this we need to assume that $L$ is a multiple of an optimal period $2h^*_{\tau,\eps}$ in order to have that minimizers among one-dimensional functions satisfy \eqref{eq:refl} for $h=h^*_{\tau,\eps}$.

   \item In~\cite{GR,DR}, the cross interaction term $\mathcal I_{\tau,L}$ is clearly positive. 
      In this paper a careful inspection is needed to prove positivity (see Lemma~\ref{lemma:positivity}). 

   \item One other crucial difference is the possibility of appearance of oscillations which are small in amplitude. 
      In~\cite{DR}, being the functions valued in $\{ 0,1\}$, this issue is not present, and many of the arguments in~\cite{DR} use the fact that the amplitude of the oscillations is always $1$. This issue is not trivial, indeed one could for example devise non-physical potentials in the Modica-Mortola term for which, when close to $0$ or $1$, oscillating at small amplitude is more convenient than being flat.
      Thus minimizers would not be one-dimensional.
      In order to deal with this issue new estimates are needed. 

   \item Moreover, transitions from values close to $0$ to values close to $1$, which in~\cite{DR} are instantaneous, in this case could happen on ``large'' intervals. 
      Our estimates lead to the following structure for slices of minimizers in direction $e_i$: either constant functions or functions which have transitions from values close to $0$ and values close to $1$ in a finite number of small intervals, each surrounded by sufficiently large intervals where functions stay close to either $0$ or $1$. Such a picture, which resembles in some sense that of the slices of minimizers for the sharp interface problem, and which cannot be obtained by simple $\Gamma$-convergence arguments, allows us to show the blow-up of the cross interaction term $\mathcal I_{\tau,L}$ when close to stripes with boundaries orthogonal to $e_i$ and having oscillations in directions $e_j\neq e_i$.

   \item{ In Section~\ref{sec:1d} we prove that the one-dimensional minimizers on which the value $C^*_{\tau,\eps}$ is attained are periodic of finite (possibly non-unique) period  $2h^*_{\tau,\eps}$ satisfying \eqref{eq:refl} for $h=h^*_{\tau,\eps}$ for $\tau$ and $\eps$ sufficiently small.}
\end{enumerate}

\subsection{Structure of the paper}
In Section~\ref{sec:notation} we recall the main notation and the results obtained for the sharp interface problem~\eqref{E:S} in~\cite{DR}.

In Section~\ref{sec:dec} we introduce the main decomposition of the functional~\eqref{E:F}.

In Section~\ref{sec:1dest} we give some crucial one-dimensional estimates.

In Section~\ref{sec:stability} we prove the main stability estimate.

In Section~\ref{sec:1d} we consider the associated one-dimensional problem and, starting from the results on general diffuse interface functionals obtained in~\cite{GLL1D} we prove existence of a finite (possibly non-unique) optimal period $2h^*_{\tau,\eps}$ and optimal functions $g$ satisfying \eqref{eq:refl} for $h=h^*_{\tau,\eps}$. Moreover, in Theorem~\ref{thm:onedimmin} we prove a crucial  optimization result needed to show one-dimensionality of minimizers (see point (ii) above).

In Section~\ref{sec:proof1} we prove Theorem~\ref{Thm:1}.


\section{Notation and preliminary results}
\label{sec:notation}

In the following, let $\N=\{1,2,\dots\}$, $d\geq 1$. 
Let $(e_1,\dots,e_d)$ be the canonical basis in $\R^d$ and for $y\in\R^d$ let $y_i=\langle y,e_i\rangle$ and $y_i^\perp:=y-y_ie_i$, where $\langle\cdot,\cdot\rangle$ is the Euclidean scalar product. 
For $y\in\R^d$, we denote by $\|y\|_1=\sum_{i=1}^d|y_i|$ its $1$-norm and we define $\|y\|_\infty=\max_i|y_i|$.
With a slight abuse of notation, we will sometimes identify $y^\perp_i\in[0,L)^d$ with its projection on the subspace orthogonal to $e_i$ or as an element of $\R^{d-1}$. 


For $z\in[0,L)^d$ and $r>0$, we also define 
\begin{equation*}
   \begin{split}
      Q_r(z)=\{x\in\R^d:\,\|x-z\|_\infty\leq r\} \qquad\text{and}\qquad Q_{r}^{\perp}(x^\perp_{i}) = \{z^\perp_{i}:\, \|x^{\perp}_{i} - z^{\perp}_{i} \|_\infty \leq r  \}. 
   \end{split}
\end{equation*}


For every $i\in\{1,\dots,d\}$ and for all $x_i^\perp\in[0,L)^{d-1}$, we define the slices of $u$ in direction $e_i$ as
\[
   u_{x_i^\perp}:\R\to[0,1],\quad u_{x_i^\perp}(s):=u(s e_i+x_i^\perp).
\]

Notice that whenever $u\in W^{1,2}_{\loc}(\R^d;\R)$ then $u_{x_i^\perp}\in  W^{1,2}_{\loc}(\R; \R)$ for almost every $x^\perp_i$. 
We denote by $\partial_i$ the partial derivatives of a function with respect to $e_i$, $i\in\{1,\dots,d\}$.

Given a measurable set $A\subset\R^k$ with $k\in\{1,\dots,d\}$, we denote by $|A|$ its $k$-dimensional Lebesgue measure (or if A is contained in some $k$-dimensional plane of $\R^d$, its Hausdorff $k$-dimensional measure), being always clear from the context which will be the dimension $k$.

Moreover, let $\chi_A:\R^d\to\R$ be the function defined by
\begin{equation*}
   \chi_A(x)=\left\{\begin{aligned}
         &1 && &\text{if $x\in A$}\\
         &0 && &\text{if $x\in\R^d\setminus A$.}
      \end{aligned}\right.
\end{equation*}

A set $E\subset\R^d$ is of (locally) finite perimeter if the distributional derivative of $\chi_E$ is a (locally) finite measure.
We denote by $\partial E$ be the reduced boundary of $E$ and by $\nu^E$ the exterior normal to $E$.

Then one can define the $1$-perimeter of a set relative to $[0,L)^d$ as

\[
   \per_1(E,[0,L)^d):=\int_{\partial E\cap [0,L)^d}\|\nu^E(x)\|_1\d\mathcal H^{d-1}(x)
\]
where $\mathcal H^{d-1}$ is the $(d-1)$-dimensional Hausdorff measure.

By extending the classical Modica-Mortola result~\cite{MM} to the anisotropic norm $\|\cdot\|_1$, one has the following
\begin{theorem}
   \label{Thm:gammaconvper}
   As $\alpha\to 0$, the functionals $\mathcal M_\alpha(\cdot;[0,L)^d)$ defined in~\eqref{eq:malpha} $\Gamma$-converge in $BV([0,L)^d;[0,1])$ to the functional $\mathcal P_1(\cdot;[0,L)^d)$ defined as follows:
   \begin{equation}
     \label{eq:p}
      \mathcal P_1(u;[0,L)^d):=\left\{\begin{aligned}
            &\per_1(E;[0,L)^d) && &\text{if $u=\chi_E$}\\
            &+\infty && &\text{otherwise.} 
         \end{aligned}\right.
   \end{equation}
\end{theorem}

Notice that the constant $3$ in~\eqref{E:F} is chosen in such a way that
\[
   6\int_0^1t(1-t)\dt=1,
\]
so that the constant in front of the $1$-perimeter in~\eqref{eq:p} is equal to $1$.

By continuity of the nonlocal term in~\eqref{E:F} with respect to $L^1$ convergence of functions valued in $[0,1]$, one has the following

\begin{corollary}
   \label{cor:gammaconv}
   As $\eps\to0$, the functionals $\Fcalt$ $\Gamma$-converge in $BV_{\loc}(\R^d;[0,1])$ to the functional
   \begin{equation}\label{eq:flim}
      \mathcal{F}_{\tau,L}(u):=\left\{
         \begin{aligned}
            &\frac{1}{L^d}\Bigl[\per_1(E;[0,L)^d)\Bigl(\int_{\R^d}K_\tau(\zeta)|\zeta_1|\d\zeta-1\Bigr)\\
            &\quad-\int_{\R^d}\int_{[0,L)^d}|\chi_E(x)-\chi_E(x+\zeta)|K_\tau(\zeta)\dx\d\zeta\Bigr] && &\text{if $u=\chi_E$}\\
            &+\infty && &\text{otherwise.}
         \end{aligned}
      \right.
   \end{equation}
\end{corollary}

The kernel $K_\tau$ is, as shown in~\cite{DR}, reflection positive, namely it satisfies the following property: the function
\begin{equation*}
   \begin{split}
      \widehat K_\tau(t) := \int_{\R^{d-1}} K_\tau(t, \zeta_2,\ldots,\zeta_d)  \d\zeta_2\cdots\d\zeta_d. 
   \end{split}
\end{equation*}

is the Laplace transform of a nonnegative function.

Regarding the limit functional~\eqref{eq:flim}, we recall the following results, obtained in~\cite{DR}.

\begin{theorem}[{\cite[Theorem~1.2]{DR}}]
   \label{Thm:DR}
   Let $d\geq1$, $p\geq d+2$, $L>0$. Then, there exists $\tilde{\tau}_L>0$ such that, for all $0<\tau\leq\tilde{\tau}_L$ the minimizers of the functional $\mathcal{F}_{\tau,L}$ in~\eqref{eq:flim} are periodic unions of stripes.
   The admissible width (which may not be unique) of these stripes is denoted by $h_{\tau,L}$.
\end{theorem} 

Moreover, for fixed $\tau > 0$, consider first for all $L > 0$ the minimal value obtained by $\mathcal F_{\tau,L}$ on $[0,L)^d$-periodic stripes 
and then the minimal among these values as $L$ varies in $(0,+\infty)$.
By the reflection positivity technique, this value is attained on periodic stripes.
Let $h^\ast_{\tau}$ be any admissible value for the width of such optimal stripes. 

In~\cite{DR} the following theorems have been proved:

\begin{theorem}[{\cite[Theorem~1.1]{DR}}]
   \label{Thm:DRunique}
   Let $d\geq1$, $p\geq d+2$. Then there exists $\check{\tau}>0$ s.t. whenever $0< \tau < \check{\tau}$, $h^*_\tau$ is unique. 
\end{theorem} 

\begin{theorem}[{\cite[Theorem~1.3]{DR}}]
   \label{Thm:DRclose}
   There exist $\tilde{\tau}>0$ with $\tilde{\tau} \leq \min\{\tilde{\tau}_L, \check \tau\}$ and a constant $C$ such that for every $0<\tau\leq\tilde{\tau}$, one has that any admissible  width $h_{\tau,L}$ of minimizers of $\mathcal F_{\tau,L}$ satisfies
   \begin{equation*}
      \label{eq:DRclose}
      |h_\tau^*-h_{\tau,L}|\leq \frac CL.
   \end{equation*}

\end{theorem} 

\begin{theorem}[{\cite[Theorem~1.4]{DR}}]
   \label{T:1.3}
   Let $d\geq1$, $p\geq d+2$ and $h^{*}_{\tau}$ be the optimal stripes' width for fixed $\tau$ sufficiently small.
   Then there exists $\tilde\tau_{0}$, such that for every $\tau< \tilde\tau_{0}$, one has that for every $k\in \N$ and  $L = 2k h_{\tau}^{*}$, the minimizers $E_{\tau}$ of $\mathcal F_{\tau,L}$ are optimal stripes of width $h_{\tau}^{*}$.
\end{theorem}

\section{Decomposition of the functional}
\label{sec:dec}

The main goal of this section is to prove the following proposition
\begin{proposition}
  \label{prop:lowbound}
 The following lower bound for the functional $\Fcalt$ holds
\begin{align}
\label{eq:3.13}
\Fcalt(u)\geq&\frac{1}{L^d} \sum_{i=1}^d\Bigl\{\int_{[0,L)^{d-1}}\Bigl[-\Mi{0}{L}{i}+\Gcal{i}\Bigr]\dx_i^{\perp}+\mathcal I^i_{\tau,L}(u)\Bigr\}\notag\\
&+\frac{1}{L^d}\mathcal W_{\tau,L,\eps}(u),\qquad
\end{align}
where 
\begin{align}
\Mi{s}{t}{i}&:=3\at\int_{[s,t]\cap\{\nabla u(\cdot e_i+x_i^\perp)\neq0\}}|\partial_iu_{x_i^\perp}(\rho)|\|\nabla u(\rho e_i+x_i^\perp)\|_1\d\rho\notag\\
&+\frac{3}{\at}\int_{[s,t]\cap\{\nabla u(\cdot e_i+x_i^\perp)\neq0\}}W(u_{x_i^\perp}(\rho)))\frac{|\partial_iu_{x_i^\perp}(\rho)|}{\|\nabla u(\rho e_i+x_i^\perp)\|_1}\d\rho,
\label{eq:3.2}
\end{align}
\begin{equation}
\label{def:Gcal}
\begin{split}
\Gcal{i}=&\Mi{0}{L}{i}\int_{\R}|\zeta_i|\widehat K_\tau(\zeta_i)\d\zeta_i - 
\\
&-\int_{\R}\int_0^L|u_{x_i^\perp}(x_i)-u_{x_i^\perp}(x_i+\zeta_i e_i)|^2\widehat K_\tau(\zeta_i)\dx_i\d\zeta_i,
\end{split}
\end{equation}

\begin{equation}
\label{eq:itl}
\mathcal I^i_{\tau,L}(u)=\frac1d\int_{\{\zeta_i>0\}}\int_{[0,L)^d}[(u(x+\zeta_i e_i)-u(x))-(u(x+\zeta)-u(x+\zeta_i^\perp))]^2K_\tau(\zeta)\dx\d\zeta
\end{equation}
and 
\begin{equation}
\label{def:Wcal}
\mathcal W_{\tau,L,\eps}(u)=\frac{3}{\at}\int_{\{\nabla u=0\}\cap[0,L)^d}W(u(x))\dx.
\end{equation}
Moreover, equality holds in~\eqref{eq:3.13} whenever the function $u$ is one-dimensional, namely whenever there exists $g\in W^{1,2}_{\mathrm{loc}}(\R;[0,1])$ $L$-periodic  and $i\in\{1,\dots,d\}$ such that $u(x)=g(x_i)$.
\end{proposition} 

In particular, since showing that the minimizers for the \rhs of~\eqref{eq:3.13} are one-dimensional implies that the minimizers for $\Fcal_{\tau,L,\varepsilon}$ are one-dimensional, this allows us to reduce ourselves to prove one-dimensionality of the minimizers for the lower bound functional.

We will need the following preliminary lemma.

\begin{lemma}
   \label{lemma:positivity}
   Let $u\in W^{1,2}_{\loc}(\R^d;[0,1])$ be a $[0,L)^d$-periodic function.
   Then, for all $j,j_1,\dots,j_k\in\{1,\dots,d\}$ with $j\neq j_1\neq\dots j_k$,
   \begin{align}
      -\int_{\R^d}&\int_{[0,L)^d}(u(x)-u(x+\zeta_j e_j))(u(x+\zeta_j e_j)-u(x+\zeta_je_j+\zeta_{j_1}e_{j_1}+\dots+\zeta_{j_k}e_{j_k}))K_\tau(\zeta)\dx\d\zeta=\notag\\
      &=\frac12\int_{\{\zeta_j>0\}}{\int_{[0,L)^d}}\Bigl[(u(x+\zeta_j e_j)-u(x))\notag\\
      &-(u(x+\zeta_je_j+\zeta_{j_1}e_{j_1}+\dots+\zeta_{j_k}e_{j_k})-u(x+\zeta_{j_1}e_{j_1}+\dots+\zeta_{j_k}e_{j_k}))\Bigr]^2K_\tau(\zeta)\dx\d\zeta. \label{eq:pos}
   \end{align}
\end{lemma}

\begin{proof}[Proof of Lemma~\ref{lemma:positivity}: ]
   One has that
  \begin{align}
         -&\int_{\R^d}\int_{[0,L)^d}(u(x)-u(x+\zeta_j e_j))(u(x+\zeta_j e_j)-u(x+\zeta_je_j+\zeta_{j_1}e_{j_1}+\dots+\zeta_{j_k}e_{j_k}))K_\tau(\zeta)\dx\d\zeta=\notag\\
         &=-\int_{\{\zeta_j>0\}\cup\{\zeta_j<0\}}\int_{[0,L)^d}(u(x+\zeta_j e_j)-u(x))(u(x+\zeta_je_j+\zeta_{j_1}e_{j_1}+\dots+\zeta_{j_k}e_{j_k})-u(x+\zeta_j e_j))K_\tau(\zeta)\dx\d\zeta\notag\\
         &=\int_{\{\zeta_j>0\}}\int_{[0,L)^d}(u(x+\zeta_j e_j)-u(x))\bigl[-u(x+\zeta_je_j+\zeta_{j_1}e_{j_1}+\dots+\zeta_{j_k}e_{j_k})+u(x+\zeta_j e_j)\notag\\
         &+u(x+\zeta_{j_1}e_{j_1}+\dots+\zeta_{j_k}e_{j_k})-u(x)\bigr]K_\tau(\zeta)\dx\d\zeta,\notag
         \end{align}
         where in the last equation we used the periodicity of $u$ when integrating on $\{\zeta_j<0\}$.

         Moreover,
         \begin{align}
         &\int_{\{\zeta_j>0\}}\int_{[0,L)^d}(u(x+\zeta_j e_j)-u(x))\bigl[-u(x+\zeta_je_j+\zeta_{j_1}e_{j_1}+\dots+\zeta_{j_k}e_{j_k})+u(x+\zeta_j e_j)\notag\\
         &+u(x+\zeta_{j_1}e_{j_1}+\dots+\zeta_{j_k}e_{j_k})-u(x)\bigr]K_\tau(\zeta)\dx\d\zeta\notag\\
         &=\int_{\{\zeta_j>0\}}\int_{[0,L)^d}(u(x+\zeta_j e_j)-u(x))\Bigl[(u(x+\zeta_j e_j)-u(x))\notag\\
         &-(u(x+\zeta_je_j+\zeta_{j_1}e_{j_1}+\dots+\zeta_{j_k}e_{j_k})-u(x+\zeta_{j_1}e_{j_1}+\dots+\zeta_{j_k}e_{j_k}))\Bigr]K_\tau(\zeta)\dx\d\zeta\notag\\
         &=\frac12\int_{\{\zeta_j>0\}}\int_{[0,L)^d}\Bigl[(u(x+\zeta_j e_j)-u(x))^2\notag\\
         &-(u(x+\zeta_je_j+\zeta_{j_1}e_{j_1}+\dots+\zeta_{j_k}e_{j_k})-u(x+\zeta_{j_1}e_{j_1}+\dots+\zeta_{j_k}e_{j_k}))^2\Bigr]K_\tau(\zeta)\dx\d\zeta,\notag\\
         &+\frac12\int_{\{\zeta_j>0\}}\int_{[0,L)^d}\Bigl[(u(x+\zeta_j e_j)-u(x))\notag\\
         &-(u(x+\zeta_je_j+\zeta_{j_1}e_{j_1}+\dots+\zeta_{j_k}e_{j_k})-u(x+\zeta_{j_1}e_{j_1}+\dots+\zeta_{j_k}e_{j_k}))\Bigr]^2K_\tau(\zeta)\dx\d\zeta\notag
         \end{align}
         where in the last equation we used the identity $a(a-b)=\frac12[a^2-b^2+(a-b)^2]$ with $a=u(x+\zeta_j e_j)-u(x)$ and $b=u(x+\zeta_je_j+\zeta_{j_1}e_{j_1}+\dots+\zeta_{j_k}e_{j_k})-u(x+\zeta_{j_1}e_{j_1}+\dots+\zeta_{j_k}e_{j_k})$.
         Thus, since  for $a,b$ as above it holds $\int a^2=\int b^2$, we conclude that
         \begin{align}
         -&\int_{\R^d}\int_{[0,L)^d}(u(x)-u(x+\zeta_j e_j))(u(x+\zeta_j e_j)-u(x+\zeta_je_j+\zeta_{j_1}e_{j_1}+\dots+\zeta_{j_k}e_{j_k}))K_\tau(\zeta)\dx\d\zeta=\notag\\
         &=\frac12\int_{\{\zeta_j>0\}}\int_{[0,L)^d}\Bigl[(u(x+\zeta_j e_j)-u(x))\notag\\
       &-(u(x+\zeta_je_j+\zeta_{j_1}e_{j_1}+\dots+\zeta_{j_k}e_{j_k})-u(x+\zeta_{j_1}e_{j_1}+\dots+\zeta_{j_k}e_{j_k}))\Bigl]^2K_\tau(\zeta)\dx\d\zeta.\label{eq:freccia3}
      \end{align}

\end{proof}

\begin{proof}[Proof of Proposition~\ref{prop:lowbound}]

Let us now start to rewrite the terms defining $\Fcal_{\tau,L, \varepsilon}$ in a way that will allow us to recover the lower bound \eqref{eq:3.13}.


First we notice that the Modica-Mortola term $\mathcal M_{\at}(\cdot,[0,L)^d)$ can be decomposed in the following way

\begin{align}\label{eq:3.1}
   \mathcal M_{\at}(u,[0,L)^d)=\sum_{i=1}^d\int_{[0,L)^{d-1}} \overline {\mathcal M}^{i}_{\at}(u,x_i^\perp,[0,L))\dx_i^\perp +\mathcal W_{\tau,L,\eps}(u).
\end{align}

To obtain~\eqref{eq:3.1} it is sufficient to notice that $\|\nabla u\|_1=\sum_i|\partial_iu|$ and use the Fubini Theorem \wrt the coordinate directions $e_i$, $i\in\{1,\dots,d\}$.

{As for the nonlocal term, using the elementary equality }

\begin{equation}\label{eq:elem}
   (a+b)^2=a^2+b^2+2ab
\end{equation}
 with $a=u(x)-u(x+\zeta_i e_i)$, $b=u(x+\zeta_i e_i)-u(x+\zeta)$  one has that

    \begin{align}
       -\int_{\R^d}\int_{[0,L)^d}|u(x)&-u(x+\zeta)|^2K_\tau(\zeta)\dx\d\zeta=-\int_{\R^d}\int_{[0,L)^d}|u(x)-u(x+\zeta_ie_i)|^2K_\tau(\zeta)\dx\d\zeta\notag\\
       &-2\int_{\R^d}\int_{[0,L)^d}(u(x)-u(x+\zeta_i e_i))(u(x+\zeta_i e_i)-u(x+\zeta))K_\tau(\zeta)\dx\d\zeta\notag\\
       &-\int_{\R^d}\int_{[0,L)^d}|u(x+\zeta)-u(x+\zeta_ie_i)|^2K_\tau(\zeta)\dx\d\zeta.\label{eq:dec2d}
    \end{align}
    Then one decomposes further the third term in the \rhs of~\eqref{eq:dec2d} using the elementary equality~\eqref{eq:elem} with $a=u(x+ \zeta_ie_i)-u(x+\zeta_ie_i+\zeta_{k_1}e_{k_1})$ and $b=u(x+\zeta_ie_i+\zeta_{k_1}e_{k_1})-u(x+\zeta)$, where $k_1\in\{1,\dots,d\}$ is the first index such that $k_1\neq i$. In this way, by periodicity of $u$
    \begin{align}
    -\int_{\R^d}&\int_{[0,L)^d}|u(x)-u(x+\zeta)|^2K_\tau(\zeta)\dx\d\zeta=-\int_{\R^d}\int_{[0,L)^d}|u(x)-u(x+\zeta_ie_i)|^2K_\tau(\zeta)\dx\d\zeta\notag\\
    &-\int_{\R^d}\int_{[0,L)^d}|u(x)-u(x+\zeta_{k_1}e_{k_1})|^2K_\tau(\zeta)\dx\d\zeta\notag\\
    &-2\int_{\R^d}\int_{[0,L)^d}(u(x)-u(x+\zeta_i e_i))(u(x+\zeta_i e_i)-u(x+\zeta))K_\tau(\zeta)\dx\d\zeta\notag\\
    &-2\int_{\R^d}\int_{[0,L)^d}(u(x+\zeta_ie_i)-u(x+\zeta_i e_i+\zeta_{k_1}e_{k_1}))(u(x+\zeta_i e_i+\zeta_{k_1}e_{k_1})-u(x+\zeta))K_\tau(\zeta)\dx\d\zeta\notag\\
    &-\int_{\R^d}\int_{[0,L)^d}|u(x+\zeta)-u(x+\zeta_ie_i+\zeta_{k_1}e_{k_1})|^2K_\tau(\zeta)\dx\d\zeta.\label{eq:dec2dbis}
    \end{align}
    Iterating this procedure on the last term of the r.h.s.\ of~\eqref{eq:dec2dbis} in the remaining $d-2$ coordinates $k_2,\dots,k_{d-1}\neq i$ and using the periodicity of $u$ one obtains
    \begin{align}
       &-\int_{\R^d}\int_{[0,L)^d}|u(x)-u(x+\zeta)|^2K_\tau(\zeta)\dx\d\zeta=-\sum_{i=1}^d\int_{\R^d}\int_{[0,L)^d}|u(x)-u(x+\zeta_ie_i)|^2K_\tau(\zeta)\dx\d\zeta\label{eq:i0}\\
       &-2\int_{\R^d}\int_{[0,L)^d}(u(x)-u(x+\zeta_i e_i))(u(x+\zeta_i e_i)-u(x+\zeta))K_\tau(\zeta)\dx\d\zeta\label{eq:i1}\\
       &-2\sum_{j=1 }^{d-1}\int_{\R^d}\int_{[0,L)^d}(u(x+\zeta_ie_i+\hat{\zeta}_{i,{j-1}})-u(x+\zeta_i e_i+\hat{\zeta}_{i,j}))(u(x+\zeta_i e_i+\hat{\zeta}_{i,j})-u(x+\zeta))K_\tau(\zeta)\dx\d\zeta\label{eq:i2}
    \end{align}
    where $\hat{\zeta}_{i,0}=0$ and for $j\geq 1$
    \[
       \hat{\zeta}_{i,j}=\zeta_1e_1+\dots\zeta_{i-1}e_{i-1}+\zeta_{i+1}e_{i+1}+\dots\zeta_{j+1}e_{j+1}.
    \]

By periodicity,~\eqref{eq:i1} and~\eqref{eq:i2} rewrite in the form~\eqref{eq:pos} and therefore by Lemma \ref{lemma:positivity} are nonnegative.
Thus, neglecting the positive term~\eqref{eq:i2}, summing the terms of~\eqref{eq:i0} and~\eqref{eq:i1} over $i\in\{1,\dots,d\}$ and dividing by $d$ one obtains

\begin{align}
   - \int_{\R^d}\int_{[0,L)^d}&|u(x)-u(x+\zeta)|^2K_\tau(\zeta)\dx\d\zeta\geq-\sum_{i=1}^d\int_{\R^d}\int_{[0,L)^d}|u(x)-u(x+\zeta_i e_i)|^2K_\tau(\zeta)\dx\d\zeta\notag\\
   & +\frac{1}{d}\sum_{i=1}^d\int_{\{\zeta_i>0\}}{\int_{[0,L)^d}}[(u(x+\zeta_i e_i)-u(x))-(u(x+\zeta)-u(x+\zeta_i^\perp))]^2K_\tau(\zeta)\dx\d\zeta.\label{eq:3.15}
\end{align}

Hence, applying the decomposition \eqref{eq:3.1} and \eqref{eq:3.15} to the definition of $\Fcal_{\tau,L, \varepsilon}$ given in \eqref{def:fcalt}, Proposition~\ref{prop:lowbound} is proved.

\end{proof}

   Given the numerous slicing arguments, it will be useful
   to define the slicing of $\mathcal I^i_{\tau,L}$ as follows
   \begin{equation*}
      \mathcal I^i_{\tau,L}(u)=\int_{[0,L)^{d-1}}\overline{\mathcal I}^i_{\tau}(u,x_i^\perp,[0,L))\dx_i^\perp,
   \end{equation*}
   where
   \begin{equation}
      \label{eq:slicingMathCalI}
      \overline{\mathcal I}^i_{\tau}(u,x_i^\perp,[0,L)) :=\frac{1}{d}\int_0^L\int_{\{\zeta_i>0\}}[(u(x+\zeta_i e_i)-u(x))-(u(x+\zeta)-u(x+\zeta_i^\perp))]^2K_\tau(\zeta)\d\zeta\dx_i
   \end{equation}
   and where $x = x_ie_i + x_i^\perp$.

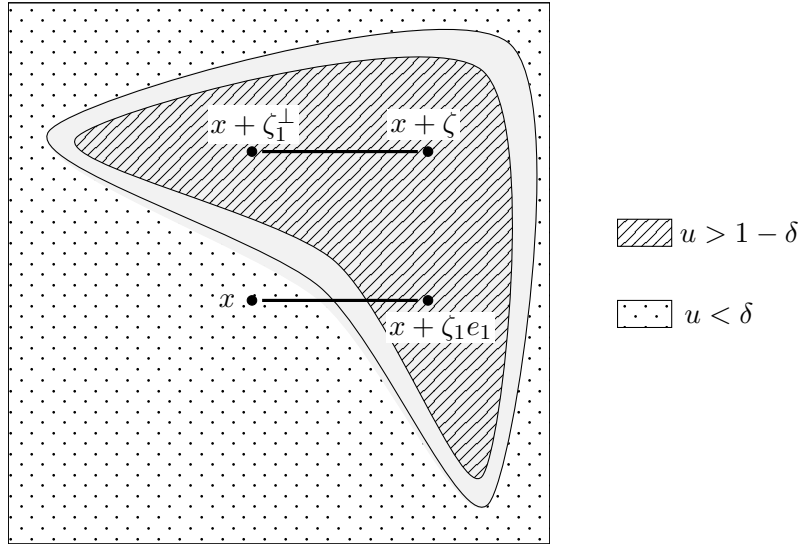
\begin{figure}
   \centering
\begin{tikzpicture}[scale=1.8]

	\node at (-4,.2) {$$};

	\draw [ pattern=nel] (2.5,.2) rectangle (2.9,.4);
	\node at (3.4,.3) {$u>1-\delta$};

	\draw [ pattern=my crosshatch dots] (2.5,-.2) rectangle (2.9,-.4);
	\node at (3.25,-.3) {$u < \delta$};

	\clip (-2, -2) rectangle (2, 2);
	\draw (-2,-2) rectangle (2,2);

	\draw [ pattern=my crosshatch dots] (-2,-2) rectangle (2,2);

	\filldraw [white!95!black] plot[smooth cycle] coordinates {
		(-1.7,1.05)
		(1.7,1.7)
		(1.55,-1.7)
		(.3,-.2)
	};

	\draw [black] plot[smooth cycle] coordinates {
		(-1.7,1.05)
		(1.7,1.7)
		(1.55,-1.7)
		(.3,-.1)
	};

	\draw [ pattern=nel] plot[smooth cycle] coordinates {
		(-1.5,1)
		(1.5,1.5)
		(1.5,-1.5)
		(.4,.1)
	};

	\node (XX) at  (-.2 ,-.2) {};
	\node (XY) at  (1.1 ,-.2) {};
	\node (XZ) at (-.2 ,.9) {};
	\node (ZZ) at (1.1 ,.9) {};
	\filldraw (XX) circle  (1pt);
	\filldraw (XY) circle  (1pt);
	\filldraw (XZ) circle  (1pt);
	\filldraw (ZZ) circle  (1pt);
	\draw[line width =1.2pt] (XX) -- (XY);
	\draw[line width = 1.2pt] (XZ) -- (ZZ);

	\node[inner sep=1pt, fill =white, left]  (X) at (-0.3,-0.2) {$x$};
	\node[inner sep=1pt, fill =white, left]  (Z1) at (1.6, -.40 ) {$x + \zeta_1 e_1$};
	\node[inner sep=1pt, fill =white, left]  (Z2) at (0.15, 1.1 ) {$x + \zeta^\perp_1$};
	\node[inner sep=1pt, fill =white, left]  (Z2) at (1.34, 1.1 ) {$x + \zeta$};

\end{tikzpicture}
\caption{\small{In this configuration, $u(x_1+\zeta_1e_1)-u(x)\geq1-2\delta$, while $|u(x+\zeta)-u(x+\zeta_1^\perp)|\leq\delta$. Hence,
$\Bigl[(u(x_1+\zeta_1e_1)-u(x))-(u(x+\zeta)-u(x+\zeta_1^\perp))\Bigr]^2K_\tau(\zeta)\geq(1-3\delta)^2\frac{1}{(\|\zeta\|_1+\tau^{1/\beta})^p}$.}}
\label{fig:1}
\end{figure}

\begin{remark}
  \label{rem:cross} The integrand in the cross interaction term $\Ical^i_{\tau,L}$ penalizes whenever the function $u$ is non-constant in more than one coordinate direction, \ie~whenever the function  $u$ is not ``one-dimensional''.
  For example, a configuration penalized by  $\Ical^i_{\tau,L}$  is depicted in Figure~\ref{fig:1}.
\end{remark}


\section{One-dimensional estimates}
\label{sec:1dest}

By Young inequality one has the following property
\begin{align}
   \Mi{s}{t}{i}&\geq6\int_s^t|\partial_iu_{x_i^\perp}(\rho)|\sqrt{W(u_{x_i^\perp}(\rho))}\d\rho\notag\\
   &=\int_s^t|D(\omega\circ u_{x_i^\perp})(\rho)|\d\rho\notag\\
   &\geq|\omega(u_{x_i^\perp}(s))-\omega(u_{x_i^\perp}(t))|,\label{eq:momega}
\end{align}

where $\omega:[0,1]\to[0,1]$ is defined by
\begin{equation}
   \label{eq:stimaOmega}
   \omega(t)=\int_0^t6\sqrt{W(s)}\ds=3t^2-2t^3.
\end{equation}

Notice that $\omega(t)$ is the optimal transition energy from $0$ to $t$ for the Modica Mortola term.  
The following remark contains an estimate relating $\omega$ and the square of the distance which will be used in Lemma~\ref{lemma:gpos} and in Proposition~\ref{prop:stability}.

\begin{remark}
   \label{lemma:omegaab}
   The optimal energy transition function  $\omega$ satisfies the following inequality: for $a,b\in[0,1]$ with $a=b+t$, $t>0$

   \begin{equation}\label{eq:omegaab}
      \frac{\omega(a)-\omega(b)}{|a-b|^2}=\frac{6b(1-b-t)}{t}+3-2t\geq3-2t\geq1.
   \end{equation}
   and in the last inequality equality holds if and only if $a=1$ and  $b=0$. The proof follows immediately from the definition of $\omega$ in~\eqref{eq:stimaOmega}.
\end{remark}

   In the following remark we collect a simple fact on periodic functions.

   \begin{remark}
      \label{lemma:simple_lemma_periodicity} 
      Let $f$ be an $L$-periodic function and $\zeta>0$. Due to the $L$-periodicity and to Fubini Theorem we have that
      \begin{equation*}
            \begin{split}
            \int_{0}^L \int_{t}^{t+\zeta} f(s) \ds \dt = 
            \zeta \int_{0}^L f(s) \ds. 
         \end{split}
      \end{equation*}
Indeed, for any $L$-periodic function $h$ we have that
      \begin{equation}\label{eq:hper}
         \begin{split}
            \int_{0}^L h(t) \dt = \int_{0}^L h(t+\zeta) \dt. 
         \end{split}
      \end{equation}
      Thus by~\eqref{eq:hper} with $h(t)=\int_{[t-\zeta,t]}f(s)\ds$ and Fubini Theorem we have that
      \begin{equation*}
         \begin{split}
            \int_{0}^L \int_{t}^{t+\zeta} f(s) \ds \dt   &= 
            \int_{0}^L \int_{t-\zeta}^{t} f(s) \ds \dt   \\
            &= \int_{0}^L \int_{0}^L \chi_{[t-\zeta,t]}(s)  f(s) \dt\ds \\   & =  \int_{0}^L \int_{0}^L \chi_{[s,s+\zeta]}(t)  f(s) \dt \ds  = \zeta\int_{0}^{L} f(s)\ds.
         \end{split}
      \end{equation*}
    \end{remark}


In the following lemma we prove the nonnegativity of $\Gcal{i}$.

\begin{lemma}
   \label{lemma:gpos}
   For any $\zeta_i\in\R$, 
   \begin{equation}
      \label{eq:partpos0}
      \begin{split}
         |\zeta_i|\Mi{0}{L}{i} & =\int_{0}^L \Mi{x_i}{x_i+\zeta_i}{i}\dx_{i}\\ &\geq\int_{0}^L|\omega(u_{x_{i}^{\perp}} (x_i +  \zeta_i)) - \omega(u_{x_{i}^{\perp}}(x_{i}))|\dx_i.
      \end{split}
   \end{equation}
   In particular, $\Gcal{i}\geq0$ and equality holds if and only if $u_{x_i^\perp}$ is constant.
\end{lemma}

\begin{proof}

   By using the definition of $\overline{\mathcal M}^{ i}_{\alpha_\varepsilon, \tau}$, Remark~\ref{lemma:simple_lemma_periodicity} and~\eqref{eq:momega},  one obtains~\eqref{eq:partpos0}.
   
   Finally, thanks to~\eqref{eq:omegaab},
   \begin{equation}
      \label{eq:gstrf2}
      |\omega(u_{x_i^\perp}(x_i))-\omega(u_{x_i^\perp}(x_i+\zeta_i ))|\geq |u_{x_i^\perp}(x_i)-u_{x_i^\perp}(x_i+\zeta_i)|^2,
   \end{equation}
   which, combined with~\eqref{eq:partpos0}, proves the nonnegativity of $\Gcal{i}$ (see~\eqref{def:Gcal}).
   {
     To prove strict positivity when $u_{x_i^\perp}$ is not constant it is sufficient to notice that, by~\eqref{eq:omegaab},~\eqref{eq:gstrf2} is a strict inequality unless $|u_{x_{i}^{\perp}}(x_i )  - u_{x_{i}^{\perp}}(x_i +\zeta_i)| \in \{ 0,1\}$.
   }
   However, since $u_{x_i^\perp}\in W^{1,2}_{\loc}$, we have that in this case it has necessarily to hold
   \begin{equation*}
      \begin{split}
         |u_{x_{i}^{\perp}}(x_i )  - u_{x_{i}^{\perp}}(x_i +\zeta_i)| =0,
      \end{split}
   \end{equation*}
   i.e. $u_{x_i^\perp}$ is constant.
   Moreover, if  $u_{x_i^\perp}$ is constant, then equality holds also in~\eqref{eq:partpos0} (cf.~\eqref{eq:3.2} for the definition of $\overline{\Mcal}^{i}_{\alpha_\varepsilon, \tau}$), hence $\Gcal{i}=0$.
\end{proof}

   In particular, the following lemma holds

   \begin{lemma}
      \label{lemma:jc}
      If $J\geq J_c$, where $J_c$ is defined in~\eqref{eq:jc}, then minimizers of~\eqref{E:F} are either $u\equiv1$ or $u\equiv0$.
   \end{lemma}

   \begin{proof}
      Let us initially observe that for $J\geq J_c$ it holds
      \begin{equation}\label{eq:jjc}
      \tilde{\mathcal F}_{J,L,\varepsilon} \geq \tilde{\mathcal  F}_{J_c,L,\varepsilon}.
      \end{equation} 

      Let us now recall that for $\tau=1$ one has that $\alpha_{\varepsilon,\tau}= \alpha_{\varepsilon,1}= \varepsilon$ and  $K_{\tau}(\zeta)= K_{1}(\zeta) = K(\zeta)$. Moreover, from  the definition of $J_c$ one has that
      \begin{equation*}
         \begin{split}
            \tilde{\mathcal F}_{J_c, L ,\varepsilon} (u) =& \frac{1}{L^d}\int_{\R^d} \int_{[0,L)^{d}} |\zeta_1| 
            \Big( 3\varepsilon\|\nabla u(x)\|_1^2 + \frac{3W(u(x))}{\varepsilon}  \Big)\  K(\zeta)  \dx \d\zeta 
            \\ &-\frac{1}{L^d}\int_{\R^d}\int_{[0,L)^{d}} |u(x+\zeta)-u(x)|^2 K(\zeta)\dx\d\zeta
         \end{split}
      \end{equation*}
      Recalling the definition of $\mathcal I^i_{1,L}$, as in~\eqref{eq:3.15}   we have that
      \[
         -\int_{\R^d}\int_{[0,L)^d}|u(x)-u(x+\zeta)|^2K(\zeta)\dx\d\zeta\geq\sum_{i=1}^d\Bigl\{
         -\int_{\R^d}\int_{[0,L)^d}|u(x)-u(x+\zeta_ie_i)|^2K(\zeta)\dx\d\zeta+\mathcal I^i_{1,L}(u)\Bigr\}.
      \]
      and using the definition of $\overline{\mathcal G}^i_{\varepsilon,1}$,  we have that
      \begin{equation*}
         \begin{split}
            \tilde{\mathcal F}_{J_c, L ,\varepsilon} (u)
            \geq \frac{1}{L^d}\sum_{i=1}^d\Bigl[\int_{[0,L)^{d-1}}\overline{\mathcal G}^i_{\varepsilon,1}(u, x^\perp_i, [0,L))\dx_{i}^\perp + \mathcal I^i_{1,L}(u)\Bigr]+\frac{1}{L^d} \mathcal W_{\tau,L,\eps}(u).
         \end{split}
      \end{equation*}

         By Lemma~\ref{lemma:gpos} and by definition  $\overline{\mathcal{G}}^{i}_{\eps,1}$, $\mathcal I^i_{1,L}$ and $\mathcal W_{\tau,L,\eps}$ are nonnegative.
         On the one hand, $\mathcal I^i_{1,L}(u)=0$ if and only if $u$ is one-dimensional.
         On the other hand, by Lemma~\ref{lemma:gpos}, one has that $\overline{\mathcal{G}}^{i}_{\eps,1}$ is zero if and only if $u_{x_i^\perp}$ is constant and $\Wcal_{\tau,L,\eps}(u)$ is zero if and only if $W(u)\chi_{\{\nabla u=0\}}\equiv0$. Hence $\tilde{\Fcal}_{J_c,L,\eps}$ is minimized by the constant functions $u\equiv0$ and $u\equiv1$. Since on such functions also $\tilde{\Fcal}_{J,L,\varepsilon}$ vanishes and~\eqref{eq:jjc} holds, the lemma is proved.
      
   \end{proof}


\section{Stability estimates}
\label{sec:stability}

In this section we assume that the $[0,L)^d$-periodic function $u$ is such that  $u\in W^{1,2}_{\loc}(\R^d;[0,1])$ and
\begin{equation}\label{eq:us}
   \|u-\chi_S\|_{L^1([0,L)^d)}\leq \bar{\sigma},
\end{equation}
for some $\bar{\sigma}>0$ sufficiently small (to be chosen later), where $S$ is a periodic union of stripes with boundaries orthogonal to $e_i$ and of width $h>0$.
As we will see in the proof of Theorem~\ref{Thm:1}, this is going to be the case for minimizers of $\Fcalt$ when $\eps,\tau$ are sufficiently small, due to Corollary~\ref{cor:gammaconv} and Theorem~\ref{Thm:DR}.

The main result of this section is the following stability estimate
\begin{proposition}
   \label{prop:stability}
   There exist $\bar{\sigma}>0$ and ${\tau'}>0$ such that, if~\eqref{eq:us} holds for  $u\in W^{1,2}_{\loc}(\R^d;[0,1])$ $[0,L)^d$-periodic function and $S$ periodic union of stripes with boundaries orthogonal to $e_i$ and of width $h>0$, then for all $j\in\{1,\dots,d\}$, $j\neq i$ and for all $0<\tau\leq\tau'$
   \begin{equation}
      \label{eq:propstability}
      \int_{[0,L)^{d-1}}\Bigl[-\Mi{0}{L}{j}+\Gcal{j}+\overline{\mathcal I}^j_{\tau}(u,x_i^\perp,[0,L))\Bigr]\dx_j^{\perp}\geq0
   \end{equation}	
   and equality holds if and only if $u$ does not depend on $x_j$.
\end{proposition}

Before going into the details of Proposition~\ref{prop:stability}, we collect some useful Lemmas. 
It might be convenient for the reader to start from the proof of Proposition~\ref{prop:stability} in page~\pageref{proof:prop:stability}, and return to the statements below when needed.

		\begin{lemma}
			\label{lemma:estimate2}
			Let $\delta_0,\delta>0$ and  let $u\in W^{1,2}_{\mathrm{loc}}(\R^d;[0,1])$ be a $[0,L)^d$-periodic function, $x_j^\perp\in[0,L)^{d-1}$ such that whenever  $|s-t| < \delta_{0}$ it holds $|u_{x_j^\perp}(s)- u_{x_j^\perp}(t)| \leq 1-\delta$. Then one has that
			
			\begin{equation}
			\label{eq:lemma:estimate2_1}
			\begin{split}
			\Gcal{j} \geq   \Big( \int_{-\delta_0}^{\delta_0} |\zeta_j| \widehat K_{\tau}(\zeta_j)\d\zeta_j \Big )  \frac{2\delta}{1+2\delta}  \Mi{0}{L}{j}
			\end{split}
			\end{equation}
			where equality holds if and only if $u_{x_j^\perp}$ is constant.
			\end{lemma} 
			
			\begin{remark}
            \label{rem:5.3}
              In particular, since for $p\geq d+1$ the following holds
              \[
                \lim_{\tau\to0}\int_{-\delta_0}^{\delta_0} |\zeta_j| \widehat K_{\tau}(\zeta_j)\d\zeta_j=+\infty.
              \]
              then  by choosing $\tau_1 := \tau_1(\delta,\delta_0)$ sufficiently small the inequality~\eqref{eq:lemma:estimate2_1} implies that, if $u_{x_j^\perp}$ is not constant, one has that 
              \begin{equation}
                \label{eq:lemma_estimate2_2}
                \begin{split}
                  \Gcal{j} > \Mi{0}{L}{j}.
                \end{split}
              \end{equation}
	\end{remark}
		
		\begin{proof}[Proof of Lemma~\ref{lemma:estimate2}]
			
			For any  $|\zeta_j| < \delta_{0}$, by using Remark~\ref{lemma:omegaab} and the hypothesis of the lemma, we have that
			\begin{equation*}
			\begin{split}
			\frac{1}{1+2\delta}|\omega(u_{x_j^\perp}(x_j + \zeta_j))-\omega(u_{x_j^\perp}(x_j))|\geq|u_{x_j^\perp}(x_j+\zeta_j)-u_{x_j^\perp}(x_j)|^2.
			\end{split}
			\end{equation*}
			Thus by using~\eqref{eq:partpos0} and the above, for any $|\zeta_j | < \delta_0 $ we have that
			\begin{equation*}
			\begin{split}
			\frac{|\zeta_j|}{1+2\delta}\Mi{0}{L}{j} - \int_{0}^L |u_{x_j^\perp}(x_j + \zeta_j) - u_{x_j^\perp}(x_j) |^2 \dx_j \geq 0.
			\end{split}
			\end{equation*}
			On the other hand, by~\eqref{eq:partpos0} and Remark~\ref{lemma:omegaab}, for any $\zeta_j$  we have
			\begin{equation}\label{eq:poselse}
			|\zeta_j|\Mi{0}{L}{j} - \int_{0}^L |u_{x_j^\perp}(x_j + \zeta_j) - u_{x_j^\perp}(x_j) |^2\dx_j  \geq 0.
			\end{equation}
			Hence, we have that
			\begin{equation*}
			\begin{split}
			\Gcal{j} \geq   \Big( \int_{-\delta_0}^{\delta_0} |\zeta_j| \widehat K_{\tau}(\zeta_j)\d\zeta_j \Big )  \frac{2\delta}{1+2\delta}  \Mi{0}{L}{j}
			\end{split}
			\end{equation*}
			Since in~\eqref{eq:poselse} equality holds for all $|\zeta_j|\geq\delta_0$ if only if $u_{x_j^\perp}$ is constant and in this case by Lemma~\ref{lemma:gpos} both $\Gcal{j}$ and $\Mi{0}{L}{j}$ are zero, the lemma is proved.
			
		\end{proof}



\begin{lemma}
	\label{lemma:estimate1}
	Let $\Upsilon>1$, let $u\in W^{1,2}_{\mathrm{loc}}(\R^d;[0,1])$ be a $[0,L)^d$-periodic function and let $\{ I_1,\ldots,I_N\}$ be disjoint closed intervals with $I_k \subset [0,L]$  and such that $\Mii{I_k}{j} \geq \Upsilon$.  
	Then 
	\begin{equation}
	\label{eq:lemma:estimate1_1}
	\begin{split}
	\Gcal{j} \geq \frac{\Upsilon-1}{\Upsilon}\sum_{k=1}^N \bar C_{\tau}(|I_k|) \Mii{I_k}{j},
	\end{split}
	\end{equation}
	where  
	\begin{equation*}
	\begin{split}
	\bar C_{\tau}(\eta)=  \eta \int_{\{|\zeta_j|>2\eta\}} \widehat{K}_{\tau}(\zeta_j) \d\zeta_j.
	\end{split}
	\end{equation*}
\end{lemma}

\begin{proof}
  
	Given that the intervals $I_k$ are disjoint, we have that for all $s<t\in[0,L]$
	\begin{equation*}
	\begin{split}
	\Mi{s}{t}{j} \geq \sum_{k:I_k \subset [s,t]} \Mii{I_k}{j}
	\end{split}
	\end{equation*}
	
	
	Moreover, for all $s,t$ such that $\Mi{s}{t}{j} \geq \Upsilon$ one has that
	\begin{align}
	\Mi{s}{t}{j} & = \frac{\Upsilon-1}{\Upsilon} \Mi{s}{t}{j}  + \frac{1}{\Upsilon} \Mi{s}{t}{j}
	\notag\\ & \geq \frac{\Upsilon-1}{\Upsilon}  \Mi{s}{t}{j} + 1
	\notag\\ & \geq \frac{\Upsilon-1}{\Upsilon}  \Mi{s}{t}{j} + (u_{x_j^\perp}(s)  - u_{x_j^\perp}(t))^2, \label{ineq:1}\\ 
   & \geq \frac{\Upsilon-1}{\Upsilon}  \sum_{k:\ I_k \subset [s,t]}\Mii{I_k}{j} + (u_{x_j^\perp}(s)  - u_{x_j^\perp}(t))^2,\label{eq:ineqgamma}
	\end{align}
	where to obtain~\eqref{ineq:1} we have used that $u_{x_j^\perp}\in[0,1]$ and thus $(u_{x_j^\perp}(s) - u_{x_j^\perp}(t))^2\leq 1$.
	
	
	Recalling~\eqref{eq:partpos0} and using~\eqref{eq:ineqgamma} we have that
	\begin{equation*}
	\begin{split}
	\Gcal{j} &\geq \int_{0}^L \int_{\R} \Mi{s}{s+\zeta_j}{j} \widehat K_{\tau}(\zeta_j) \d\zeta_j\ds \\
	&-\int_0^L\int_{\R}|u_{x_j^\perp}(s)-u_{x_j^\perp}(s+\zeta_j)|^2\widehat K_{\tau}(\zeta_j)\d\zeta_j\ds \\
	&\geq \frac{\Upsilon-1}{\Upsilon}\int_{0}^L \int_{0}^{+\infty}\sum_{k: I_k \subset [s,s+\zeta_j] } \Mii{I_k}{j} \widehat K_{\tau}(\zeta_j) \d\zeta_j\ds\\
	&\geq \frac{\Upsilon-1}{\Upsilon} \sum_{k} D_k \,\, \Mii{I_k}{j},
	\end{split}
	\end{equation*}
	where in the last inequality we have exchanged sum and integral and used the notation
	\begin{equation*}
	\begin{split}
	D_k = \int\int_{\{ (s,\zeta_j)\in [0,L)\times[0,+\infty):\ I_k \subset [s,s+\zeta_j]\}} \widehat K_\tau(\zeta_j) \ds  \d\zeta_j. 
	\end{split}
	\end{equation*}
	Due to the periodicity of $u$, we may assume without loss of generality that $I_{k} = [L,L+\eta]$, $\eta=|I_k|$. 
Fixing $\zeta_j$ with $|\zeta_j|>\eta$ we have that
	\begin{equation*}
	\begin{split}
	|\{s\in [0,L]: [s,s+\zeta_j] \supseteq I_k\}| = \min\{|\zeta_j|-\eta,L\}.
	\end{split}
	\end{equation*}
	hence,
	\begin{equation*}
	\begin{split}
	D_k = \int_{|\zeta_j| > \eta} \min\{|\zeta_j|-\eta,L\}\,\, \widehat K_{\tau}(\zeta_j) \d\zeta_j > 
	\eta\int_{|\zeta_j| > 2\eta} \widehat K_{\tau}(\zeta_j) \d\zeta_j,
	\end{split}
	\end{equation*}
	which yields the desired result. 
	
\end{proof}

As a consequence of~Lemma~\ref{lemma:estimate1}, one has the following
\begin{corollary}
   \label{cor:claim2}
	For all $\Upsilon>1$, ${\eta}>0$, whenever 
	$t_{1}\leq\ldots\leq t_{N}$ with $t_1 = 0 $, $t_N = L$ satisfy
	\begin{equation*}
	\begin{split}
	&\Mi{t_{k}}{t_{k+1}}{j} = \Upsilon\qquad\text{for $k=2,\ldots,N-2$},\qquad\text{and}\\
   &\Mi{t_{k}}{t_{k+1}}{j} \leq \Upsilon\qquad\text{for $k=1,N-1$},
	\end{split}
	\end{equation*}
	then
	\begin{align}
	- \Mi{0}{L}{j} + \Gcal{j} \geq -\Upsilon\frac{L}{{\eta}}+\sum_{k\in G_{{\eta}}}\Bigl(\frac{\Upsilon-1}{\Upsilon}\bar C_\tau(|[t_k,t_{k+1}]|)-1\Bigr)\Upsilon,\label{eq:firstineq}
	\end{align}
	where $G_{{\eta}} = \{ k:\ 2\leq k \leq N-2\text{ and  }|t_{k}  - t_{k+1}| \leq {\eta}\}$.

	In particular, there exist $\bar{\tau}>0$ and $\bar{\eta}>0$ such that for every $0<\tau\leq\bar \tau$
		\begin{equation}
	\label{eq:lowbound0}
	-\Mi{0}{L}{j}+\Gcal{j}\geq-\Upsilon\frac{L}{\bar{\eta}}.
	\end{equation}

Moreover, there exist $\tau_2>0$, $\eta_0>0$ such that  if $0<\tau\leq\tau_2$ and  $|t_{k+1}-t_k|<\eta_0$ for some $k$, then 
\begin{equation}
  \label{eq:-m+g>0}
	- \Mi{0}{L}{j} + \Gcal{j} >0.
\end{equation}
	\end{corollary}

\begin{proof} 
   Let $G^c_\eta$ be the complementary set, namely $G^c_{{\eta}} = \{ k:\ 2\leq k \leq N-2\text{ and  }|t_{k}  - t_{k+1}| > {\eta}\}$. 
   By~\eqref{eq:lemma:estimate1_1} one has that
	\begin{align}
	-\Mi{0}{L}{j}+\Gcal{j}&\geq -\sum_{k\in G_{{\eta}}^c}\Mi{t_k}{t_{k+1}}{j}\notag\\
	&+\sum_{k\in G_{{\eta}}}\Bigl(\frac{\Upsilon-1}{\Upsilon}\bar C_\tau(|[t_k,t_{k+1}]|)-1\Bigr)\Mi{t_{k}}{t_{k+1}}{j}\notag. 
	\end{align}
	Then,~\eqref{eq:firstineq} follows from the following two facts: there are at most $L/{\eta}$ intervals in $G_{{\eta}}^c$, on which $\Mi{t_k}{t_{k+1}}{j}\leq\Upsilon$, and for $k\in G_{{\eta}}$ one has that $\Mi{t_k}{t_{k+1}}{j}=\Upsilon$.
	
   Since
		\begin{equation}
	\label{eq:limCtau}
	\begin{split}
	\lim_{\eta\downarrow 0}\lim_{\tau\downarrow 0 }\bar C_{\tau}(\eta) = +\infty,
	\end{split}
	\end{equation}

	we can choose $\bar{\tau},\bar{\eta}$ such that for every $\tau< \bar{\tau}$ it holds
	\begin{equation*}
	\begin{split}
	\Bigl(\frac{\Upsilon-1}{\Upsilon}\bar C_{\tau}({\bar \eta})-1\Bigr)\Upsilon  >0
	\end{split}
	\end{equation*}
	and in particular 
	\begin{equation*}
	-\Mi{0}{L}{j}+\Gcal{j}\geq-\Upsilon\frac{L}{\bar{\eta}}.
	\end{equation*}
	Moreover, again by~\eqref{eq:limCtau}, there exist $\eta_0<\bar{\eta}$, $\tau_2<\bar{\tau}$ such that for  $\eta_0$, $\tau<\tau_2$ then
		\begin{equation*}
	\begin{split}
	\Bigl(\frac{\Upsilon-1}{\Upsilon}C_{\tau}({\eta_0})-1\Bigr)\Upsilon  >\Upsilon \frac{L}{\bar{\eta}},
	\end{split}
	\end{equation*}
and thus~\eqref{eq:-m+g>0} follows from~\eqref{eq:firstineq} as soon as there is an interval of size smaller than $\eta_0$ on which the Modica-Mortola term is greater than $\Upsilon$. 
	
	\end{proof}

\begin{lemma}
   \label{lemma:int}
   Let $S$ be a periodic union of stripes with boundaries orthogonal to $e_i$ and $u\in W^{1,2}_{\mathrm{loc}}(\R^d;[0,1])$ be a $[0,L)^d$-periodic function such that~\eqref{eq:us} holds.
   Then, for any $\alpha>0$ and  $|s_0-t_0|\leq\alpha$, if $\bar{\sigma}$ is sufficiently small and $j\neq i$
   \begin{equation}
   \label{eq:gstr3}
   \begin{split}
   \int_{\{ |\zeta_{j}^\perp| < \alpha \}} \int_{s_{0}-\alpha}^{s_{0}} \int_{t_0-x_j}^{t_0-x_j + \alpha}\!\! \Big\{\frac14- \big[ u(x_j^\perp + \zeta_j^\perp + x_j e_j ) - u(x_j^\perp + \zeta_j^\perp + x_je_j+ \zeta_j e_j)  \big]\Big\}^2 \d\zeta_j\dx_j\d\zeta_j^\perp > \frac{1}{8}\alpha^{d+1}.
   \end{split}
   \end{equation}
   \end{lemma}
 
 \begin{proof}
  First of all, we claim that for every $\mu>0$ there exists $\hat{\sigma}$ such that if the assumptions of the lemma hold with $\bar{\sigma}<\hat \sigma$ one has that
\begin{equation}\label{eq:int1}
   \begin{split}
      \frac{1}{\alpha^d}\int_{\{ |\zeta|_1 < \alpha \}}  \Big| u(x_j^\perp+x_je_j+ \zeta_j^\perp  ) - u(x_j^\perp + \zeta_j^\perp + x_je_j +\zeta_j e_j)\Big| \d\zeta \leq \mu. 
    \end{split}
\end{equation}
Indeed, suppose that the claim is false.
In this case there exists $\mu_0>0$ such that 
\begin{equation*}
   \begin{split}
      \mu_0 \alpha^d &< \int_{\{ |\zeta|_1 \leq \alpha\}}  \Big| u(x_j^\perp+ \zeta_j^\perp +x_je_j ) - u(x_j^\perp + \zeta_j^\perp + x_je_j +\zeta_j e_j) \Big| \d\zeta 
      \\ &\leq \int_{[0,L)^d}  \Big|u(x_j^\perp+ \zeta_j^\perp +x_je_j ) - u(x_j^\perp + \zeta_j^\perp + x_je_j +\zeta_j e_j) \Big|  \d\zeta.
   \end{split}
\end{equation*}
Given that for  $j\neq i$ $\chi_S(x_j^\perp+ \zeta_j^\perp +x_je_j ) - \chi_S(x_j^\perp + \zeta_j^\perp + x_je_j +\zeta_j e_j)=0$ and~\eqref{eq:us} holds, for  $\bar{\sigma}\downarrow0$ the \rhs~of the above converges to $0$ and then we obtain a contradiction.

The inequality~\eqref{eq:gstr3} is an immediate consequence of~\eqref{eq:int1} provided $\mu$ is sufficiently small.

\end{proof}


\begin{proof}[{Proof of Proposition~\ref{prop:stability}}]
\label{proof:prop:stability}

   Given $j\neq i$,  we want to show that
   \begin{equation}
      \label{eq:propstability1}
      \int_{[0,L)^{d-1}}\Bigl[-\Mi{0}{L}{j}+\Gcal{j}+\overline{\mathcal I}^j_{\tau}(u,x_i^\perp,[0,L))\Bigr]\dx_j^{\perp}\geq0.
   \end{equation}	
   We will show that the integrand of~\eqref{eq:propstability1}, namely
   \begin{equation}\label{eq:integrand0}
      -\Mi{0}{L}{j}+\Gcal{j}+\overline{\mathcal I}^j_{\tau}(u, x_j^\perp,[0,L))
   \end{equation}
   is non-negative and equal to $0$ if and only if $u$ does not depend on $x_j$. 

    We will use a partition $[0,L)^{d-1} = A_0 \cup A_1 \cup A_2 \cup A_3$, and show for each $x_j^\perp \in A_k$ with $k=1,2,3$ the expression in~\eqref{eq:integrand0} is strictly positive  and for $k=0$ the expression in~\eqref{eq:integrand0} is non-negative.

In order to define the sets $A_k$, let us introduce
   \begin{equation*}
      B_{x_j^\perp}:=\Bigl\{(s,t):\,s\in[0,L),\ t>s,  \ \ \Mi{s}{t}{j}\geq\frac{17}{16}\Bigr\}, 
   \end{equation*}
   and, for some $\delta>0$ sufficiently small (to be fixed later independently of $\eps$ and $\tau$),
   \begin{equation*}
      D_{x_j^\perp}:=\{(s,t):\,s\in[0,L), t\in\R, |u_{x_j^\perp}(s)-u_{x_j^\perp}(t)|\geq 1-2\delta\}.
   \end{equation*}
   Moreover, define
   \begin{align}
      b(x_j^\perp)&:=\inf\{|s-t|: (s,t)\in B_{x_j^\perp}\},\notag\\
      d(x_j^\perp)&:=\inf\{|s-t|: (s,t)\in D_{x_j^\perp}\},
   \end{align}
   and we set them equal to $+\infty$ if the corresponding sets are empty.

   Then fix $\delta_0,\eta_0>0$ and partition $[0,L)^{d-1}$ as follows
   \begin{equation*} 
      [0,L)^{d-1}=A_0\cup A_1\cup A_2\cup A_3 
   \end{equation*}
   where
   \begin{align}
      A_0 & := \insieme{x_j^\perp\in[0,L)^{d-1}:\ \text{ $u_{x_j^\perp}$ is constant}}\\
      A_1=A_1(\delta_0,\delta,\eta_0)&:=\{x_j^\perp\in[0,L)^{d-1}\setminus A_0:  \,b(x_j^\perp)\geq\eta_0,\,d(x_j^\perp)\geq\delta_0\}\\
      A_2=A_2(\eta_0)&:=\{x_j^\perp\in[0,L)^{d-1}:\ \,b(x_j^\perp)<\eta_0\}\\
      A_3=A_3(\delta_0,\delta,\eta_0)&:=\{x_j^\perp\in[0,L)^{d-1}:\ \, b(x_j^\perp)\geq\eta_0,\, d(x_j^\perp)\leq\delta_0\}.
   \end{align}

   In the proof we will show the following: for every $k=1,2,3$, {provided $\eta_0,\delta_0,\delta$, $\bar{\sigma}$ and $\tau'$ are small enough} it holds (\textbf{Claim $A_k$})
   \begin{equation}
      \label{eq:claimAK}
      \begin{split}
         -\Mi{0}{L}{j}+\Gcal{j} + \overline{\Ical}_\tau^j(u, x_{j}^\perp, [0,L))> 0, \quad\forall\, x_j^\perp \in A_k.
      \end{split}
   \end{equation}
   When $u_{x_j^\perp}$ is constant, namely $x_j^\perp\in A_0$, then~\eqref{eq:integrand0} reduces to $\overline{\mathcal I}_\tau^j(u, x_{j}^\perp, [0,L))$. 
   Thus its nonnegativity is trivial and follows from the definitions of the terms involved. 
   Moreover, if we show (\textbf{Claim $A_k$}) for $k=1,2,3$ it follows immediately that~\eqref{eq:propstability1} is an equality if and only if $|A_1|=|A_2|=|A_3|=0$. Indeed, in this case $u$ does not depend on $x_j$ and thus even on $A_0$ one has that
   $\overline{\mathcal I}_\tau^j(u, x_{j}^\perp, [0,L))=0$.

   Let us also recall that $\Gcal{j}$ and $\overline{\Ical}_\tau^{j}(u,x_j^\perp,[0,L))$ are nonnegative.
   The term $\Gcal{j}$ will be used to balance $-\Mi{0}{L}{j}$ for $x_j^\perp\in A_1\cup A_2$, while the term $\overline{\Ical}_\tau^{j}(u,x_j^\perp,[0,L))$ will be used only to prove~\eqref{eq:claimAK} for $k=3$.	The only set on which we will use the fact that $j\neq i$ and~\eqref{eq:us} holds is $A_3$.

   Let us now be more precise on how the parameters $\eta_0,\delta_0,\delta, \bar{\sigma}$ and $\tau'$ will be chosen:
   \begin{itemize}
      \item The parameter $\eta_0$ is chosen such that~\eqref{eq:-m+g>0} holds under the assumptions of Corollary~\ref{cor:claim2} with $\Upsilon=\frac{17}{16}$.
     \item The parameter $\delta$ is chosen such that the last inequality in~\eqref{eq:5.28bis} holds.
           One possible choice is  $\delta \leq  2 ^{-10}$.
      \item We choose $\bar{\eta}$ as in Corollary~\ref{cor:claim2} with $\Upsilon=\frac{17}{16}$ such that~\eqref{eq:lowbound0} holds.
      \item We choose $\alpha>0$ such that $\alpha<\eta_0/3$ and~\eqref{eq:5.31bis} holds.
      \item The parameter $\delta_0$ is chosen to satisfy $\delta_0\leq\alpha$.
      \item The parameter $\bar{\sigma}$ is such that~\eqref{eq:gstr3} holds for $\alpha$ as above.
      \item Finally one chooses $\tau'=\min\{\tau_1,\bar \tau,\tau_2,\tau_3\}$ where: for $\tau\leq\tau_1$~\eqref{eq:lemma_estimate2_2} holds, for $\tau\leq\bar{\tau}$ one has~\eqref{eq:lowbound0}, for $\tau\leq\tau_2$ one has~\eqref{eq:-m+g>0} and for $\tau\leq\tau_3$ one has that $\tau^{1/\beta}\leq\alpha$.
   \end{itemize}

   \vskip 0.2 cm

   \textbf{Claim $A_1$}

   By definition of $A_1$, for all $x_j^\perp\in A_1$ and $x_j\in[0,L)$, $|\zeta_j|<\delta_0$
   \[
      |u_{x_j^\perp}(x_j)-u_{x_j^\perp}(x_j+\zeta_j)|\leq 1-\delta.
   \]
   Therefore, the slices in $A_1$ are characterized by having phase transitions from values close to $0$ to values close to $1$ which are not ``sharp'' (i.e.\ require at least an interval of length $\delta_0$).

   In this case we are in the situation analysed in Lemma~\ref{lemma:estimate2}. Hence,~\eqref{eq:claimAK} holds provided $0<\tau\leq\tau_1$, where $\tau_1(\delta,\delta_0)$ is chosen as in Remark~\ref{rem:5.3}, namely so that~\eqref{eq:lemma_estimate2_2} holds.

   \vskip 0.3 cm	
   \textbf{Claim $A_2$} 

   In order to prove Claim $A_2$, take   $\Upsilon = 17/16$ and choose $\eta_0$, $0<\tau\leq\tau_2$ as in Corollary~\ref{cor:claim2} for such $\Upsilon$. By the assumptions on $x_j^\perp\in A_2$, it is not difficult to see that we can find $0=t_1\leq\ldots\leq t_N=L$ such that
   \begin{equation*}
      \begin{split}
         \Mi{t_k}{t_{k+1}}{j} &= \frac{17}{16} \qquad \text{ for $k=2,\ldots,N-2$ },\\ 
         \Mi{t_k}{t_{k+1}}{j} &\leq \frac{17}{16} \qquad \text{ for $k=1,N-1$ }
      \end{split}
   \end{equation*}
   and there exists $k\in\{2,N-2\}$ such that $|t_k - t_{k+1} | < \eta_0$. 
   Hence, the assumptions of  Corollary~\ref{cor:claim2} are satisfied and by~\eqref{eq:-m+g>0} we have the desired claim.

   \vskip 0.3 cm	
   \textbf{Claim $A_3$}

   We want to show that
   \begin{equation*}
      \begin{split}
         -\Mi{0}{L}{j}+\Gcal{j} + \overline{\Ical}_\tau^j(u, x_{j}^\perp, [0,L))> 0, \quad\forall\, x_j^\perp \in A_3.
      \end{split}
   \end{equation*}
   \vskip 0.2 cm	

      Before going into the details, let us give an idea of the proof. 
      The situation considered in this case is analogous to the image depicted in Figure~\ref{fig:2}. 
      \begin{figure}
         \centering
         \begin{tikzpicture}[scale=1.8]

   \node at (-4,.2) {$$};

   \draw [ pattern=nel] (2.5,.2) rectangle (2.9,.4);
   \node at (3.4,.3) {$u>1-\delta$};

   \draw [ pattern=my crosshatch dots] (2.5,-.2) rectangle (2.9,-.4);
   \node at (3.25,-.3) {$u < \delta$};

   \clip (-2, -2) rectangle (2, 2);
   \draw (-2,-2) rectangle (2,2);

   \draw [ pattern=nel] plot[smooth cycle] coordinates {
      (-10,-6)
      (2.799192683765379e-17, -6.0)
      (0.01485926758218416, -5.8788)
      (0.029802476129068323, -5.7576)
      (0.04384699405808514, -5.6364)
      (0.0559955074330209, -5.5152)
      (0.06530500655134194, -5.3939)
      (0.07092281960965897, -5.2727)
      (0.07219600833034359, -5.1515)
      (0.06869358295874177, -5.0303)
      (0.06026157231880619, -4.9091)
      (0.04705546034081432, -4.7879)
      (0.02955587737666491, -4.6667)
      (0.00856520998754399, -4.5455)
      (-0.014836187482785533, -4.4242)
      (-0.03925519342852494, -4.303)
      (-0.0631679184552604, -4.1818)
      (-0.08492559369072711, -4.0606)
      (-0.10286963457235164, -3.9394)
      (-0.11543800537165687, -3.8182)
      (-0.12127579603494786, -3.697)
      (-0.11934364478568121, -3.5758)
      (-0.10900505129307546, -3.4545)
      (-0.09015089265242522, -3.3333)
      (-0.06320858297707038, -3.2121)
      (-0.02919943178797972, -3.0909)
      (0.010270665565306968, -2.9697)
      (0.05305444561442353, -2.8485)
      (0.09654760201882771, -2.7273)
      (0.1378141963165086, -2.6061)
      (0.17377714882154155, -2.4848)
      (0.20130062758891, -2.3636)
      (0.2175932284922374, -2.2424)
      (0.22031594558372125, -2.1212)
      (0.20784609690826522, -2.0)
      (0.179488834670059, -1.8788)
      (0.13564846390887017, -1.7576)
      (0.07793831702019156, -1.6364)
      (0.009208315401522893, -1.5152)
      (-0.06659212849276086, -1.3939)
      (-0.144343750349391, -1.2727)
      (-0.21844291431753837, -1.1515)
      (-0.2829669290522049, -1.0303)
      (-0.33223458209119067, -0.9091)
      (-0.3613532988139411, -0.7879)
      (-0.36675797520087017, -0.6667)
      (-0.3466752423695679, -0.5455)
      (-0.30139699467754216, -0.4242)
      (-0.23356877544692262, -0.303)
      (-0.1478364781357881, -0.1818)
      (-0.05060097411974193, -0.0606)
      (0.05060097411974193, 0.0606)
      (0.1478364781357881, 0.1818)
      (0.23356877544692262, 0.303)
      (0.30139699467754216, 0.4242)
      (0.3466752423695679, 0.5455)
      (0.36675797520087017, 0.6667)
      (0.3613532988139411, 0.7879)
      (0.33223458209119067, 0.9091)
      (0.2829669290522049, 1.0303)
      (0.21844291431753837, 1.1515)
      (0.144343750349391, 1.2727)
      (0.06659212849276086, 1.3939)
      (-0.009208315401522893, 1.5152)
      (-0.07793831702019156, 1.6364)
      (-0.13564846390887017, 1.7576)
      (-0.179488834670059, 1.8788)
      (-0.20784609690826522, 2.0)
      (-0.22031594558372125, 2.1212)
      (-0.2175932284922374, 2.2424)
      (-0.20130062758891, 2.3636)
      (-0.17377714882154155, 2.4848)
      (-0.1378141963165086, 2.6061)
      (-0.09654760201882771, 2.7273)
      (-0.05305444561442353, 2.8485)
      (-0.010270665565306968, 2.9697)
      (0.02919943178797972, 3.0909)
      (0.06320858297707038, 3.2121)
      (0.09015089265242522, 3.3333)
      (0.10900505129307546, 3.4545)
      (0.11934364478568121, 3.5758)
      (0.12127579603494786, 3.697)
      (0.11543800537165687, 3.8182)
      (0.10286963457235164, 3.9394)
      (0.08492559369072711, 4.0606)
      (0.0631679184552604, 4.1818)
      (0.03925519342852494, 4.303)
      (0.014836187482785533, 4.4242)
      (-0.00856520998754399, 4.5455)
      (-0.02955587737666491, 4.6667)
      (-0.04705546034081432, 4.7879)
      (-0.06026157231880619, 4.9091)
      (-0.06869358295874177, 5.0303)
      (-0.07219600833034359, 5.1515)
      (-0.07092281960965897, 5.2727)
      (-0.06530500655134194, 5.3939)
      (-0.0559955074330209, 5.5152)
      (-0.04384699405808514, 5.6364)
      (-0.029802476129068323, 5.7576)
      (-0.01485926758218416, 5.8788)
      (-2.799192683765379e-17, 6.0)
      (-10,6)
   };
   \draw [ pattern=my crosshatch dots] plot[smooth cycle] coordinates {
      (10,-6)
      (0.20000000000000004, -6.0)
      (0.21485926758218418, -5.8788)
      (0.22980247612906834, -5.7576)
      (0.24384699405808516, -5.6364)
      (0.2559955074330209, -5.5152)
      (0.2653050065513419, -5.3939)
      (0.27092281960965897, -5.2727)
      (0.27219600833034363, -5.1515)
      (0.2686935829587418, -5.0303)
      (0.2602615723188062, -4.9091)
      (0.24705546034081433, -4.7879)
      (0.2295558773766649, -4.6667)
      (0.208565209987544, -4.5455)
      (0.18516381251721448, -4.4242)
      (0.16074480657147508, -4.303)
      (0.1368320815447396, -4.1818)
      (0.1150744063092729, -4.0606)
      (0.09713036542764837, -3.9394)
      (0.08456199462834314, -3.8182)
      (0.07872420396505216, -3.697)
      (0.0806563552143188, -3.5758)
      (0.09099494870692455, -3.4545)
      (0.10984910734757479, -3.3333)
      (0.13679141702292963, -3.2121)
      (0.17080056821202028, -3.0909)
      (0.210270665565307, -2.9697)
      (0.25305444561442353, -2.8485)
      (0.29654760201882774, -2.7273)
      (0.3378141963165086, -2.6061)
      (0.37377714882154156, -2.4848)
      (0.40130062758891005, -2.3636)
      (0.4175932284922374, -2.2424)
      (0.42031594558372126, -2.1212)
      (0.40784609690826523, -2.0)
      (0.379488834670059, -1.8788)
      (0.33564846390887015, -1.7576)
      (0.27793831702019156, -1.6364)
      (0.2092083154015229, -1.5152)
      (0.13340787150723915, -1.3939)
      (0.055656249650609, -1.2727)
      (-0.01844291431753836, -1.1515)
      (-0.08296692905220487, -1.0303)
      (-0.13223458209119066, -0.9091)
      (-0.16135329881394112, -0.7879)
      (-0.16675797520087016, -0.6667)
      (-0.14667524236956792, -0.5455)
      (-0.10139699467754215, -0.4242)
      (-0.033568775446922605, -0.303)
      (0.052163521864211915, -0.1818)
      (0.14939902588025808, -0.0606)
      (0.2506009741197419, 0.0606)
      (0.3478364781357881, 0.1818)
      (0.4335687754469226, 0.303)
      (0.5013969946775422, 0.4242)
      (0.546675242369568, 0.5455)
      (0.5667579752008702, 0.6667)
      (0.5613532988139411, 0.7879)
      (0.5322345820911907, 0.9091)
      (0.4829669290522049, 1.0303)
      (0.4184429143175384, 1.1515)
      (0.344343750349391, 1.2727)
      (0.26659212849276087, 1.3939)
      (0.19079168459847712, 1.5152)
      (0.12206168297980845, 1.6364)
      (0.06435153609112984, 1.7576)
      (0.020511165329941017, 1.8788)
      (-0.00784609690826521, 2.0)
      (-0.02031594558372124, 2.1212)
      (-0.01759322849223738, 2.2424)
      (-0.0013006275889100027, 2.3636)
      (0.02622285117845846, 2.4848)
      (0.062185803683491414, 2.6061)
      (0.1034523979811723, 2.7273)
      (0.1469455543855765, 2.8485)
      (0.18972933443469303, 2.9697)
      (0.22919943178797975, 3.0909)
      (0.2632085829770704, 3.2121)
      (0.2901508926524252, 3.3333)
      (0.3090050512930755, 3.4545)
      (0.3193436447856812, 3.5758)
      (0.3212757960349479, 3.697)
      (0.3154380053716569, 3.8182)
      (0.30286963457235166, 3.9394)
      (0.2849255936907271, 4.0606)
      (0.26316791845526044, 4.1818)
      (0.23925519342852494, 4.303)
      (0.21483618748278555, 4.4242)
      (0.19143479001245603, 4.5455)
      (0.1704441226233351, 4.6667)
      (0.1529445396591857, 4.7879)
      (0.13973842768119382, 4.9091)
      (0.13130641704125823, 5.0303)
      (0.12780399166965642, 5.1515)
      (0.12907718039034105, 5.2727)
      (0.13469499344865807, 5.3939)
      (0.14400449256697911, 5.5152)
      (0.15615300594191486, 5.6364)
      (0.17019752387093168, 5.7576)
      (0.18514073241781584, 5.8788)
      (0.19999999999999998, 6.0)
      (10,6)
   };
   \draw  (0,-.-6) -- (0,6);
   \filldraw (0,1.5) circle  (.6pt);
   \filldraw (0,0) circle    (.6pt);
   \filldraw (0,-.25) circle (.6pt);
   \filldraw (0,-1.18) circle (.6pt);
   \draw (0,1.5) -- (0,0);
   \node[inner sep=1pt, fill =white, left]  at (-0.14,1.5) { \Large $t_1$};
   \node[inner sep=1pt, fill =white, left]  at (-0.14,0.1) { \Large $t_0$};
   \node[inner sep=1pt, fill =white, left]  at (0.4,-0.25) { \Large $s_0$};
   \node[inner sep=1pt, fill =white, left]  at (0.4,-1.16) { \Large $s_1$};

\end{tikzpicture}
         \caption{In the above figure we depict a typical situation for $x_{j}^\perp \in A_{3}$.}
         \label{fig:2}
      \end{figure}
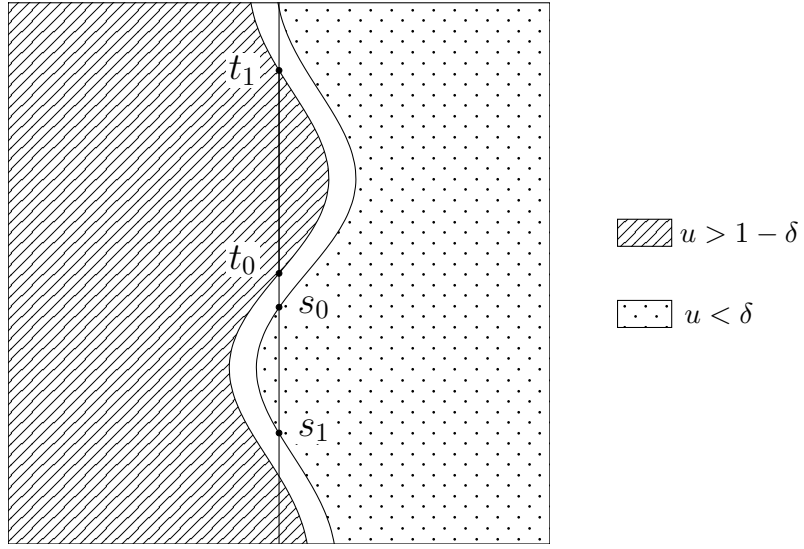
      Namely, there will be at least two  points $s_0$ and $t_0$ on the slice in direction $e_j$ orthogonal to $e_i$ where the function $u_{x_j^\perp}$ crosses the two thresholds $\delta$ and $1-\delta$ (see Figure~\ref{fig:2}). This transition, due to the definition of the set $A_3$, has to be ``almost sharp'' (\ie~happening in an interval $[s_0,t_0]$ of length controlled by $\delta_0$).
      On the other hand (see~Figure~\ref{fig:3}), the condition on the Modica-Mortola term ($b(x_{j}^\perp) > \eta_0$) together with the requirement that $\delta_0<\eta_0/3$ imposes that in a neighbourhood of size roughly $\eta_{0}$ around the transition, the function $u_{x_j^\perp}$ will be close to  $\delta$ on one side of the transition (interval $[s_0-\eta_{0}+\delta_0,s_0]$) and close to $1-\delta$ on the other side  of the transition (interval $[t_0,t_0 + \eta_0 - \delta_0]$). 

      Moreover, given that the function $u$ is close to a union of stripes with boundaries orthogonal to $e_i$, for most of the $\zeta^\perp_{j}$   the slice  $u_{x_j^\perp+\zeta_j^\perp}$ will take either values close to $1-\delta$ or values close to $\delta$. 
      Therefore the integrand of the cross interaction term $\overline{\mathcal I}_\tau^j(u,x_j^\perp,[0,L))$, which is given by
      \begin{equation}\label{eq:f}
         \begin{split}
            f(u, x_{j}^{\perp},x_{j}, \zeta_{j}^\perp,\zeta_j) := &\Big[\big(u(x_{j}^{\perp}+x_j e_j)-u(x_{j}^{\perp} +x_j e_j+\zeta_j e_j)\big)\\
            &-\big(u(x_{j}^\perp+\zeta_{j}^\perp + x_j e_j)-u(x_{j}^\perp+\zeta_j^\perp  +x_j e_j+\zeta_j e_j)\big)\Big]^2,
         \end{split}
      \end{equation}
  will be bigger than a given positive constant whenever $x_j\in[s_1,s_0]$, $x_j+\zeta_j\in[t_0,t_1]$ and either $u_{x_j^\perp+\zeta_j^\perp}\geq1-\delta$ or $u_{x_j^\perp+\zeta_j^\perp}\leq\delta$.
  Since this happens for most $\zeta$ in a neighbourhood of $0$ (being $\delta_0$ and $\bar\sigma$ in~\eqref{eq:us} small) and
  given that the kernel $K_{\tau}(\cdot)$ converges to a singular kernel $\zeta\to \frac{1}{|\zeta|^p}$,
      for $\tau$ sufficiently small the cross interaction term will be large implying~{\bf Claim $A_3$}.

      \vskip 0.2 cm	

      Let us now proceed with the formal proof.  Recall that our goal is to show that
   there exist $\delta_0>0$, $\delta>0$, $\tau_3>0$ and $\bar\sigma>0$ small enough such that if $x_j^\perp\in A_3(\delta_0,\delta,\eta_0)$ and $0<\tau\leq\tau_3$, $\|u-\chi_S\|_{L^1}\leq\bar{\sigma}$, then 
   \begin{equation*}
      -\Mi{0}{L}{j}+\Gcal{j}+\overline{\mathcal I}^j_{\tau}(u,x_j^\perp,[0,L))>0.
   \end{equation*} 

   By definition of $A_3$, for every $s\in[0,L)$, $t\in\R$ with $|s-t|<\eta_0$,$\Mi{s}{t}{j}\leq\frac{17}{16}$. Moreover by the second condition on $A_{3}$, there exist $s_0\in[0,L)$, $t_0\in\R$, $t_0>s_0$ with $|s_0-t_0|\leq\delta_0$ and $|u_{x_j^\perp}(s_0)-u_{x_j^\perp}(t_0)|\geq1-2\delta$. (see~also~Figure~\ref{fig:3}).  W.l.o.g., assume that $u_{x_j^\perp}(t_0)\geq1-\delta$ and $u_{x_j^\perp}(s_0)\leq\delta$. In particular, choosing $\delta_0\leq\eta_0/3$ by~\eqref{eq:momega} and Remark~\ref{lemma:omegaab}
   \begin{align*}
      \Mi{t_0}{t_0+\eta_0-\delta_0}{j}&\leq\Mi{s_0}{s_0+\eta_0}{j}-\Mi{s_0}{t_0}{j}\notag\\
      &\leq\frac{17}{16}-(u_{x_j^\perp}(s_0)-u_{x_j^\perp}(t_0))^2\notag\\
      &\leq\frac{1}{16}+4\delta
   \end{align*}
   Thus for every $t\in [t_{0},t_{0} + \eta_0 - \delta_{0} ]$ applying again~\eqref{eq:momega} and Lemma~\ref{eq:omegaab}, one has that if $\delta$ is small enough
   \begin{align}
         u_{x_j^\perp}(t) &\geq u_{x_j^\perp}(t_0) -\sqrt{\Mi{t_0}{t_0+\eta_0-\delta_0}{j}}\notag\\
         &\geq1-\delta-\frac14-2\sqrt{\delta}\geq\frac58.\label{eq:5.28bis}
   \end{align}
   Similarly, for $s\in[s_0 - \eta_0 +\delta_{0}, s_{0}]$
   \begin{figure}
      \centering
      \begin{tikzpicture}[scale=2]

   \draw[thick] (-2,-0.3) --  (2.2,-0.3);
   \draw[thick, dashed] (-2,.2) --  (2.2,.2);
   \draw[thick, dashed] (-2,1.8) -- (2.2,1.8);
   \draw[thick, dashed] (0.14,1.8) -- (.14,-.3);
   \node at (-2.4,1.8) {$1-\delta$};
   \node at (-2.3,.2) {$\delta$};
   \filldraw (-0.12,-.3) circle (1pt);
   \draw[thick, dashed] (-0.12,.2) -- (-0.12,-.3) ;
   \node at (-0.1,-.5) {$s_0$};
   \filldraw (0.14,-0.3) circle (1pt);
   \node at (.20,-.5) {$t_0$};
   \filldraw (1.7,-0.3) circle (1pt);
   \node at (1.7,-.5) {$t_0+ \eta_0 -\delta_0$};
   \filldraw (-1.4,-0.3) circle (1pt);
   \draw[thick, dashed] (-1.4,.25) -- (-1.4,-.3);
   \draw[thick, dashed] (1.7,1.8) -- (1.7,-.3);
   \node at (-1.4,-.5) {$s_0- \eta_0 +\delta_0$};
   \clip (-2, -1) rectangle (2, 2.6);

   \draw[thick]  plot[smooth] coordinates {
(-4.0, 0.6462)
(-3.95, 0.6243)
(-3.9, 0.6108)
(-3.85, 0.6089)
(-3.8, 0.6154)
(-3.75, 0.6218)
(-3.7, 0.62)
(-3.65, 0.6064)
(-3.6, 0.5846)
(-3.55, 0.5628)
(-3.5, 0.5492)
(-3.45, 0.5474)
(-3.4, 0.5538)
(-3.35, 0.5603)
(-3.3, 0.5585)
(-3.25, 0.5449)
(-3.2, 0.5231)
(-3.15, 0.5012)
(-3.1, 0.4877)
(-3.05, 0.4859)
(-3.0, 0.4923)
(-2.95, 0.4988)
(-2.9, 0.4969)
(-2.85, 0.4834)
(-2.8, 0.4615)
(-2.75, 0.4397)
(-2.7, 0.4262)
(-2.65, 0.4243)
(-2.6, 0.4308)
(-2.55, 0.4372)
(-2.5, 0.4354)
(-2.45, 0.4218)
(-2.4, 0.4)
(-2.35, 0.3782)
(-2.3, 0.3646)
(-2.25, 0.3628)
(-2.2, 0.3692)
(-2.15, 0.3757)
(-2.1, 0.3738)
(-2.05, 0.3603)
(-2.0, 0.3385)
(-1.95, 0.3166)
(-1.9, 0.3031)
(-1.85, 0.3012)
(-1.8, 0.3077)
(-1.75, 0.3141)
(-1.7, 0.3123)
(-1.65, 0.2988)
(-1.6, 0.2769)
(-1.55, 0.2551)
(-1.5, 0.2415)
(-1.45, 0.2397)
(-1.4, 0.2462)
(-1.35, 0.2526)
(-1.3, 0.2508)
(-1.25, 0.2372)
(-1.2, 0.2154)
(-1.15, 0.1936)
(-1.1, 0.18)
(-1.05, 0.1782)
(-1.0, 0.1846)
(-0.95, 0.1911)
(-0.9, 0.1892)
(-0.85, 0.1757)
(-0.8, 0.1538)
(-0.75, 0.132)
(-0.7, 0.1185)
(-0.65, 0.1166)
(-0.6, 0.1231)
(-0.55, 0.1296)
(-0.5, 0.1278)
(-0.45, 0.1144)
(-0.4, 0.093)
(-0.35, 0.0723)
(-0.3, 0.0619)
(-0.25, 0.0685)
(-0.2, 0.0975)
(-0.15, 0.1693)
(-0.1, 0.3092)
(-0.05, 0.5816)
(-0.0, 1.0)
(0.05, 1.4184)
(0.1, 1.6908)
(0.15, 1.8307)
(0.2, 1.9025)
(0.25, 1.9315)
(0.3, 1.9381)
(0.35, 1.9277)
(0.4, 1.907)
(0.45, 1.8856)
(0.5, 1.8722)
(0.55, 1.8704)
(0.6, 1.8769)
(0.65, 1.8834)
(0.7, 1.8815)
(0.75, 1.868)
(0.8, 1.8462)
(0.85, 1.8243)
(0.9, 1.8108)
(0.95, 1.8089)
(1.0, 1.8154)
(1.05, 1.8218)
(1.1, 1.82)
(1.15, 1.8064)
(1.2, 1.7846)
(1.25, 1.7628)
(1.3, 1.7492)
(1.35, 1.7474)
(1.4, 1.7538)
(1.45, 1.7603)
(1.5, 1.7585)
(1.55, 1.7449)
(1.6, 1.7231)
(1.65, 1.7012)
(1.7, 1.6877)
(1.75, 1.6859)
(1.8, 1.6923)
(1.85, 1.6988)
(1.9, 1.6969)
(1.95, 1.6834)
(2.0, 1.6615)
(2.05, 1.6397)
(2.1, 1.6262)
(2.15, 1.6243)
(2.2, 1.6308)
(2.25, 1.6372)
(2.3, 1.6354)
(2.35, 1.6218)
(2.4, 1.6)
(2.45, 1.5782)
(2.5, 1.5646)
(2.55, 1.5628)
(2.6, 1.5692)
(2.65, 1.5757)
(2.7, 1.5738)
(2.75, 1.5603)
(2.8, 1.5385)
(2.85, 1.5166)
(2.9, 1.5031)
(2.95, 1.5012)
(3.0, 1.5077)
(3.05, 1.5141)
(3.1, 1.5123)
(3.15, 1.4988)
(3.2, 1.4769)
(3.25, 1.4551)
(3.3, 1.4415)
(3.35, 1.4397)
(3.4, 1.4462)
(3.45, 1.4526)
(3.5, 1.4508)
(3.55, 1.4372)
(3.6, 1.4154)
(3.65, 1.3936)
(3.7, 1.38)
(3.75, 1.3782)
(3.8, 1.3846)
(3.85, 1.3911)
(3.9, 1.3892)
(3.95, 1.3757)
   };

\end{tikzpicture}
      \caption{
         In the above figure we depict a typical situation for a slice in $A_{3}$, with $|s_{0}-t_{0}| \leq \delta_0$, $\Mi{s_0}{t_0 + \eta_0 - \delta_0}{j}\leq \frac{17}{16}$ and $\Mi{s_0-\eta_0+\delta_0}{t_0}{j}\leq \frac{17}{16}$. 
      }
      \label{fig:3}
   \end{figure}
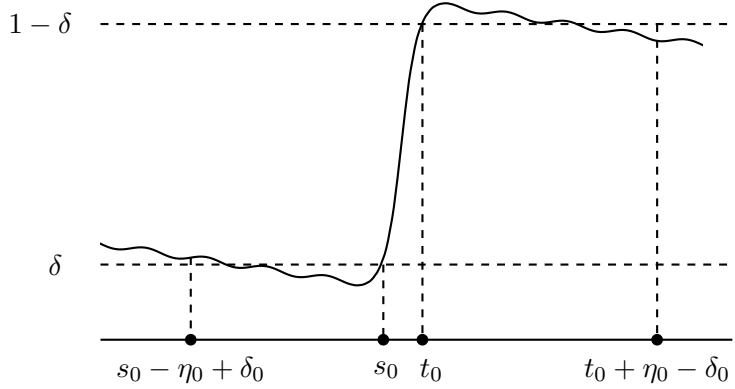

   \begin{equation*}
      \begin{split}
         u_{x_j^\perp}(s) \leq \frac38.
      \end{split}
   \end{equation*}

   Hence, for every $s\in[s_{0}-\eta_0 + \delta_{0},s_{0}]$ and every $t\in [t_{0}, t_{0}+\eta_0 - \delta_{0}] $ we have that  
   \begin{equation}\label{eq:gstr1}
      |u_{x_j^\perp}(t) - u_{x_j^\perp}(s)| \geq \frac14.
   \end{equation}

   Recalling~\eqref{eq:f}
   and using~\eqref{eq:gstr1} we have that
   \begin{equation*}
      \begin{split} 
         \int_{0}^L\int_{\R}&f(u, x_{j}^{\perp},x_{j}, \zeta_{j}^\perp,\zeta_j) \, K_\tau(\zeta)\d\zeta_j\dx_j\geq
         \int_{s_0 - \eta_0 +\delta_0}^{s_0}\int_{t_0-x_j}^{t_{0}-x_j+\eta_0 - \delta_0}
         f(u, x_{j}^{\perp},x_{j}, \zeta_{j}^\perp,\zeta_j)
         K_\tau(\zeta)\d\zeta_j\dx_j\\&
         \geq 
         \int_{s_0 - \eta_0 +\delta_0}^{s_0}\int_{t_0-x_j}^{t_{0}-x_j+\eta_0 - \delta_0}\Big(\frac{1}{4}- \big[ u(x_j^\perp + \zeta_j^\perp + x_j e_j ) - u(x_j^\perp + \zeta_j^\perp + x_je_j+ \zeta_j e_j)  \big]\Big)^2 K_{\tau}(\zeta)\d\zeta_j\dx_j\\ 
         &\geq
         \int_{s_0 - \alpha}^{s_0}\int_{t_0-x_j}^{t_{0}-x_j+\alpha}\Big(\frac14 - \big[ u(x_j^\perp + \zeta_j^\perp + x_j e_j ) - u(x_j^\perp + \zeta_j^\perp +  x_je_j+\zeta_j e_j)  \big]\Big)^2 K_{\tau}(\zeta)\d\zeta_j\dx_j
      \end{split}
   \end{equation*}
   where $\alpha < \eta_0 - \delta_0\leq\frac23\eta_0$ (to be chosen later). Integrating over $\zeta^\perp_j < \varepsilon$, and using the notation~\eqref{eq:slicingMathCalI}     we have that
   \begin{equation*}
      \begin{split}
         \overline{\Ical}^j_{\tau}&(u,x_j^\perp,[0,L)) \geq\\
         &\geq \frac{1}{d}\int_{\{ |\zeta_{j}^\perp| < \alpha \}} \int_{s_{0}-\alpha}^{s_{0}} \int_{t_0-x_j}^{t_0-x_j + \alpha} \Big(\frac14- \big[ u(x_j^\perp + \zeta_j^\perp + x_j e_j ) - u(x_j^\perp + \zeta_j^\perp +  x_je_j+\zeta_j e_j)  \big]\Big)^2 K_{\tau}(\zeta)\d\zeta \dx
         \\ &\geq\frac{1}{d}
         \int_{\{ |\zeta_{j}^\perp| < \alpha \}} \int_{s_{0}-\alpha}^{s_{0}} \int_{t_0-x_j}^{t_0-x_j + \alpha}  \frac{\Big(\frac14- \big[ u(x_j^\perp + \zeta_j^\perp + x_j e_j ) - u(x_j^\perp + \zeta_j^\perp +x_je_j+  \zeta_j e_j)  \big]\Big)^2}{(3\alpha+\delta_0 + \tau^{1/\beta})^{p}}\d\zeta \dx
      \end{split}
   \end{equation*}
   since $|\zeta_j|\leq2\alpha+\delta_0$ in the above integral. 

   By Lemma~\ref{lemma:int}, if $\delta_0\leq\alpha$ and  $\bar{\sigma}$ in~\eqref{eq:us} is  sufficiently small one has that
   \begin{equation*}
      \begin{split}
         \int_{\{ |\zeta_{j}^\perp| < \alpha \}} \int_{s_{0}-\alpha}^{s_{0}} \int_{t_0-x_j}^{t_0-x_j + \alpha} \Big(\frac14- \big[ u(x_j^\perp + \zeta_j^\perp + x_j e_j ) - u(x_j^\perp + \zeta_j^\perp +x_je_j + \zeta_j e_j)  \big]\Big)^2 \d\zeta_j\dx_j\d\zeta_j^\perp > \frac{1}{8}\alpha^{d+1}.
      \end{split}
   \end{equation*}
   Moreover assuming that $0<\tau\leq\tau_3$ is such that  $\alpha \geq \tau_3^{1/\beta}$ one has that ${1}/{(\tau^{1/\beta} + 3\alpha+\delta_0)^p} \geq 1/(5\alpha)^{p}$. 
   Thus since $p\geq d+2$, one has that $\overline{\Ical}^j_{\tau}(u, x_{j}^\perp, [0,L)) \geq \frac{1}{d5^p\alpha}$. 

   To conclude it is sufficient to observe that by Corollary~\ref{cor:claim2}, provided $\bar{\eta}$ is small enough
   \begin{equation*}
      -\Mi{0}{L}{j}+\Gcal{j}+\overline{\mathcal I}^j_{\tau}(u,x_j^\perp,[0,L))\geq  -  \frac{17}{16}\frac{L}{\bar{\eta}}  + \frac{1}{d5^p\alpha}, 
   \end{equation*} 
  thus by taking $\alpha $ such that 
   \begin{equation}\label{eq:5.31bis}
   -\frac{17}{16}\frac{L}{\bar{\eta}}+\frac{1}{d5^p\alpha}>0
   \end{equation} we have the desired claim. 
\end{proof}


\section{One-dimensional problem}
\label{sec:1d}

Let $L>0$, let $u\in W^{1,2}_{\loc}(\R^d;[0,1])$ be an $L$-periodic  one-dimensional function, namely $u(x) = g(x_i)$ for some $i \in \{1, \dots, d\}$ with $g$ $L$-periodic and $\gamma\geq1$ a measurable $L$-periodic function. We define the following one-dimensional functional

 \begin{align}\label{eq:functgamma0}
	\Fcal^{1d}_{\tau,L,\eps}(\gamma,g)=\frac{3(C_\tau-1)}{L}\int_0^L\Bigl[\eps\tau^{1/\beta}|g'(x)|^2\gamma(x)+\frac{W(g(x))}{\eps\tau^{1/\beta}\gamma(x)}\Bigr]\dx-\frac1L\int_0^L\int_{ \R}|g(x)-g(x+z)|^2\widehat K_\tau(z)\dz\dx,
\end{align} 
where 
\begin{equation}
	\label{eq:ctau}
	C_\tau=\int_{\R}\widehat{K}_\tau(\rho)|\rho|\d\rho.
\end{equation}

Notice that $\Fcalt^{1d}(1,g) = \Fcalt(u)$. We will need to consider the functional for general $\gamma\geq1$ in the proof of Theorem \ref{Thm:1}. In Section \ref{subsec:micoef} we will show that when $L=2kh^*_{\tau,\eps}$ the functional in \eqref{eq:functgamma0} is minimized when $\gamma=1$ a.e. (see Theorem \ref{thm:onedimmin}).

   For any $h > 0$, let $\mathcal{C}_h = \{g \in W^{1,2}([0,h];[0,1]): g \geq \frac{1}{2}, g(0)=g(h)=1/2 \}$.    We define also $\Gamma_h=\{\gamma:\R\to[1,+\infty]:\,
   \gamma(x+kh)=\gamma((k+1)h-x),\quad\forall\,k\in\N\}$.
   
    Given $g\in\mathcal C_h$ let us define the periodic reflection of $g$ as follows
  \begin{equation*}
      \begin{split}
         \varphi_{h}[g](x) = \begin{cases}
            g(x) &\text{if }x\in [0,h]\\
            1- g(2h- x) &\text{if }x\in [h, 2h]\\
            \varphi_h[g](y) &\text{where $y\in[0,2h]$ such that }x = y + 2kh\text{ with $k\in\Z$}.
         \end{cases}
      \end{split}
   \end{equation*}
Moreover, for all $\gamma:[0,h]\to[1,+\infty]$, let
\begin{equation*}
	\varphi_h[\gamma](x)= \begin{cases}
		\gamma(x) &\text{if }x\in [0,h]\\
		\gamma(x+kh)=\gamma((k+1)h-x)\quad&\forall\,k\in\N.
	\end{cases}
\end{equation*}

   Whenever clear from the context we will drop the $h$ from the index and write $\varphi[g](x)$ instead of $\varphi_{h}[g](x)$ and $\varphi[\gamma](x)$ instead of $\varphi_{h}[\gamma](x)$.


   \subsection{{Existence of an optimal period}}
   \label{ssec:1d}

Using the identity $|g(s) - g(s + \rho)|^2 = |(g(s) -\frac{1}{2})  - (g(s + \rho) - \frac{1}{2})|^2   $ we rearrange the above functional in the following way. 

Expanding the quadratic part in the nonlocal term, we have that
\begin{equation*}
	\begin{split}
		\Fcal_{\tau,L,\varepsilon}^{1d}(\gamma, g) = &\frac{3(C_\tau-1)}{L}\int_{0}^{L} \eps\tau^{1/\beta}\gamma(x)|g'(x)|^{2} \dx+ \frac{3(C_\tau-1)}{L}\int_{0}^{L}\frac{W(g(x))}{\eps\tau^{1/\beta}\gamma(x)}\dx\\
		&
		- 2 \int_{0}^{L}\int_{\R} \Bigl|g(x)-\frac{1}{2}\Bigr|^{2}  \widehat K_\tau(x-y)\dy\dx \\ &+ \frac{2}{L}\int_{0}^{L}\int_{\R}\Bigl(g(x)-\frac{1}{2}\Bigr)\Bigl(g(y)-\frac{1}{2}\Bigr)\widehat K_\tau(x-y)\dy\dx,\label{eq:dec12}
	\end{split}
\end{equation*}
where in the above  we have used
\begin{equation*}
	\begin{split}
		\int_{0}^{L}\int_{\R} \Bigl(g(y)-\frac12\Bigr)^{2}\widehat K_\tau(x-y) \dy\dx = \int_{0}^{L}\int_{\R}\Bigl( g(x)-\frac12\Bigr)^{2}\widehat K_\tau(x-y) \dy\dx.
	\end{split}
\end{equation*}
Indeed, by the periodicity of $g$ we have that
\begin{equation*}
	\begin{split}
		\int_{0}^{L}\int_{\R} \Bigl( g(y) - \frac{1}{2} \Bigr)^{2}\widehat K_\tau(x-y) \dy\dx &= \sum_{k\in \Z}\int_{0}^{L}\int_{0}^{L} \Bigl( g(y+kL) - \frac{1}{2} \Bigr)^{2}\widehat K_\tau(x-y - kL) \dy\dx \\
		\\ &= \sum_{k\in \Z}\int_{0}^{L}\int_{0}^{L} \Bigl( g(y) - \frac{1}{2} \Bigr)^{2}\widehat K_\tau(x-y - kL) \dy\dx
		\\ &= \sum_{k\in \Z}\int_{0}^{L}\int_{0}^{L} \Bigl( g(x)-\frac{1}{2} \Bigr)^{2}\widehat K_\tau(x-y - kL) \dy\dx
		\\ &= \int_{0}^{L}\int_{\R} \Bigl(g(x) - \frac{1}{2}\Bigr)^{2}\widehat K_\tau(x-y) \dy\dx
	\end{split}
\end{equation*}

For simplicity of notation let us denote the local part in \eqref{eq:dec12} as $A(\gamma,g,x)$. Thus the functional is written as
\begin{equation*}
	\begin{split}
		\Fcal_{\tau,L, \varepsilon}^{1d}(\gamma, g) = \frac{1}{L}\int_{0}^{L} A(\gamma, g,x)\dx + \frac{2}{L}\int_{0}^{L}\int_{\R}\Bigl(g(x) - \frac{1}{2}\Bigr) \Bigl(g(y) - \frac{1}{2}\Bigr) \widehat K_\tau(x-y)\dx\dy
	\end{split}
\end{equation*}
where  the function $A(\gamma,g,x)$ depends only on the values  $\gamma(x)$ and $g(x)$. 

Similarly to \cite{GLL1D} one has that 

\begin{align}\label{eq:barctaueps}
	\bar C^*_{\tau,\eps}&=\inf_L\inf_{\gamma,g\text{ $L$-per.}}\Fcal_{\tau,L, \varepsilon}^{1d}(\gamma,g)\notag\\
	&=\lim_{N\to\infty}\inf_{g\in W^{1,2}([-NL,NL]),\gamma\geq1}\bar\Fcal^{1d}_{\tau,[-NL,NL], \varepsilon}(\gamma,g)\notag\\
	&=\lim_{N\to\infty}\inf_{g\in W^{1,2}([-NL,NL]),g(-NL)=g(NL)=\frac12,\gamma\geq1}\bar\Fcal^{1d}_{\tau,[-NL,NL], \varepsilon}(\gamma,g),
\end{align}
where
\begin{align}\label{eq:NL}
	\bar\Fcal^{1d}_{\tau,[-NL,NL], \varepsilon}(\gamma,g)=\frac{1}{2NL}\int_{-NL}^{NL} A(\gamma, g,x)\dx + \frac{1}{NL}\int_{-NL}^{NL}\int_{-NL}^{NL}\Bigl(g(x) - \frac{1}{2}\Bigr) \Bigl(g(y) - \frac{1}{2}\Bigr) \widehat K_\tau(x-y)\dx\dy
\end{align}
Thus if one is interested to find functions ${\gamma},g$ which realize the value $\bar C^*_{\tau,\eps}$ then one can consider free boundary conditions or Dirichlet boundary conditions instead of $L$-periodicity. In particular, free boundary conditions will be convenient in the following argument, involving right and left reflections of the functions $g$ and $\gamma$ where periodicity is not preserved.

Similarly to \cite{GLL1D} one can show, by using the reflection positivity of the kernel $\widehat K_\tau$, that the following holds:

\begin{align}\label{eq:per0}
	\bar C^*_{\tau,\eps}=\inf_h\bar C_{\tau,\eps}(h), \quad\text{where}\quad \bar C_{\tau,\eps}(h)=\inf_{g\in \mathcal C_h,\gamma\in\Gamma_h}\lim_{N\to+\infty}\bar\Fcal^{1d}_{\tau,[-NL,NL], \varepsilon}(\varphi[\gamma],\varphi[g]).
\end{align}

Let us resume, for completeness, the main ideas of the proof of \eqref{eq:per0}.

Let $x_{0}$ be such that $g(x_{0}) = 1/2$.  We define the left and right reflections $\theta^{l}_{x_{0}}$ and $\theta^{r}_{x_{0}}$ in the following way
\begin{equation}
	\label{def:lr_odd_reflections}
	\begin{split}
		\theta^{l}_{x_{0}} g(x) &= \begin{cases}
			g(x) &\text{if $x \leq x_0$}\\
			1-g(2x_{0}- x) &\text{otherwise}
		\end{cases} \qquad\text{and}\qquad
		\bar\theta^{l}_{x_0}\gamma(x) = \begin{cases}
			\gamma(x)  &\text{if }x\leq x_{0}\\
			\gamma(2x_{0}- x)  &\text{otherwise}
		\end{cases}\\
		\theta^{r}_{x_{0}} g(x) &= \begin{cases}
			g(x) &\text{if $x \geq x_0$}\\
			1-g(2x_{0}- x) &\text{otherwise}
		\end{cases} \qquad\text{and}\qquad
		\bar\theta^{r}_{x_0}\gamma(x) = \begin{cases}
			\gamma(x)  &\text{if }x\geq x_{0}\\
			\gamma(2x_{0}- x)  &\text{otherwise.}
		\end{cases}
	\end{split}
\end{equation}

The main property of a reflection positive kernel is that either the right or the left reflection does not increase the energy of the nonlocal part. Namely
\begin{align}
		\int_{-NL}^{NL}\int_{-NL}^{NL} \Bigl(g(x)-\frac12\Bigr)&\Bigl(g(y)-\frac12\Bigr)\widehat K_\tau(x-y)\dx\dy \geq\notag\\
		&\geq
		\frac{1}{2} \int_{-NL}^{2 x_{0}+NL}\int_{-NL}^{2x_0+NL} \Bigl(\theta^{l}_{x_{0}} g(x)-\frac12\Bigr) \Bigl(\theta^{l}_{x_{0}} g(y)-\frac12\Bigr) \widehat K_\tau(x-y) \dx\dy\notag\\
		& +\frac{1}{2} \int_{2x_{0}-NL}^{NL}\int_{2x_0-NL}^{NL} \Bigl(\theta^{r}_{x_{0}} g(x)-\frac12\Bigr) \Bigl(\theta^{r}_{x_{0}} g(y)-\frac12\Bigr)\widehat  K_\tau(x-y)\dx\dy.	\label{eq:basic_rp_inequality_nl}
\end{align}

For the local part it is not difficult to see that, due to the locality of the function $A(\gamma,g,x)$ and  the symmetry w.r.t. $1/2$ of the double well potential, it holds
\begin{equation}
	\int_{-NL}^{NL}A(\gamma,g,x)\dx=\frac12\int_{-NL}^{2x_0+NL} A(\bar \theta ^l_{x_0}\gamma, \theta^l_{x_0}g, x)\dx+\frac12\int_{2x_0-NL}^{NL} A(\bar \theta ^r_{x_0}\gamma, \theta^r_{x_0}g, x)\dx.
\end{equation}

Suppose now that $g = (g_{1},g_{2},\ldots, g_{m})$, $\gamma=(\gamma_1,\ldots,\gamma_m)$. With the above notation we mean that there exist $x_{1}<\dots< x_{m+1}$ such that $g(x_{i})=1/2$ and for every $x\in (x_{k},x_{k+1})$ either $g(x)\geq 1/2$ or $g(x)\leq 1/2$ (and $g_k=g_{|_{[x_k,x_{k+1}]}}$, $\gamma_k=\gamma_{|_{[x_k,x_{k+1}]}}$). Then, the following \emph{chessboard estimate} holds
\begin{equation}\label{eq:chess}
	\bar \Fcal^{1d}_{\tau,[x_1,x_{m+1}],\eps}(\gamma,g)\geq\sum_{k=1}^m|x_{k+1}-x_k|\lim_{N\to\infty}\bar \Fcal^{1d}_{\tau,[-NL,NL],\eps}(\varphi[g_k], \varphi[\gamma_k]).
\end{equation}
The proof of \eqref{eq:chess} can be obtained by induction directly from \eqref{eq:basic_rp_inequality_nl} and the analogous inequality (which in that case is indeed an equality) for the local term in \eqref{eq:NL} for the reflections of the functions $g_k,\gamma_k$. 

Once \eqref{eq:barctaueps} is given, the proof of \eqref{eq:per0} reduces to show that 
\begin{equation}
	\liminf_{N\to\infty}\inf_{g\in W^{1,2}([-NL,NL]),g(-NL)=g(NL)\geq\frac12,\gamma\geq1}\bar \Fcal^{1d}_{\tau,[-NL,NL],\eps}(\gamma,g)\geq \inf_h\bar C_{\tau,\eps}(h).
\end{equation}

For this purpose, the chessboard estimate \eqref{eq:chess} is the fundamental ingredient. Indeed, given $N$ and $L$, let $g$ be any function in $W^{1,2}([-NL,NL])$ with $g(-NL)=g(NL)=\frac12$ and let us denote by $x_1=-NL<\dots<x_m=NL$ the points such that $g(x_k)=\frac12$ (if $g$ is identically equal to $\frac12$ one an interval $[a,b]$, let $x_k=a$ and $x_{k+1}=b$). Let  $g_k=g_{|_{[x_k,x_{k+1}]}}$, $\gamma_k=\gamma_{|_{[x_k,x_{k+1}]}}$. By construction, either $g_k\geq\frac12$ or $g_k\leq\frac12$. By the chessboard estimate \eqref{eq:chess}
	\begin{equation}\label{eq:chess1}
		\bar \Fcal^{1d}_{\tau,[-NL,NL],\eps}(\gamma,g)\geq\sum_{k=1}^m|x_{k+1}-x_k|\lim_{\bar N\to\infty}\bar \Fcal^{1d}_{\tau,[-\bar NL,\bar NL],\eps}(\varphi[g_k], \varphi[\gamma_k]).
	\end{equation}
Then notice that, for every $k$,
\begin{equation}
	\lim_{\bar N\to\infty}\bar \Fcal^{1d}_{\tau,[-\bar NL,\bar NL],\eps}(\varphi[g_k], \varphi[\gamma_k])\geq \inf_h\bar C_{\tau,\eps}(h),
\end{equation}  
since $g_k\in\mathcal C_{|x_{k+1}-x_k|}$ and $\gamma_k\in\Gamma_{|x_{k+1}-x_k|}$.

By \eqref{eq:per0},  $\bar C^*_{\tau,\eps}$ is attained on functions $g$ of the following type: either always bigger or always smaller than $1/2$ or periodic of some finite period $2h^*_{\tau,\eps}$ obtained reflecting functions $g\in\mathcal C_{h^*_{\tau,\eps}}$,
namely of the form $\varphi_{h^*_{\tau,\eps}}[g]$. Moreover, $\gamma\in\Gamma_{h^*_{\tau,\eps}}$. In particular, if $L=2kh^*_{\tau,\eps}$ for some $k\in\N$ then minimizers $g$ of $\Fcal^{1d}_{\tau,L,\eps}$ are either always bigger or always smaller than $1/2$ or periodic of period $2h^*_{\tau,\eps}$ of the form $\varphi_{h^*_{\tau,\eps}}[g]$.
In principle, the admissible periods $2h^*_{\tau,\eps}$ for such minimizers might not be unique.

If $\tau$ is sufficiently small, in our case we are indeed able to exclude the first scenario (namely non-existence of a finite period with minimizers $g$ always above or below $1/2$). More precisely, we have  the following 

\begin{proposition}\label{lemma:period}
	If $\tau$ is sufficiently small, a function $g$ satisfying 
	\[
	g\geq\frac12\quad\text{or}\quad g\leq\frac12
	\]
	cannot be a minimizer in~\eqref{eq:barctaueps}.
	
	In particular,  there exists $\bar{\eps}$ and $\bar{\tau}>0$ such that for any $0<\eps\leq\bar{\eps}$, $0<\tau\leq\bar{\tau}$, $k\in\N$ and $L=2kh^*_{\tau,\eps}$ then minimizers $(\gamma,g)$ of $\Fcal^{1d}_{\tau,L, \varepsilon}$  among $L$-periodic functions are such that: $g$ is periodic of period $2h^*_{\tau,\eps}$ and of the form $\varphi[g_{h^*_{\tau,\eps}}]$ for some $g_{h^*_{\tau,\eps}}\in\mathcal C_{h^*_{\tau,\eps}}$, while $\gamma\in\Gamma_{h^*_{\tau,\eps}}$.
\end{proposition}

\begin{proof} Assume $g\in W^{1,2}_{\mathrm{loc}}(\R;[0,1])$ $L$-periodic satisfies $g\geq1/2$. Hence, for all $x,y\in \R$ it holds  $|g(x)-g(y)|\leq1/2$ and by~\eqref{eq:omegaab}
	\begin{equation}
      \label{eq:ineqomega}
	\frac12[\omega(g(x))-\omega(g(y))]\geq|g(x)-g(y)|^2.
	\end{equation}
Now observe that
  \begin{equation}\label{eq:partpos2}
|g'(s)|^2\gamma(s)+\frac{W(g(s))}{\gamma(s)}\geq2|g'(s)|\sqrt{W(g(s))}=|(\omega\circ g)'|(s).
\end{equation}
	Therefore, as in Step 3 of Proposition~\ref{prop:stability}
	\begin{align}
	\Fcalo(\gamma,g)&=\frac3L\Bigl(\frac12C_\tau-1\Bigr)\int_{0}^{L} \eps\tau^{1/\beta}\gamma(x)|g'(x)|^{2} +\frac{W(g(x))}{\eps\tau^{1/\beta}\gamma(x)}\dx\label{eq:term1}\\
	&+\frac{3}{2L}C_\tau\int_{0}^{L} \eps\tau^{1/\beta}\gamma(x)|g'(x)|^{2} +\frac{W(g(x))}{\eps\tau^{1/\beta}\gamma(x)}\dx-\frac1L\int_{\R}\int_0^L|g(x)-g(x+z)|^2\widehat K_\tau(z)\dx\dz\label{eq:term2bis1}
	\end{align}
	where~\eqref{eq:term1} is positive if $\tau>0$ is small (since $C_\tau\sim\frac{1}{\tau}$) and~\eqref{eq:term2bis1} is positive by~\eqref{eq:partpos2} and~\eqref{eq:ineqomega} as in \eqref{eq:partpos0}.

	\end{proof}


\subsection{Minimal coefficients}
\label{subsec:micoef}

Let us now consider the functional introduced in \eqref{eq:functgamma0}. Namely,
\begin{align}
  \label{eq:f1diml}
 \Fcalt^{1d}(\gamma, g) := \frac{3(C_\tau- 1)}{L} \int_{0}^{L} \Big(\varepsilon\tau^{1/\beta} \gamma(x) |g'(x)|^{2} + \frac{W(g(x))}{\gamma(x)\varepsilon\tau^{1/\beta}} \Big ) \dx\notag\\ - \frac1L\int_{0}^{L}\int_{\R} (g(x) - g(y))^2\Khattau(x-y)\dy\dx,
\end{align}
where $\gamma:\R\to[1,+\infty]$ is a measurable $L$-periodic function and $g\in W^{1,2}_{\mathrm{loc}}(\R;[0,1])$ is an $L$-periodic function such that $\gamma |g'|^2\in L^1_{\mathrm{loc}}(\R)$.
  Thus we have that  $\Fcalt^{1d}(1,g)=\Fcalt^{1d}(g)$.

{As proved in Section \ref{ssec:1d}, whenever $L=2kh^*_{\tau,\eps}$, for some $k\in\N$ and for some admissible optimal period $h^*_{\tau,\eps}$, minimizers $(\gamma,g)$  are periodic of period $2h^*_{\tau,\eps}$.

Moreover, the following holds: for every $k\in\N\cup \{0\}$, $x\in[0,h^*_{\tau,\eps}]$
}
\begin{align}
\gamma(x+kh^*_{\tau,\eps})&=\gamma((k+1)h^*_{\tau,\eps}-x),\label{eq:gammarefl}\\
g((2k+1)h^*_{\tau,\eps}+x)&=1-g((2k+1)h^*_{\tau,\eps}-x).\label{eq:grefl}
\end{align} 
In particular, $g$ is of the form $\varphi[g]$ for some profile $g\in\mathcal C_{h^*_{\tau,\eps}}$ and $\gamma\in\Gamma_{h^*_{\tau,\eps}}$.

From now onwards we fix one of the optimal periods $2h^*_{\tau,\eps}$ and we set for  simplicity of notation $h=h^*_{\tau,\eps}$.
In order to study minimizers of~\eqref{eq:f1diml} we can thus reduce to study the following functional

\begin{equation} 
\label{eq:onedimmin}
\begin{split}
F_{\tau,\eps}(\gamma, g) := \frac{3(C_\tau- 1)}{2h}  \int_{0}^{2h}\Big(\varepsilon\tau^{1/\beta} \gamma(x) |g'(x)|^{2} + \frac{W(g(x))}{\gamma(x)\varepsilon\tau^{1/\beta}} \Big) \dx\\ - \frac{1}{2h}\int_{0}^{2h}\int_{\R} (g(x) - \varphi[g](y))^2\Khattau(x-y)\dy\dx,
\end{split}
\end{equation}

  where $\gamma:\R\to[1,+\infty]$ is a $2h$-periodic measurable function such that $\gamma\in\Gamma_h$ as in~\eqref{eq:gammarefl}  and $g\in \mathcal C_h$ is such that $\gamma |g'|^2\in L^1([0,h])$. In particular, $F_{\tau,\eps}(\gamma,g)=\Fcal^{1d}_{\tau,2h,\eps}(\gamma,g)$. {We omit the explicit dependency of $F_{\tau,\eps}$ on $h$ since $h=h^*_{\tau,\eps}$ will be fixed for the rest of this section.   }

The function $\varphi[g]$ is defined through reflection of $g$ as in~\eqref{eq:grefl}. From now onwards we will set for simplicity of notation $\varphi[g]=g$, being clear that in this section we will consider reflections of functions in $\mathcal C_h$.

The aim of this section is to prove the following 
\begin{theorem}\label{thm:onedimmin}
  Let $F_{\tau,\eps}=F_{\tau,\eps}(\gamma,g)$ be the functional in~\eqref{eq:onedimmin}.
  Then there exist $\eps_2>0$, $\tau_2>0$ such that whenever $0<\eps\leq\eps_2$ and $0<\tau\leq\tau_2$ it holds $\inf_{g\in \mathcal C_h}F_{\tau,\eps}(\gamma,g) \geq\inf_{g\in \mathcal C_h}F_{\tau,\eps}(1,g)$ for all $\gamma\geq1$ and $g$ as above.
\end{theorem}

In particular, $\bar C^*_{\tau,\eps}=C^*_{\tau,\eps}$, where $C^*_{\tau,\eps}$ was defined in \eqref{eq:cstartau} and $\bar C^*_{\tau,\eps}$ was defined in \eqref{eq:barctaueps}.

The proof will look for a contradiction in the coexistence of an Euler-Lagrange equation for the minimizer $g$ of $F_{\tau,\eps}(\gamma,\cdot)$ among functions in $\mathcal C_h$ (i.e., \eqref{eq:g1}) and an Euler-Lagrange equation for the optimal coefficient $\gamma$ (i.e., \eqref{eq:gamma1}-\eqref{eq:gamma3}) unless $\gamma\equiv1$ on $\{g\neq1\}$. However, due to the fact that a minimizing $\gamma$ can a priori take the value $+\infty$ on a set of positive measure, one cannot  immediately derive Euler-Lagrange equations for $F_{\tau,\eps}(\gamma,\cdot)$ in $g$. Thus, we introduce a family of auxiliary functionals which are finite only on $L^2$ integrable coefficients $\gamma$ and for which we can derive Euler-Lagrange equations.
The minimizers of such functionals will converge to a minimizer of the original functional $F_{\tau,\eps}$, which will satisfy Euler-Lagrange equations as a consequence  of a limiting procedure (Proposition~\ref{prop:gammalim}).
Finally, in Theorem~\ref{thm:pos}, we will prove that such a minimizer satisfies $\gamma=1$ a.e..
For the proof of Theorem~\ref{thm:pos} we will need Proposition~\ref{prop:gammalim}, Lemma~\ref{lemma:I},  Lemma~\ref{lemma:gammaconvx}, Lemma~\ref{lemma:1-ctau} and their Corollaries.

Let us then introduce the following auxiliary functional: for any $\gamma\geq1$, $g\in\mathcal C_h$, $\gamma |g'|^2\in L^1$, $m\geq 1$, let

\begin{align}
\hat F_m(\gamma,g)&=F_{\tau,\eps}(\gamma,g)+\frac{1}{4h}\int_0^{2h}(\gamma(x)-m)^2_+\dx,\label{eq:fm}
\end{align} 



Let moreover 
\begin{equation}
   \label{eq:defI_g}
    I_g(x)=\int_{\R}(g(y)-g(x))\widehat K_\tau(y-x)\dy.
\end{equation}

Notice that	
\begin{equation}\label{eq:0l}
\begin{split}
\int_{0}^{+\infty}\Khattau(z)\dz = 
\frac{\tilde{C_1}}{\tau \tau^{1/\beta}}\qquad\text{and}\qquad C_{\tau} =  \frac{\tilde{C}_{2}}{\tau},
\end{split}
\end{equation}
for some $\tilde{C_1}, \tilde{C}_2 > 0$, where $C_\tau$ has been defined in~\eqref{eq:ctau}.

We will prove the following
\begin{proposition}
   \label{prop:gammalim}
	Let $(\gamma_m,g_m)$ be minimizers of $\hat F_m$. Then, there exist $(\gamma,g)$ such that $\gamma\geq1$, $\gamma g'\in L^2$, $g\in W^{1,2}$ and, up to subsequences, the following  holds
	\begin{align}
	F_{\tau,\eps}(\gamma,g)&\leq\lim_{m\to\infty}\hat F_m(\gamma_m,g_m)\label{eq:conv1}\\
	g_m&\rightharpoonup g \text{ in $W^{1,2}$}\label{eq:conv2}\\
	\frac{1}{\gamma_m}&\rightharpoonup \frac1\gamma\text{ in $L^\infty$}\label{eq:conv3}\\
	\gamma_mg'_m&\rightarrow \gamma g' \text{in $L^2(\{g\neq1\})$. }\label{eq:conv4}
	\end{align}
	The pair  $(\gamma,g)$ is a minimizer of $F_{\tau,\eps}$ and satisfies
		\begin{align}
			3(C_\tau-1)\eps\tau^{1/\beta}(\gamma g')'&\geq3(C_\tau-1)\frac{W'(g)}{2\gamma\eps\tau^{1/\beta}}+2{I_g}\label{eq:g1ineq}
	\end{align}
    where equality holds when $g\neq1$, namely
		\begin{align}
	3(C_\tau-1)\eps\tau^{1/\beta}(\gamma g')'&=3(C_\tau-1)\frac{W'(g)}{2\gamma\eps\tau^{1/\beta}}+2{I_g}\quad\text{on $\{g\neq1\}$}\label{eq:g1}
	\end{align}
    and
		\begin{align}
	3(C_\tau-1)\Bigl[\eps\tau^{1/\beta}\gamma^2|g'|^2-\frac{W(g)}{\eps\tau^{1/\beta}}\Bigr]\Big|_{\xi_1}^{\xi_2}&=4\int_{\xi_1}^{\xi_2}\gamma(x)g'(x)I_g(x)\dx\quad\forall\,\xi_1<\xi_2\text{ s.t. }(\xi_1,\xi_2)\subset\{g\neq1\}.\label{eq:g2}
	\end{align}
Moreover, the following holds
	\begin{align}
	\eps\tau^{1/\beta}|g'(x)|^2&\geq \frac{W(g(x))}{\eps\tau^{1/\beta}} \quad\text{ if $\gamma(x)=1$}\label{eq:gamma1}\\
	\eps\tau^{1/\beta}\gamma(x)^2|g'(x)|^2&=\frac{W(g(x))}{\eps\tau^{1/\beta}}\quad \text{ if $1<\gamma(x)<+\infty$}\label{eq:gamma2}\\
	\gamma(x) g'(x)&=0 \quad\text{ if $\gamma(x)=+\infty$ and $g(x)\neq1$}.\label{eq:gamma3}
	\end{align}
\end{proposition}

Theorem~\ref{thm:onedimmin}  follows then immediately  from Proposition~\ref{prop:gammalim} and the following

\begin{theorem}\label{thm:pos}
	The functions $(\gamma,g)$ as in Proposition~\ref{prop:gammalim} satisfy
	\begin{equation}\label{eq:ineqstrict}
	\eps\tau^{1/\beta}\gamma^2(x)|g'(x)|^2>\frac{W(g(x))}{\eps\tau^{1/\beta}}\quad \forall\,x\text{ s.t. $g(x)\neq1$}.
	\end{equation}
	In particular, $\gamma= 1$ a.e..
\end{theorem}

Indeed, the fact that $\gamma=1$ a.e. follows from \eqref{eq:ineqstrict} in the following way: by~\eqref{eq:gamma1} and~\eqref{eq:gamma2}, whenever the strict inequality~\eqref{eq:ineqstrict} holds then $\gamma=1$. We will be able to show that the set $\{g=1\}$ is an interval (Corollary \ref{cor:mon}), thus $g'=0$ a.e. on the set $\{g=1\}$. In particular, one can choose arbitrarily the value of $\gamma$ (e.g. $\gamma=1$) on $\{g=1\}$ without changing the value of the functional.

To prove Theorem~\ref{thm:pos} we adopt the following strategy.
By~\eqref{eq:g2} with $\xi_2\in\partial\{g\neq1\}$  and $\xi_1=x\in[0,\xi_2)$ such that $(x,\xi_2)\subset \{g\neq 1\}$, one gets (since $W(g(\xi_2))=W(1)=0$)
\begin{align}
\eps\tau^{1/\beta}\gamma^2(x)|g'(x)|^2-\frac{W(g(x))}{\eps\tau^{1/\beta}}&= \eps\tau^{1/\beta}\gamma^2(\xi_2)|g'(\xi_2)|^2-\frac{4}{3(C_\tau-1)}\int_{x}^{\xi_2}\gamma(z)g'(z)I_g(z)\dz\notag\\
&\geq-\frac{4}{3(C_\tau-1)}\int_{x}^{\xi_2}\gamma(z)g'(z)I_g(z)\dz.
\end{align}

Thus, $\eps\tau^{1/\beta}\gamma^2(x)|g'(x)|^2-\frac{W(g(x))}{\eps\tau^{1/\beta}}>0$ whenever e.g. $g'>0$ and $I_g<0$ on the interval $[x,\xi_2]$ (respectively  whenever $g'<0$ and $I_g<0$ on $[\xi_2,x]$ if $x\in(\xi_2,h]$).

Concerning the sign of $g'$, from Proposition~\ref{prop:gammalim} one has the following

\begin{corollary}
  \label{cor:mon}
	The functions $(\gamma,g)$ in Proposition~\ref{prop:gammalim} satisfy the following conditions: $\gamma<+\infty$ and there exist $\bar x_1\leq\bar x_2\in [0,h]$ s.t.    $g$ is strictly monotone increasing on $[0,\bar x_1]$, $g=1$ on $[\bar x_1,\bar x_2]$  and  $g$ is strictly monotone decreasing on $[\bar x_2,h]$.
\end{corollary}
{
  In particular, it is sufficient to prove~\eqref{eq:ineqstrict} on $[0,\bar x_1)$, being $\bar x_1$ the first point larger than $0$ in which the minimizer $g$ attains the value $1$. Indeed, notice that to prove the inequality \eqref{eq:ineqstrict} on $(\bar x_2,h]$ for the function $g$ is equivalent to prove the same inequality on the interval $[0, h-x_2)$ for the function $\bar g(x)=g(h-x)$, which is also a minimizer of $F_{\tau,\eps}(\gamma,\cdot)$.}

  {By Corollary~\ref{cor:mon}, $g'>0$ on $[0,\bar x_1)$.
  Thus, by \eqref{eq:g2}, for any $x\in[0,\bar x_1)$,  $\eps\tau^{1/\beta}\gamma^2(x)|g'(x)|^2-\frac{W(g(x))}{\eps\tau^{1/\beta}}>0$ whenever  $I_g<0$ on $[x,\bar x_1]$.
  In the following lemmas we will prove the following:
  \begin{itemize}
  	\item  For any $0<\delta\ll1$, let $x_\delta$ be the point such that $g(x_\delta)=\frac12+\delta$. Then, provided  $\eps$ and $\tau$ are sufficiently small, $I_g(x)<0$ for all $x\in[x_\delta,\bar x_1]$. In particular, by \eqref{eq:g2}, \eqref{eq:ineqstrict} is satisfied on $[x_\delta,\bar x_1]$. 
  	\item For any $x\in[0,x_\delta]$, 
  	\[
  	\int_{x}^{x_{\delta}}\Bigl(\gamma(z)g'(z)I_g(z)\Bigr)_+\dz\ll\int_{x_\delta}^{\bar x_1}\Bigl(\gamma(z)g'(z)I_g(z)\Bigr)_-\dz.
  	\]
  Thus, by \eqref{eq:g2}, the equality $\eps\tau^{1/\beta}\gamma^2(x)|g'(x)|^2=\frac{W(g(x))}{\eps\tau^{1/\beta}}$ cannot be reached also in the remaining  interval $[0,x_\delta]$ and \eqref{eq:ineqstrict} is satisfied on the whole $[0,\bar x_1]$. 
\end{itemize}
}
Moreover, from Corollary \ref{cor:mon} one has that the set $\{g=1\}$ is an interval, thus $g'=0$ in its interior.

Before stating and proving all the preliminary lemmas to the proof of Theorem \ref{thm:pos}, we  give a proof of Proposition~\ref{prop:gammalim} and Corollary~\ref{cor:mon}.
\begin{proof}
	[Proof of Proposition~\ref{prop:gammalim}:]
	
	Since $(\gamma_m,g_m)$ are minimizers of $\hat F_m$ and since $\hat F_m(\bar\gamma,\bar g)=F_{\tau,\eps}(\bar{\gamma},\bar g)$ whenever $\bar{\gamma}\leq m$, one has that
	\[
	\sup_m\hat F_m(\gamma_m,g_m)\leq\sup_m\hat F_m(1,g_{1,\tau,\eps})=F_{\tau,\eps}(1,g_{1,\tau,\eps})\leq C<0,
	\]
	where $g_{1,\tau,\eps}$ is a minimizer of $F_{\tau,\eps}(1,\cdot)$. 
	Moreover, since $\gamma_m\geq1$,
	\[ \sup_m\int_0^h|g_m'|^2\dx\leq\sup_m\int_0^h\gamma_m|g_m'|^2\dx<+\infty.
	\] This implies the convergence in~\eqref{eq:conv2}, up to subsequences.
    In particular, $g_m$ converges to $g$ in $C^0$.
	Since $\frac{1}{\gamma_m}\leq1$, up to subsequences $\frac{1}{\gamma_m}$ converges weakly in $L^\infty$ to $\frac{1}{\gamma}$, thus giving~\eqref{eq:conv3}.
    {
      Being $g_m$ a minimizer of $\hat F_m(\gamma_m,\cdot)$, and being $\gamma_m\in L^2$, one can make variations of $\hat F_m(\gamma_m,\cdot)$ around $g_m$ of the form $g_m+t\psi$ with $\psi\leq0$ where $g_m=1$ and obtain the following inequality
    }
    \begin{align}
        3(C_\tau-1)\eps\tau^{1/\beta}(\gamma_m g_m')'&\geq3(C_\tau-1)\frac{W'(g_m)}{2\gamma_m\eps\tau^{1/\beta}}+{2I_{g_m}}.\label{eq:g1mineq}
        \end{align}
    Moreover, making variations of $\hat F_m(\gamma_m,\cdot)$ around $g_m$ of the form $g_m+t\psi$ where the sign of $\psi$ is arbitrary and $\mathrm{supp}\psi\subset\subset\{g_m\neq1\}$ one obtains 
    \begin{align}
         3(C_\tau-1)\eps\tau^{1/\beta}(\gamma_m g_m')'\chi_{\{g_m\neq1\}}&=\Bigl(3(C_\tau-1)\frac{W'(g_m)}{2\gamma_m\eps\tau^{1/\beta}}+{2I_{g_m}}\Bigr)\chi_{\{g_m\neq1\}}.\label{eq:g1m}
      \end{align}
  
      Since $g_m\to g$ uniformly, $\forall\,\delta>0$ there exists $\bar m$ such that, for all $m\geq\bar m$, one has that $\{g<1-\delta\}\subset\{g_m<1-\delta/2\}$.
      Hence, by \eqref{eq:conv2} and \eqref{eq:conv3}
  \begin{align}
  	\Bigl(3(C_\tau-1)\frac{W'(g_m)}{2\gamma_m\eps\tau^{1/\beta}}+{2I_{g_m}}\Bigr)\chi_{\{g<1-\delta\}}\rightharpoonup\Bigl(3(C_\tau-1)\frac{W'(g)}{2\gamma\eps\tau^{1/\beta}}+{2I_{g}}\Bigr)\chi_{\{g<1-\delta\}}\quad\text{ in $L^2$}.\label{eq:6.38}
  \end{align} 
    Hence, $(\gamma_mg'_m)'\rightharpoonup T'$ in  $L^2(\{g<1-\delta\})$ and $\gamma_mg'_m\to T$ in $L^2(\{g<1-\delta\})$. Being $\delta$ arbitrary and the r.h.s. of  \eqref{eq:6.38} bounded in $L^2$ as $\delta\to0$, one has that $\gamma_mg'_m\to T$ in $L^2(\{g\neq1\})$.

	Thus one has that (first by $\eqref{eq:conv2}$ and then by~\eqref{eq:conv3} and the convergence above)
	\begin{align}
	\frac{1}{\gamma_m}(\gamma_mg_m')=g_m'\rightharpoonup g'\text{ in $L^2$}\notag\\
	\frac{1}{\gamma_m}(\gamma_mg_m')\rightharpoonup\frac{T}{\gamma}\text{ in $L^2(\{g\neq1\})$}.
\end{align}
Hence $T=\gamma g'$ on $\{g\neq1\}$ and~\eqref{eq:conv4} holds.
Notice that, multiplying~\eqref{eq:g1m} by $\gamma_mg_m'$ on the set where $g_m\neq1$, and integrating, one obtains
\begin{align}
3(C_\tau-1)\Bigl[\eps\tau^{1/\beta}\gamma_m^2|g_m'|^2-\frac{W(g_m)}{\eps\tau^{1/\beta}}\Bigr]\Big|_{\xi_1}^{\xi_2}&={2}\int_{\xi_1}^{\xi_2}\gamma_m(x)g_m'(x)I_{g_m}(x)\dx\quad\forall\,\xi_1<\xi_2\text{ s.t. }(\xi_1,\xi_2)\subset\{g_m\neq1\}.\label{eq:g2m}
\end{align}
By the convergences~\eqref{eq:conv2}-\eqref{eq:conv4}, both~\eqref{eq:g1m} and~\eqref{eq:g2m} pass to the limit, giving~\eqref{eq:g1} and~\eqref{eq:g2}. Moreover, also~\eqref{eq:conv1} clearly holds.

Let $(\gamma,g)$ be obtained as above.
We want to prove that $(\gamma,g)$ is a minimizer of $F_{\tau,\eps}$. 
	To this aim, for any measurable $\bar\gamma\geq1$ let us define the truncation
	\begin{equation}\label{eq:bargammam}
	\bar\gamma_m(x)=\left\{\begin{aligned}
	&\bar \gamma(x) && &\text{if $\bar\gamma(x)\leq m$}\\
	&m && &\text{if $\bar\gamma(x)>m$}. 
	\end{aligned}\right.
	\end{equation}
	Then, notice that 
	\begin{align}
	\hat F_m(\bar \gamma_m,\bar g)&\leq F_{\tau,\eps}(\bar \gamma,\bar g)+\frac{3(C_\tau-1)}{2h}\int_{ \{\bar \gamma>m\}}\frac{W(\bar g)}{\eps\tau^{1/\beta}m}\notag\\
	&\leq F_{\tau,\eps}(\bar \gamma,\bar g)+\frac{C(h,\eps,\tau)}{m}.
	\end{align}
	Hence, for any $\bar{\gamma},\bar g$ and letting $\bar{\gamma}_m$ as in \eqref{eq:bargammam}
	\begin{align}
	F_{\tau,\eps}(\gamma,g)\leq\lim_{m\to\infty}\hat F_m(\gamma_m,g_m)\leq \limsup_{m\to\infty} \hat F_{m}(\bar{\gamma}_m,\bar g)\leq\limsup_{m\to\infty} \Bigl[F_{\tau,\eps}(\bar \gamma, \bar g) + \frac{C(h,\eps,\tau)}{m}\Bigr]=F_{\tau,\eps}(\bar \gamma, \bar g).
	\end{align}
	{Conditions~\eqref{eq:gamma1} and~\eqref{eq:gamma2} hold since $\gamma$ is a minimizer of $F_{\tau,\eps}(\cdot,g)$ and are obtained performing suitable variations of the form $\gamma+t\gamma_0$  where  $\gamma_0\geq0$ whenever $\gamma=1$.  }
    Being by \eqref{eq:g1} $\gamma g'\in W^{1,2}(\{g\neq1\})$ and thus continuous on the set $\{g\neq1\}$, equation~\eqref{eq:gamma3} follows.
	 
\end{proof}

\begin{proof}[Proof of Corollary~\ref{cor:mon}:]
  {
    Since $\gamma g'$ is continuous on the set $\{g\neq1\}$ and $\gamma\geq1$, we have that at a local minimum/maximum of $g$ in $[0,h]$ with $g\neq1$ one has that $g'=0$. From~\eqref{eq:gamma1}-\eqref{eq:gamma3}, at such a point one has that  $\gamma=+\infty$.
  }
   Indeed, if $\gamma(x)<+\infty$ and $g(x)\neq1$, then $0=\gamma(x)^2|g'(x)|^2\geq W(g(x))/\eps^2\tau^{2/\beta}>0$ which leads to a contradiction.  
   
   Assume then that  there exist points $x\in[0,h]$ such that $g(x)<1$ and $\gamma(x)=+\infty$.
  Hence by~\eqref{eq:gamma3}, at such points 
	\begin{equation}
	\eps\tau^{1/\beta}\gamma^2(x)|g'(x)|^2-\frac{W(g(x))}{\eps\tau^{1/\beta}}=-\frac{W(g(x))
}{\eps\tau^{1/\beta}}<0.	\end{equation}
On the other hand, when $\gamma(x)<+\infty$~\eqref{eq:gamma2} or~\eqref{eq:gamma1} hold and then $\eps\tau^{1/\beta}\gamma^2(x)|g'(x)|^2-\frac{W(g(x))}{\eps\tau^{1/\beta}}\geq0$.  Let $x_0$ be such that $\gamma(x_0)=+\infty$ and $g(x_0)<1$. Then, there exists a neighbourhood $[x_0-\delta,x_0+\delta]$ on which $g<1$ and  the continuous function $\eps\tau^{1/\beta}\gamma^2|g'|^2-\frac{W(g)}{\eps\tau^{1/\beta}}$ is negative. Hence, $\gamma=+\infty$ and $g'=0$ on $[x_0-\delta,x_0+\delta]$, thus the function $\eps\tau^{1/\beta}\gamma^2|g'|^2-\frac{W(g)}{\eps\tau^{1/\beta}}$ is constant on that interval. Applying the same reasoning to the points $x_0\pm\delta$ and interating the procedure, one would have that $\gamma\equiv+\infty$ and $g\equiv g(x_0)$. However, this possibility  is excluded by the fact that in this case $F_{\tau,\eps}(\gamma,g)=0$ and thus $(\gamma,g)$ would  not be optimal since the minimal value of $F_{\tau,\eps}$ must be negative.
Thus, $\gamma<+\infty$ and $g'\neq0$ when $g\neq1$. Thus, the function $g$ attains its maximum value $g=1$ on an interval, which we denote by $[\bar x_1,\bar x_2]$, and is monotone increasing on $[0,\bar x_1]$, monotone decreasing on $[\bar x_2,h]$.

\end{proof}

{Let us define the class of functions
\begin{align}
\mathcal S_h=\Bigl\{g\in BV\Bigl([0,h];\Bigl[\frac12,1\Bigr]\Bigr):\,&g(0)=g(h)=\frac12,\notag\\
&\exists\,\bar x \text{ s.t. }g(\bar x)=1\text{ and $Dg\geq0$ on $[0,\bar x]$, $Dg\leq0$ on  $[\bar x,h]$} \Bigr\}.\label{eq:scal}
\end{align}}

By Corollary~\ref{cor:mon} the minimizers of $F_{\tau,\eps}$ found in Proposition~\ref{prop:gammalim} belong to this class.
{Some of the following lemmas will hold not only for the minimizing pair $(\gamma,g)$ of Proposition~\ref{prop:gammalim} but more in general for functions in $\mathcal S_h$.}

\begin{lemma}
  \label{lemma:I}
	There exists $\bar c_0<0$ and $\bar c_1=\bar c_1(h,\bar c_0)$ such that whenever $g\in\mathcal S_h$ then \begin{equation}\label{eq:stimagi}
	g(x)\geq1-\bar c_1\tau\tau^{1/\beta}\quad\Rightarrow\quad I_g(x)<\bar c_0.
\end{equation}
{Moreover, for $x\in[0,\bar x]$ with $\bar x$ as in \eqref{eq:scal} one has the following estimate}
\begin{align}\label{eq:stimai}
I_g(x)\leq (1-g(x))\int_0^{2x}\widehat K_\tau(z)\dz+(1-2g(x))\int_x^{+\infty}\widehat K_\tau(z)\dz. 
\end{align}
\end{lemma}

\begin{proof}
	By~\eqref{eq:0l} one has that
	\begin{align}
	I_g(x)&=\int_{\R}(g(y)-g(x))\widehat K_\tau(y-x)\dy=	\int_{\R}g(y)\widehat K_\tau(y-x)\dy-g(x)\int_{\R}\widehat K_\tau(z)\dz\notag\\
	&\leq\int_{ \R}\varphi[\chi_{[0,h]}](y)\widehat K_\tau(y-x)\dy-g(x)\int_{\R}\widehat K_\tau(z)\dz\notag\\
	&\leq\int_{ \R}\varphi[\chi_{[0,h]}](y)\widehat K_\tau(h/2-y)\dy-g(x)\int_{\R}\widehat K_\tau(z)\dz\notag\\
	&\leq2\Bigl[\int_0^{h/2}\widehat K_{\tau}(z)\dz+\int_{3h/2}^{+\infty}\widehat K_\tau(z)\dz-g(x)\int_0^{+\infty}\widehat K_\tau(z)\dz\Bigr]\notag\\
	&\leq-2\int_{h/2}^{3h/2}\widehat K_\tau(z)\dz+2\bar c_1\tau\tau^{1/\beta}\int_{0}^{+\infty}\widehat K_\tau(z)\dz\notag\\
	&\leq-2\int_{h/2}^{3h/2}\widehat K_\tau(z)\dz+2\bar c_1\tilde C_1,
	\end{align}
	hence~\eqref{eq:stimagi} holds provided $\bar c_1$ is sufficiently small.
	
	As for~\eqref{eq:stimai}, we first obtain the following decomposition:  given that $g(-y) = 1- g(y)$, we have that
	\begin{align}
	\int_{-\infty}^{+\infty}(g(y) - g(x)) \widehat{K}_{\tau}(y-x) \dy &= 
	\int_{0}^{+\infty}(g(y) - g(x)) \widehat{K}_{\tau}(y-x) \dy 
	+\int_{0}^{+\infty}(g(-y) - g(x)) \widehat{K}_{\tau}(y+x) \dy\notag\\
	&= \int_{0}^{+\infty}(g(y) - g(x)) \widehat{K}_{\tau}(y-x) \dy 
	+\int_{0}^{+\infty}(g(-y) - g(x)) \widehat{K}_{\tau}(y-x) \dy\notag\\
	&\ \ +\int_{0}^{+\infty}(g(-y) - g(x)) \big(\widehat{K}_{\tau}(y+x) -  \widehat{K}_{\tau}(y-x)\big) \dy \notag\\
	&= \int_{0}^{+\infty}\big( 1- 2g(x) \big) \widehat{K}_{\tau}(y-x) \dy \notag \\
	&\ \ +\int_{0}^{+\infty}(g(y)-1+g(x)) \big(\widehat{K}_{\tau}(y-x) -  \widehat{K}_{\tau}(y+x)\big) \dy.\notag
	\end{align}
	
	Then using the inequality $g(y)-1\leq0$ when $y\in[x,+\infty)$ and $g(y)\leq g(x)$ when $y\in[0,x]$, $x\leq\bar x$ one has that
	\begin{align}
	\int_{0}^{+\infty}\big( 1- 2g(x) \big) \widehat{K}_{\tau}(y-x)\dy
	&+\int_{0}^{+\infty}(g(y)-1+g(x)) \big(\widehat{K}_{\tau}(y-x) -  \widehat{K}_{\tau}(y+x)\big) \dy\leq\notag\\
	&\leq(1-2g(x))\Bigl(\int_0^{+\infty}\widehat K_\tau(z)\dz+\int_0^x\widehat K_\tau(z)\dz\Bigr)\notag\\
	&+(g(x)-1)\Bigl(\int_0^x\widehat K_\tau(z)\dz-\int_x^{2x}\widehat K_\tau(z)\dz\Bigr)\notag\\
	&+2g(x)\int_0^x\widehat K_\tau(z)\dz\notag\\
	&\leq(1-g(x))\int_0^{2x}\widehat K_\tau(z)\dz+(1-2g(x))\int_x^{+\infty}\widehat K_\tau(z)\dz.
	\end{align}
	
\end{proof}

Thanks to the upper bound~\eqref{eq:stimai} one has the following
\begin{corollary}
	\label{cor:ineg}
	Let $g\in\mathcal S_h\cap W^{1,2}$. Then, the following holds:
	\begin{enumerate}
	\item Let $0<\delta<\frac12$. Then there exists $c_0=c_0(\delta)$ s.t. if $g(x)\geq\frac12+\delta$ and $x\leq c_0\tau^{1/\beta}$, then $I_g(x)<0$.	
	\item Assume $g(x)\geq 1-c\tau$, where $c\tau\ll1$ is such that 
	\begin{equation}\label{eq:tauprecise}
		(1-2c\tau)^{1/(q-1)}\geq\frac12.
	\end{equation}  Then,
	\begin{equation}\label{eq:xctau}
	|x|\leq\frac14\tau^{1/\beta}(c\tau)^{-1/(\beta+1)}\quad\Rightarrow\quad I_g(x)<0.
	\end{equation}
	\end{enumerate}
	\end{corollary}

\begin{proof}
	By~\eqref{eq:stimai} and assuming that $g(x)\geq\frac12+\delta$ one has that
	\begin{align}
	I_g(x)&\leq (1-g(x))\int_0^{2x}\widehat K_\tau(z)\dz+(1-2g(x))\int_x^{+\infty}\widehat K_\tau(z)\dz\notag\\
	&\leq \Bigl[\Bigl(\frac{1}{2}-\delta\Bigr)\frac{2x}{\tau^{q/\beta}}-\frac{2\delta}{q-1}\frac{1}{(x+\tau^{1/\beta})^{q-1}}\Bigr].\label{eq:6.47}
	\end{align}
	Hence choosing $x\leq c_0\tau^{1/\beta}$ for some $c_0=c_0(\delta)\ll1$, one has the first claim.
	
	Using~\eqref{eq:stimai} and assuming that $g(x)\geq1-c\tau$, one has that
	\begin{align}
	I_g(x)&\leq\frac{1}{q-1}\Bigl[c\tau\Bigl(\frac{1}{\tau^{(q-1)/\beta}}-\frac{1}{(2x+\tau^{1/\beta})^{q-1}}\Bigr)+(2c\tau-1)\frac{1}{(x+\tau^{1/\beta})^{q-1}}\Bigr]\notag\\
	&\leq\frac{1}{q-1}\Bigl[c\tau\frac{1}{\tau^{(q-1)/\beta}}+(2c\tau-1)\frac{1}{(x+\tau^{1/\beta})^{q-1}}\Bigr].\label{eq:6.48}
	\end{align}
	Hence, setting $x=t\tau^{1/\beta}$, by straightforward computations on sees that
	\[
	(1+t)^{q-1}<\frac{1-2c\tau}{c\tau}\quad\Rightarrow\quad I_g(x)<0.
	\]
	 Since $c\tau\ll1$, one can assume that $t>1$ and then one gets the estimate
	\[
	t<\frac{1}{2}\Bigl(\frac{1-2c\tau}{c\tau}\Bigr)^{1/(q-1)},
	\]
	which leads to~\eqref{eq:xctau} by \eqref{eq:tauprecise} and since $q-1=\beta+1$.
\end{proof}

\begin{lemma}
	\label{lemma:gammaconvx}
	Let $x=x(\eps,\tau,g,\gamma)$ be such that the minimal function $g$ of  Proposition~\ref{prop:gammalim} satisfies $g(x)=\frac34$. Then, for all $c>0$ there exist $\eps_0,\tau_0$ such that $\forall\,\eps\leq\eps_0,\tau\leq\tau_0$ it holds
	\begin{equation*}
	|x|\leq c\tau^{1/\beta}.
	\end{equation*}
 \end{lemma}

\begin{proof}
The proof will be based on a $\Gamma$-convergence argument, namely the fact that minimizers of $F_{\tau,\eps}$ must converge, as $\eps,\tau\to0$ to the  characteristic function  $\varphi[\chi_{[0,h]}]$.
	Assume by contradiction that there exists $c>0$ and sequences $(\eps_n,\tau_n)\to(0,0)$ with minimizers $(\gamma_n,g_n)$ of $F_{\tau_n,\eps_n}$  such that
\begin{equation}\label{eq:gnxn}
g_n(x_n)=\frac34\quad\Rightarrow\quad|x_n|\geq c\tau_n^{1/\beta}.
\end{equation} 

Let us perform the  inverse rescaling of $F_{\tau_n,\eps_n}$ w.r.t. the one performed in the introduction.
More precisely, set $\tilde x=x\tau_n^{-1/\beta}$, $\tilde g_n(\tilde x)=g_n(x)$ and $\tilde\gamma_n(\tilde x)=\gamma_n(x)$. One obtains that
\[
F_{\tau_n,\eps_n}(\gamma_n,g_n)=\frac{1}{\tau_n\tau_n^{1/\beta}}\tilde F_{\tau_n,\eps_n}(\tilde{\gamma}_n,\tilde g_n),
\]  
where
\begin{align*}
\tilde F_{\tau_n,\eps_n}(\tilde{\gamma}_n,\tilde g_n)=\frac{\tau_n^{1/\beta}}{2h}\Bigl[3(J_c&-\tau_n)\int_0^{2h/\tau_n^{1/\beta}}\Bigl(\eps_n\gamma_n|\tilde g_n'|^2+\frac{W(\tilde g_n)}{\eps_n\gamma_n}\Bigr)\dx\\ &-\int_0^{2h/\tau_n^{1/\beta}}\int_{\R}|\tilde g_n(x)-\tilde g_n(y)|^2\widehat K_1(x-y)\dx\dy\Bigr]
\end{align*}
and 
\[
J_c=\int_{ \R}|z|\widehat K_1(z)\dz,\quad \widehat K_1(z)=\frac{1}{(|z|+1)^q}.
\]

Let us now introduce the following functional

\begin{align*}
\bar M_{\tau}(\tilde g)&=\frac{\tau^{1/\beta}}{2h}\Bigl[3(J_c-\tau)\int_0^{2h/\tau^{1/\beta}}|\tilde g'|\sqrt{W(\tilde g)}\dx-\int_0^{2h/\tau^{1/\beta}}\int_{\R}|\tilde g(x)-\tilde g(y)|^2\widehat K_1(x-y)\dx\dy\Bigr]\notag\\
&=\frac{\tau^{1/\beta}}{2h}\Bigl[3(J_c-\tau)\int_0^{2h/\tau^{1/\beta}}|(\omega\circ \tilde g)'(x)|\dx-\int_0^{2h/\tau^{1/\beta}}\int_{\R}|\tilde g(x)-\tilde g(y)|^2\widehat K_1(x-y)\dx\dy\Bigr],
\end{align*}
where $\omega(t)=3t^2-2t^3$.

Notice as usual that 
\begin{equation}\label{eq:tfbf}
\tilde F_{\tau,\eps}(\tilde \gamma, \tilde g)\geq\bar M_\tau(\tilde g)
\end{equation}
for all $\tilde g\in \mathcal S_{h/\tau^{1/\beta}}$.

We claim that, for any $\tilde g\in \mathcal S_{h/\tau^{1/\beta}}$, 
\begin{equation}\label{eq:fbar}
\bar M_\tau(\tilde g)\geq \bar M_\tau(\chi_{[0,h/\tau^{1/\beta}]})
\end{equation}	
and equality holds if and only if $\tilde g=\chi_{[0,h/\tau^{1/\beta}]}$. 

Indeed, first of all notice that for all $ \tilde g\in\mathcal S_{{h/\tau^{1/\beta}}}$, by the fact that there exists $\tilde x$ such that  $\tilde g$ is monotone nondecreasing on $[0,\tilde x]$ and monotone nonincreasing on $[\tilde x,h/\tau^{1/\beta}]$ and $\tilde g$ reaches the value $1$ in $\tilde x$ (thus $0$ in $2h/\tau^{1/\beta}-\tilde x$) it holds   
	\begin{equation}\label{eq:6.60}
	\int_0^{2h/\tau^{1/\beta}}|(\omega\circ \tilde g)'(x)|\dx=2(\omega(1)-\omega(0))=2.
	\end{equation}
	
	Given the $2{h/\tau^{1/\beta}}$-periodicity of the function $\tilde g$, one has that
    \begin{equation*}
      \begin{split}
      \int_{0}^{2h/\tau^{1/\beta}}\int_{0}^{+\infty}\big(\tilde g(x+z) - \tilde g(x)\big)^{2} \widehat K_1(z)\dz\dx\\
      =\sum_{k=0}^{+\infty} \int_{0}^{2h/\tau^{1/\beta}} \int_{0}^{2h/\tau^{1/\beta}}\big( \tilde g(x+z)-\tilde g(x)\big)^2\widehat K_1(z+2kh/\tau^{1/\beta})\dz\dx.
      \end{split}
    \end{equation*}

	Now observe that the functional 
	\[
	\tilde g\mapsto\int_0^{2h/\tau^{1/\beta}}\int_0^{2h/\tau^{1/\beta}}(\tilde g(x)-\tilde g(x+z))^2\widehat K_1(z+2kh/\tau^{1/\beta})\dz\dx
	\]
	is strictly convex, and the set of functions $ \mathcal S_{h/\tau^{1/\beta}}$ is convex as well.
    Hence, by the above and \eqref{eq:6.60} the functional attains its maximum on the extremal points of $\mathcal S_{h/\tau^{1/\beta}}$, namely on the characteristic function $\chi_{[0,{h/\tau^{1/\beta}}]}$. Thus our claim \eqref{eq:fbar} is proved.

By~\eqref{eq:gnxn}, one has that
\begin{equation}
\|\tilde g_n-\chi_{[0,h/\tau_n^{1/\beta}]}\|_{L^1([0,2h/\tau_n^{1/\beta}])}>\frac14c.
\end{equation} 
Hence, by~\eqref{eq:fbar} and the uniqueness of the minimizer $\chi_{[0,h/\tau^{1/\beta}]}$, there exists a constant $\tilde C_2>0$ such that
\begin{equation}\label{eq:c2}
\bar M_{\tau_n}(\tilde g_n)\geq\bar M_{\tau_n}(\chi_{[0,h/\tau_n^{1/\beta}]})+\tilde C_2. 
\end{equation}	

On the other hand, denoting by  $g_{1,n}$ be the minimizers of $F_{\tau_n,\eps_n}(1,\cdot)$, one has that by the $\Gamma$-convergence of the Modica-Mortola term and the continuity in $L^1$ of the nonlocal term as in Corollary~\ref{cor:gammaconv},
\begin{equation}\label{eq:lim0}
\lim_{n\to+\infty}\tilde F_{\tau_n,\eps_n}(1,\tilde g_{1,n})=0.
\end{equation} 

Putting together~\eqref{eq:lim0}, the fact that $(\tilde \gamma_n,\tilde g_n)$ is a minimizer of $\tilde F_{\tau_n,\eps_n}$,~\eqref{eq:tfbf} and~\eqref{eq:c2} one obtains that

\begin{equation*}
  \begin{split}
0\geq\lim_n\tilde F_{\tau_n,\eps_n}(1,\tilde g_{1,n})& \geq\liminf_n\tilde F_{\tau_n,\eps_n}(\tilde \gamma_n,\tilde g_n)\geq \liminf_n\bar M_{\tau_n}(\tilde g_n) \\ &\geq \liminf_n\bar M_{\tau_n}(\chi_{[0,h/\tau_n^{1/\beta}]})+ \tilde C_2\geq\tilde C_2>0,
  \end{split}
\end{equation*}
 thus reaching a contradiction.

	\end{proof}

\begin{lemma}
	\label{lemma:1-ctau}
	Let $0<c\ll1$ and  $g$ as in Proposition~\ref{prop:gammalim}. Then, there exists $c_1\geq1$, $\tau_1>0$ such that  for all $\tau\leq\tau_1$
	\begin{equation}
	g(x)=1-c_1\tau\quad\Rightarrow \quad |x|\leq c\tau^{1/\beta}
	\end{equation}
\end{lemma}

\begin{proof}
  Let $x=c\tau^{1/\beta}$ be such that $g(x)=1-c_1\tau>\frac34$, with $c,c_1$ to be fixed later.
  Then, by \eqref{eq:grefl} and the monotonicity of $g$ on the intervals $[0,\bar x_1]$, $[\bar x_2,h]$, whenever  $|s-t| < c\tau^{1/\beta}$, $s,t\in\R$, it holds $|g(s)- g(t)| \leq 1-c_1\tau$. Hence, as in Lemma~\ref{lemma:estimate2} for the case in which $\gamma\equiv1$,
	\begin{align}\label{eq:estpos}
	F_{\tau,\eps}(\gamma,g)\geq\Bigl[\frac{3}{h}\int_0^h\Big(\eps\tau^{1/\beta}\gamma|g'|^2+\frac{W(g)}{\gamma\eps\tau^{1/\beta}}\Big)\dx\Bigr]\cdot\Bigl[-1+\frac{2c_1\tau}{1+2c_1\tau}\int_{-c\tau^{1/\beta}}^{c\tau^{1/\beta}}|z|\widehat K_{\tau}(z)\dz\Bigr].
	\end{align}
	When $c\ll1$, then 
	\[
	\int_{-c\tau^{1/\beta}}^{c\tau^{1/\beta}}|z|\widehat K_{\tau}(z)\dz\sim\frac{2c^2}{\tau(c+1)^q},
	\]
	thus one has that
	\[
	\frac{2c_1\tau}{1+2c_1\tau}\int_{-c\tau^{1/\beta}}^{c\tau^{1/\beta}}|z|\widehat K_{\tau}(z)\dz\geq 	\frac{4c_1c^2}{(1+2c_1\tau)(c+1)^q}\geq1,
	\]
	provided $c_1\geq(c+1)^q/c^2$ and $\tau$ is sufficiently small so that $1+2c_1\tau\leq2$. In this case by~\eqref{eq:estpos} one has that $F_{\tau,\eps}(\gamma,g)>0$, which contradicts the minimality of $(\gamma,g)$.
\end{proof}

\begin{proof}
	[Proof of Theorem~\ref{thm:pos}: ] Fix $\delta\ll1$ (to be chosen later) and let $c_0=c_0(\delta)$ as in Corollary~\ref{cor:ineg}, namely such that whenever $g(x)\geq\frac12+\delta$ and $x\leq c_0\tau^{1/\beta}$, then $I_g(x)<0$. Choose $\eps_0,\tau_0$ as in Lemma~\ref{lemma:gammaconvx} for $c=c_0$, namely such that whenever $g(x)\leq\frac{3}{4}$ then $|x|\leq c_0\tau^{1/\beta}$. Thus, by such choices 
	\begin{equation}
	g(x)\in\Bigl[\frac{1}{2}+\delta,\frac34\Bigr]\quad\Rightarrow \quad |x|\leq c_0\tau^{1/\beta},\quad I_g(x)<0.
	\end{equation}
	Choosing $c_1=c_1(c_0)$, $\tau_1$ as in Lemma~\ref{lemma:1-ctau} and $\tau_2\leq\min\{\tau_0,\tau_1\}$, we also have that, for any $\tau\leq\tau_2$
	\begin{equation}\label{eq:gineqx2}
	g(x)\in\Bigl[\frac34,1-c_1\tau\Bigr]\quad\Rightarrow\quad|x|\leq c_0\tau^{1/\beta}, \quad I_g(x)<0.
	\end{equation}
	By Lemma~\ref{lemma:I}, we have that
	\begin{equation}
	g(x)\geq 1-\bar c_1\tau\tau^{1/\beta}\quad\Rightarrow\quad I_g(x)<\bar c_0<0.
	\end{equation}
	Now we want to show that 
	\begin{equation}
	g(x)\in\Bigl[1-c_1\tau, 1-\bar c_1\tau\tau^{1/\beta}\Bigr]\quad\Rightarrow\quad I_g(x)<0.
	\end{equation}
	In order to do so, we first apply the second statement of Corollary~\ref{cor:ineg} with $c=c_1$ and we deduce that $I_g(x)<0$ for all $x$ such that  $0\leq x\leq\frac14\tau^{1/\beta}(c_1\tau)^{-1/(\beta+1)}$. Let $x_\tau$ be such that $g(x_\tau)=1-c_1\tau$.
    By~\eqref{eq:gineqx2} and the fact that $c_0\ll1, c_1\tau\ll1$ it holds $|x_\tau|\leq c_0\tau^{1/\beta}\ll\frac14\tau^{1/\beta}(c_1\tau)^{-1/(\beta+1)}$. On $[x_\tau, \frac14\tau^{1/\beta}(c_1\tau)^{-1/(\beta+1)}]$, $I_g<0$
    hence $\gamma=1$ (otherwise by~\eqref{eq:g2} and~\eqref{eq:gamma2} if $\gamma>1$ then $I_g=0$).
	Hence,~\eqref{eq:gamma1} holds and therefore by comparison with the optimal profile function for the Modica-Mortola term $\bar g'=\frac{\sqrt{W(\bar g)}}{\eps\tau^{1/\beta}}$ the function $g$ reaches the value $1-\bar c_1\tau\tau^{1/\beta}$ from $x_\tau$ in an interval of the order $-\eps\tau^{1/\beta}\log(\tau^{1/\beta})$, which is much smaller (for $\eps$ and $\tau$ sufficiently small) than $\frac14\tau^{1/\beta}(c_1\tau)^{-1/(\beta+1)}-c_0\tau^{1/\beta}$, namely of the distance from $x_\tau$ for which we know that $I_g<0$ and $\gamma=1$.
	
	Thus, the only interval on which one could have that $I_g>0$ and in principle $\gamma>1$ (i.e.,~\eqref{eq:gamma2} holds) is $[0,x_\delta]$ where $x_\delta$ is such that $g(x_\delta)=\frac12+\delta$.
	
		Let us now assume that there exists $\hat x\in[0,x_\delta]$ such that 
	$\varepsilon\tau^{1\beta} \gamma^2(\hat x)|g'(\hat x) |^{2} - \frac{W(g(\hat x))}{\varepsilon\tau^{1/\beta}} = 0 $.  Then by~\eqref{eq:g2}  one has that
	\begin{equation}\label{eq:posneg}
	\begin{split}
	\int_{\hat x}^{x_\delta}
	\Bigl( \gamma(z) g'(z) I_g(z)\Bigr)_+\dz \geq 	-\int_{x_\delta}^{\bar x_1}
	g'(z) I_g(z)\dz. 
	\end{split}
	\end{equation}
	In particular, since $\gamma=1$ whenever $I_g\neq0$, we can assume that $\gamma(z)=1$ in the l.h.s. of~\eqref{eq:posneg}.
	
	On the one hand observe that, by~\eqref{eq:6.47} and~\eqref{eq:6.48}, there exists $\bar c_2$ such that 
	\[
	g(x)\in\Bigl[\frac34,(1-c_1\tau)\Bigr]\quad\Rightarrow\quad I_g(x)<-\frac{\bar c_2}{\tau\tau^{1/\beta}}. 
	\]
	Then, denoting by $x_{3/4}$ the point such that $g(x_{3/4})=\frac34$ and by $x_\tau$ the point such that $g(x_\tau)=1-c_1\tau$ one has that 	\begin{align}
	-\int_{x_\delta}^{\bar x_1}g'(x)I_g(x)\dx&\geq\frac{\bar c_2}{\tau\tau^{1/\beta}}\int_{x_{3/4}}^{x_\tau}g'(x)\dx\notag\\
	&\geq\frac{\bar c_2}{\tau\tau^{1/\beta}} \Bigl(\frac14-c_1\tau\Bigr). \label{eq:-}
	\end{align}
	
	On the other hand, using~\eqref{eq:0l} and the fact that $g\leq1$
	\begin{align}
	\int_0^{x_\delta}g'(x)I_g(x)\dx&\leq\int_0^{x_\delta}g'(x)(1-g(x))\frac{\tilde C_1}{\tau\tau^{1/\beta}}\notag\\
	&\leq \frac{\tilde C_1}{2\tau\tau^{1/\beta}}\Bigl((1-g(0))^2-(1-g(x_\delta))^2\Bigr)\notag\\
	&\leq \frac{\tilde C_1\delta}{2\tau\tau^{1/\beta}} +o(\delta).\label{eq:+}
	\end{align}
	Thus, given~\eqref{eq:-} and~\eqref{eq:+},~\eqref{eq:posneg} cannot hold provided $\delta$ is chosen sufficiently small at the beginning of the proof.
    As a consequence, $\eps\tau^{1/\beta}\gamma^2|g'|^2-\frac{W(g)}{\eps\tau^{1/\beta}}>0$ on $[0,\bar x_1)$ (and by symmetry on $(\bar x_2,h]$). In particular, by~\eqref{eq:gamma1}-\eqref{eq:gamma2}, $\gamma=1$ a.e..

\end{proof}

In the following proposition we give an estimate on the size of the interval $[\bar x_1,\bar x_2]$ on which $g\equiv 1$. In particular, we show that such an interval is non-degenerate (i.e. $\bar x_1<\bar x_2$) and thus the Euler-Lagrange equation \eqref{eq:g2} is a free-boundary problem where the function $g$ hits the obstacle $\{g=1\}$.

\begin{proposition}
	\label{prop:x1x2}
	Let $g$ be a minimizer of $F_{\tau,\eps}(1,\cdot)$ and $\bar x_1\leq\bar x_2$ as in Corollary \ref{cor:mon}. The, provided $\eps,\tau$ are sufficiently small, it holds
	\begin{equation}
		\bar x_1<\tau^{1/\beta}<h-\tau^{1/\beta}<\bar x_2.
	\end{equation} 
\end{proposition}

\begin{proof}
	We will show that $\bar x_{1}<\tau^{1/\beta}$ and by symmetry one can deduce that $h-\tau^{1/\beta}<\bar x_{2}$.

	As $g$ is a minimizer of $F_{\tau,\eps}(1,\cdot)$, and thus the Modica-Mortola term approximates the perimeter of the set $[0,h]$ when $\varepsilon,\tau$ are sufficiently small, the following holds:  there exists $\varepsilon_{0},\tau_{0}>0$ such that whenever $\varepsilon < \varepsilon_{0}$ and $\tau<\tau_{0}$
	\begin{equation*}
		\begin{split}
			\Big|\frac{3}{h}\int_{0}^{h} \eps\tau^{1/\beta}\, |g'|^{2}\dx + \frac{3}{h}\int_{0}^{h}\frac{W(g)}{\eps\tau^{1/\beta}}\dx -1\Big| <\delta_{0}.
		\end{split}
	\end{equation*}
	
	Moreover, let $\bar x_{\delta}$ be such that such that $g(\bar x_{\delta})= 1-\delta$. For $\varepsilon$ sufficiently small one has that $\bar x_{\delta}\leq \sqrt{\varepsilon}\tau^{1/\beta}$, otherwise the part of the Modica-Mortola term containing the double-well potential $W$ would explode as $\eps\to0$. 
By the first statement of Corollary \ref{cor:ineg}, and in particular \eqref{eq:6.47}, one has then that there exists a constant $c_3>0$ such that $I_g<-\frac{c_3}{\tau\tau^{1/\beta}}$ for every $x\in (\bar x_\delta,\tau^{1/\beta})$.
	Assume now that in the interval $(\bar x_{\delta},\tau^{1/\beta})$ it holds $g < 1$, thus $\bar x_1>\tau^{1/\beta}$. This implies that on this interval $g$ satisfies the Euler-Lagrange equation \eqref{eq:g1}, i.e. 
	\begin{equation*}
		\begin{split}
			3(C_{\tau}-1)\eps\tau^{1/\beta} g''(x) = 3(C_{\tau}-1)\frac{W'(g(x))}{\eps\tau^{1/\beta}} +2 I_g(x),
		\end{split}
	\end{equation*}
	where $C_{\tau}= \frac{\tilde C_2}{\tau}$. 
	Thus, given that $W'<0$, for every $x\in (\bar x_\delta, \tau^{1/\beta})$ one has that
	\begin{equation}
		\label{eq:gstrss2}
		\begin{split}
			g''(x) \leq -\frac{c_4}{\varepsilon \tau^{2/\beta}}
		\end{split}
	\end{equation}
	
	for some constant $c_4>0$.
	Let $\tilde{x}= \argmin\insieme{g'(x):\ x\in [\bar x_{\delta},\tau^{1/\beta}]}$.
	Then from \eqref{eq:gstrss2} we have that
	\begin{equation*}
		\begin{split}
			g(x) \leq g(\tilde{x}) + g'(\tilde{x}) (x-\tilde{x}) - \frac{c_4}{2 \varepsilon\tau^{2/\beta}}(x-\tilde{x})^{2}
		\end{split}
	\end{equation*}
	
	Since $g''<0$ one has that $\tilde{x} = \tau^{1/\beta}$. Thus
	\begin{equation}
		\label{eq:gstrss3}
		\begin{split}
			g(\bar x_{\delta}) \leq 1 - \frac{c_4}{2\varepsilon\tau^{2/\beta}} (\tau^{1/\beta} - \bar x_{\delta})^2 \leq1- \frac{c_4}{2\varepsilon\tau^{2/\beta}} (\tau^{1/\beta} - \sqrt{\varepsilon}\tau^{1/\beta})^{2} \leq 1- \frac{c_5}{\varepsilon}.
		\end{split}
	\end{equation}
	
	Thus for $\varepsilon$ sufficiently small we have a contradiction to the fact $g(\bar x_{\delta}) = 1-\delta$.
	
	 Hence $\bar x_1<\tau^{1/\beta}$ and the proof is concluded.
	
\end{proof}

\section{Proof of Theorem~\ref{Thm:1}}
\label{sec:proof1}

By the $\Gamma$-convergence result of Corollary~\ref{cor:gammaconv} and Theorem~\ref{Thm:DR}, there exist $\eps'_L>0$ and $\tau'_L>0$ such that, for all $0<\eps\leq\eps'_L$, $0<\tau\leq\tau'_L$, then minimizers $u$ of $\Fcalt$ satisfy
\begin{equation}\label{eq:estL1}
   \|u-\chi_S\|_{L^1([0,L)^d)}\leq\bar\sigma,
\end{equation}
with $\bar\sigma$ as in Proposition~\ref{prop:stability} and $S$ periodic union of stripes with boundaries orthogonal to $e_i$ for some $i\in\{1,\dots,d\}$. 
Without loss of generality, let us assume that $i=1$.

Recall now the lower bound for the functional~\eqref{E:F} given in Proposition~\ref{prop:lowbound} by

\begin{align}
   \Fcalt(u)&\geq\frac{1}{L^{d-1}} \int_{[0,L)^{d-1}}\frac{1}{L}\Bigl[-\Mi{0}{L}{1}+\Gcal{1}\Bigr]\dx_1^\perp+\frac{1}{L^d}\mathcal W_{\tau,L,\eps}(u)\label{eq:1d}\\
   &+\sum_{i=2}^d\frac{1}{L^{d-1}}\Bigl\{\int_{[0,L)^{d-1}}\frac1L\Bigl[-\Mi{0}{L}{i}+\Gcal{i}\Bigr]\dx_i^{\perp}+\frac1L\mathcal I^i_{\tau,L}(u)\Bigr\}.\label{eq:multid}
\end{align}

Now notice that, using the definitions of $\overline{\mathcal M}^1_{\at}$, $\overline{\mathcal G}^1_{\at,\tau}$ and $\mathcal W_{\tau,L,\eps}$, the r.h.s. of ~\eqref{eq:1d} can be rewritten as

\begin{align}\label{eq:7.2}
\frac{1}{L^{d-1}}&\int_{[0,L)^{d-1}}\Bigl[\frac{3(C_\tau-1)}{L}\int_0^L\Bigl(\eps\tau^{1/\beta}|u_{x_1^\perp}'(s)|\|\nabla u(x_1^\perp+se_1)\|_1+\frac{W(u_{x_1^\perp}(s))|u_{x_1^\perp}'(s)|}{\eps\tau^{1/\beta}\|\nabla u(x_1^\perp+se_1)\|_1}\Bigr)\ds\notag\\
&-\frac1L\int_0^L\int_{ \R}|u_{x_1^\perp}(s)-u_{x_1^\perp}(s+z)|^2\widehat K_\tau(z)\dz\ds\Bigr]\dx_1^\perp
\end{align}
with the convention that $|u_{x_1^\perp}'(s)|/\|\nabla u(x_1^\perp+se_1)\|_1=1$ whenever $\|\nabla u(x_1^\perp+se_1)\|_1=0$.

Setting $g(s)=u_{x_1^\perp}(s)$ and 
\[
\gamma(s)=\frac{\|\nabla u(x_1^\perp+se_1)\|_1}{|u_{x_1^\perp}'(s)|},\quad s\in[0,L],
\]
the functional inside the integral in~\eqref{eq:7.2} takes the form
\begin{equation}\label{eq:functgamma2}
\Fcalt^{1d}(\gamma,g)=	\frac{3(C_\tau-1)}{L}\int_0^L\Bigl[\eps\tau^{1/\beta}|g'(s)|^2\gamma(s)+\frac{W(g(s))}{\eps\tau^{1/\beta}\gamma(s)}\Bigr]\ds-\frac1L\int_0^L\int_{ \R}|g(s)-g(s+z)|^2\widehat K_\tau(z)\dz\ds
\end{equation} 
as in~\eqref{eq:functgamma0}. By the results of Section \ref{ssec:1d}, such a functional is minimized by periodic functions $(\gamma,g)$ of period $2h^*_{\tau,\eps}$, with $\gamma\in\Gamma_{h^*_{\tau,\eps}}$ and $g=\varphi[g]$ for some function  $g\in \mathcal C_{h^*_{\tau,\eps}}$.

Then, Theorem~\ref{thm:onedimmin} shows that for $\eps\leq\eps_2$ and $\tau\leq\tau_2$ the minimal values of such a functional among all $\gamma\geq1$ and $g\in\mathcal C_{h^*_\tau,\eps}$ is attained for $\gamma=1$ a.e.. This makes the minimal values of the functional in~\eqref{eq:functgamma2} to be equal to the minimal values of the functional  $\Fcal^{1d}_{\tau,L,\eps}(1,\cdot)$.

By Proposition~\ref{prop:stability} we know that, if additionally $0<\tau\leq\tau'$,  each of the $d-1$ terms of the sum in~\eqref{eq:multid} is zero if $u(x)=g(x_1)$ and strictly positive otherwise.
Therefore, if $\eps\leq\eps_L=\min\{\eps_2,\eps_L'\}$ and $\tau\leq\tau_L=\min\{\tau_2,\tau_L',\tau'\}$ then $u(x)=g(x_1)$ minimizes both~\eqref{eq:1d} and~\eqref{eq:multid} and thus the whole functional $\mathcal F_{\tau,L,\eps}$.




\end{document}